\documentclass[amssymb,amsfonts,refcheck,11pt,verbatim,righttag]{amsart}

\usepackage{amssymb,mathrsfs,hyperref}
\usepackage{soul}
\usepackage{graphicx,color,tikz,caption,subcaption}
\RequirePackage{yhmath} 
  \renewcommand\widering[1]{\ring{#1}}

\numberwithin{equation}{section}
\theoremstyle{plain}

\newtheorem{thm}{Theorem}[section]
\newtheorem{lem}[thm]{Lemma}
\newtheorem{pro}[thm]{Proposition}
\newtheorem{cor}[thm]{Corollary}

\newtheorem{de}[thm]{Definition}
\newtheorem{rem}[thm]{Remark}

\def\R {{\Bbb R}}
\def\N {{\Bbb N}}

\def\M {{\mathcal M}}

\def\A {{\mathcal A}}

\def\E{{\mathbb E}}

\begin{document}

\title{Dimensions of random statistically self-affine Sierpinski sponges in $\mathbb R^k$}

\author{Julien Barral}
\address{Laboratoire de G\'eom\'etrie, Analyse et Applications, CNRS, UMR 7539, Universit\'e Sorbonne Paris Nord, CNRS, UMR 7539,  F-93430, Villetaneuse, France}
\email{barral@math.univ-paris13.fr}
\author{De-Jun Feng}
\address{
Department of Mathematics\\
The Chinese University of Hong Kong\\
Shatin,  Hong Kong\\
} \email{djfeng@math.cuhk.edu.hk}


\date{}

\begin{abstract} We compute the Hausdorff dimension of any random statistically self-affine Sierpinski sponge $K\subset \R^k$ ($k\ge 2$) obtained by using some percolation process in  $[0,1]^k$. To do so, we first exhibit a Ledrappier-Young type formula for the  Hausdorff dimensions of statistically  self-affine measures supported on $K$.  This formula presents a new feature compared to its deterministic or random dynamical version. Then, we establish a variational principle expressing  $\dim_H K$ as the supremum of the Hausdorff dimensions of  statistically self-affine measures supported on $K$, and show that  the supremum is uniquely attained. The value of $\dim_H K$ is also expressed in terms of the weighted pressure function of some deterministic potential.  As a by-product, when $k=2$, we give an alternative approach to  the Hausdorff dimension of $K$,  which was first obtained by  Gatzouras and Lalley \cite{GL94}. The value of the box counting dimension of $K$  and its equality with $\dim_H K$ are also studied. We also obtain a variational formula for the Hausdorff dimensions of some orthogonal projections of $K$, and for statistically self-affine measures supported on~$K$, we establish a  dimension conservation property through these projections.
 \end{abstract}

\maketitle

\section{Introduction}

This paper deals with dimensional properties of a natural class of random statistically self-affine sets and measures in $\R^k$ ($k\ge 2$), namely random Sierpinski sponges and related Mandelbrot measures, as well as certain of their projections and related fibers and conditional measures. These random sponges can be also viewed as limit sets of some percolation process on the unit cube endowed with an $(m_1,\ldots,m_k)$-adic grid, where $m_1\ge \cdots \ge m_k\ge 2$ are integers.

Until now the Hausdorff dimension of such a set $K$ is known in the deterministic case and only when $k=2$ in the random case, while the projections of $K$ and the random Mandelbrot measures to be considered have been studied in the conformal case only. Understanding the missing cases is our goal, with Mandelbrot measures and their projections as a main tool for a variational approach to the Hausdorff dimension of $K$ and its orthogonal projections, together with a ``weighted'' version of the thermodynamic formalism on symbolic spaces.

Before coming in more details to these objects and our motivations, it seems worth giving an overview of the nature of the main results known in  dimension theory of self-affine sets. Recall that given an integer $N\ge 2$ and an iterated function system (IFS) $\{f_i\}_{1\le i\le N}$ of contractive maps of a complete metric space $X$, there exists a unique non-empty compact set $K\subset X$ such that
$$
K=\bigcup_{i=1}^N f_i(K),
$$
(see \cite{hutch}). When $X$ is a Euclidean space and $f_i$ are affine maps, due to the above equality $K$ is called self-affine. In particular, $K$ is called self-similar if  $f_i$ are all similitudes; we will not  focus on the self-similar case and refer the reader to \cite{HochICM} for a recent survey of this topic.

The first class of strictly self-affine sets that have been studied in detail are certainly Bedford-McMullen carpets in $\R^2$, also known as self-affine Sierpinski carpets. They are among the most natural classes of fractal sets having different Hausdorff and box counting dimensions. To be specific, one fixes two integers $m_1>m_2\ge 2$ and a  subset $A\subset \{0,\ldots,m_1-1\}\times \{0,\ldots, m_2-1\}$ of cardinality at least 2; the Bedford-McMullen carpet~$K$ associated with $A$  is the attractor of the system $$
S_A=\Big \{f_{a}(x_1,x_2)=\Big (\frac{x_1+a_1}{m_1},\frac{x_2+a_2}{m_2}\Big ): \; a=(a_1,a_2)\in A\Big \}$$
of contractive affine maps of the Euclidean plane; note that  by construction $K\subset [0,1]^2$.
Set $N_{i}=\#\{(a_1, a_2)\in A: \, a_2= i\}$ for all $0\le i\le m_2-1$ and $\psi:\theta\in \R_+\mapsto \log  \sum_{i=0}^{m_2-1} N_i^\theta$. Bedford and McMullen  proved independently \cite{Bed84,McMullen} that
\begin{align*}
\dim_H K= \frac{\psi(\alpha)}{\log(m_2)}, \text{ where }\alpha=\frac{\log(m_2)}{\log(m_1)},
\end{align*}
and
 \begin{align*}
  \dim_B K=\frac{\psi(1)}{\log(m_1)}+\Big (\frac{1}{\log(m_2)}-\frac{1}{\log (m_1)} \Big ) \psi(0),
\end{align*}
where $\dim_H$ and $\dim_B$ respectively stand for the Hausdorff  and the box counting dimension, and $\psi(1)$ and $\psi(0)$ are the topological entropy of $K$ and that of its projection to the $x_2$-axis respectively. Moreover, $\dim_H K=\dim_B K$ if and only if the positive $N_i$ are all equal. Note that the possible dimension gap $\dim_H K<\dim_B K$ cannot hold for self-similar sets (see~\cite{Fal0}). For a general self-affine Sierpinski sponge  $K\subset [0,1]^k$ invariant under the action of an expanding diagonal endomorphism $f$ of $\mathbb T^k$ with eigenvalues the integers $m_1>\cdots>m_k\ge 2$ (we identify $[0,1]^k$ with $\mathbb T^k$), similar formulas as in the 2-dimensional case hold. In particular, the approach to dimensional properties of compact $f$-invariant sets developed by Kenyon and Peres  \cite{KePe96} extends the results of \cite{McMullen} and establishes the following variational principle: the Hausdorff dimension of $K$ is the supremum of the Hausdorff dimensions of ergodic measures supported on $K$, i.e. the Hausdorff  and  dynamical dimensions of $K$ coincide. Moreover, the dimension of any ergodic measure is given by the  Ledrappier-Young formula
\begin{equation}\label{LY}
\dim (\mu)=\frac{1}{\log(m_1)}\cdot h_\mu(f)+\sum_{i=2}^k \Big (\frac{1}{\log(m_i)}-\frac{1}{\log (m_{i-1})} \Big) h_{\mu\circ\Pi_i^{-1}}(\Pi_i\circ f),
\end{equation}
where $\Pi_i:\; \R^k\to \R^{k-i+1}$ is the projection defined by $(x_1,\ldots,x_k)\mapsto (x_i,\ldots, x_k)$. Also, in the variational principle, the supremum is uniquely  attained at some Bernoulli product measure.

Dimension theory of self-affine sets has been developed to understand the case of more ``generic'' self-affine IFSs as well. First, given a family $\{M_i\}_{1\le i\le N}$ of linear automorphisms of  $\R^k$ to itself whose norm subordinated to the Euclidean norm on $\R^k$ is smaller that $1/2$, it was shown \cite{Fal,KaeVil10} that for $\mathcal L^{Nk}$-almost  every choice of $N$ translation vectors $v_1,\ldots, v_N$ in $\R^{k}$, the Hausdorff dimension of the attractor $K$ of the affine IFS $\{M_i x+v_i\}_{1\le i\le N}$ is the maximum of the Hausdorff dimensions of the natural projections of ergodic measures on $(\{1,\ldots,N\}^{\mathbb N^*},\sigma)$ to $K$ (here $\sigma$ stands for the  left shift operation); in addition for such a measure $\mu$, $\dim(\mu)=\min (k,\dim_L(\mu))$, where $\dim_L(\mu)$ is the so-called Lyapunov dimension of $\mu$ \cite{Fal,Sol98,Kae04,JorPolSim}.  A similar result holds for $\mathcal L^{Nk^2}$-almost every choice of contractive $\{M_i\}_{1\le i\le N}$ in some non-empty open set, with a fixed  $(v_1,\ldots, v_N)$ \cite{BaranyKaenKoi}. In these contexts one has $\dim_H K=\dim_B K=\min (k,\dim_AK)$, where $\dim_AK$ stands for the so-called affinity dimension of $K$ (note that for a Bedford-McMullen carpet, the affinity dimension equals $\frac{\psi(1)}{\log(m_2)}$ if $\psi(1)\le \log(m_2)$ and $\frac{\psi(1)}{\log(m_1)}+(\frac{1}{\log(m_2)}-\frac{1}{\log(m_1)})\log(m_2)$ otherwise, so that in this case the   previous equality between dimensions only occurs exceptionally).  If both $\{M_i\}_{1\le i\le N}$  and $(v_1,\ldots, v_N)$ are fixed,  the following stronger result appeared recently in the 2-dimensional case: if $\{M_i\}_{1\le i\le N}$ satisfies the strong irreducibility property and $\{M_i/\sqrt{|\mathrm{det}(M_i)|}\}_{1\le i\le N}$ generate a non-compact group in $GL_2(\mathbb R)$, and if the IFS  $\{f_i\}_{1\le i\le N}$ is exponentially separated, then $\dim_H K$ is the supremum of the Lyapounov dimensions of the self-affine measures supported on $K$ (it is not known whether this supremum is reached in general); here again for such a measure, the Lyapunov and Hausdorff dimensions coincide, and  $\dim_H K=\dim_B K=\dim_AK$ \cite{BaranyHochRap,HochRap}. The last two contexts make a central use of the notion of Furstenberg measure associated to a self-affine measure, whose crucial role in the subject was first pointed out in~\cite{FalKem}. There are also similar results in the case that the strong irreducibility fails but the $M_i$ cannot be simultaneously reduced  to  diagonal automorphisms \cite{FalMiao,BaranyRamsSim2,BaranyHochRap,HochRap}.

Let us come back to self-affine carpets. Their study was further developed with the introduction of Gatzouras-Lalley carpets \cite{LG92} (with  an application to the study of some non-conformal nonlinear repellers \cite{GP97}) and  Baranski carpets \cite{Baranski}. There, the linear parts are no more subject to be equal, but they are still diagonal, and it is not true in general that there is a unique  ergodic measure with maximal Hausdorff dimension \cite{KaeVil10,BF2011} (see also  \cite{KolSim} for a study of Gatzouras-Lalley type carpets when the linear parts are trigonal).  It turns out that extending the dimension theory of these carpets to the higher dimensional case raises serious difficulties in general, as it was shown in \cite{DasSimm} that the attractor may have  a Hausdorff dimension strictly larger than its dynamical dimension.

On the side of random fractal sets, one naturally meets random statistically self-affine sets. Such a set $K$ obeys almost surely an equation of the form $K(\omega)=\bigcup_{i=0}^N f^\omega_i(K_i(\omega))$, where the $f_i^\omega$ are random contractive affine maps and the sets $K_i$ are copies of $K$. Results similar to those obtained for almost all self-affine sets described above exist in the following situation: the sets $K_i$ are mutually independent and independent of the $f_i$, the linear maps of the $f_i$ are deterministic, but the translation parts are i.i.d and follow a law compactly supported and absolutely continuous with respect to $\mathcal L^d$~\cite{JorPolSim}.  Results are also known for random Sierpinski carpets. There are two natural ways to get such random sets. The first one falls in the setting of random dynamical systems. It consists in considering an ergodic dynamical system $(\Omega, \mathcal F, \mathbb P, T)$, on which is defined a random non-empty  subset $A(\omega)$ of $\{0,\ldots,m_1-1\}\times \{0,\ldots, m_2-1\}$ such that $\mathbb{E}(\# A)>1$. Then one  starts with the set of maps $S_{A(\omega)}$, and recursively, at each step $n\ge 2$  of the iterative construction of the random attractor $K(\omega)$, replace the set of contractions $S_{A(T^{n-2}(\omega))}$ by $S_{A(T^{n-1}(\omega))}$, so that the contractive maps used after $n$ iterations take the form $f_{a_0}\circ \cdots\circ f_{a_{n-1}}$, with $a_i\in S_{A(T^i(\omega)})$. By construction, $K(\omega)=\bigcup_{a\in A(\omega)} f_a(K(\sigma(\omega)))$. The Hausdorff  and box-counting dimensions of such sets and their higher dimensional versions have been determined in \cite{Kifer} (in a slightly more general setting); the situation is close to that in the deterministic one.

The other natural way to produce random statistically self-affine carpets is related to branching processes and consists in using a general percolation scheme detailed below; at the moment let us just say that one starts with a possibly empty random subset $A(\omega)$ of $\{0,\ldots,m_1-1\}\times \{0,\ldots, m_2-1\}$, and again assumes that $\mathbb{E}(\# A)>1$. Then, one constructs on the same probability space the set $A(\omega)$ and a random compact set $K(\omega)\subset \mathbb T^2$, and $m_1 \times m_2$ random  compact sets $K(a,\omega)$, $a\in \{0,\ldots,m_1-1\}\times \{0,\ldots, m_2-1\}$, so that $K(\omega)=\bigcup_{a\in A(\omega)} f_a(K(a,\omega))$, where  the $K(a)$ are independent copies of $K$, and they are also independent of $A$. The set $K$ is non-empty with positive probability.  These  random sets have been studied in~\cite{GL94}, and their self-similar versions have been investigated extensively (see e.g. \cite{GMW,Wa07,RamSi3}). Setting now $\psi(\theta)=\log \sum_{i=0}^{m_2-1} \mathbb{E}(N_i)^\theta$ and letting $t$ be  the unique point at which the convex function $\psi$ attains its minimum over $[0,1]$ if $\psi$ is not constant, and $t=1$ otherwise, one has, with probability 1, conditional on $K\neq\emptyset$,
\begin{align}
\label{B-MM}\dim_H K= \frac{\psi(\alpha)}{\log(m_2)}\text{ where $\alpha=\max \Big (t,\frac{\log(m_2)}{\log(m_1)}\Big )$},
\end{align}
and %
\begin{align}
\nonumber\dim_B K=\frac{\psi(1)}{\log(m_1)}+\Big (\frac{1}{\log(m_2)}-\frac{1}{\log(m_1)} \Big )\psi(t).
\end{align}
Moreover, $\dim_H K=\dim_B K$ iff $t=1$ or all the positive $\mathbb{E} (N_i)$ are equal (we note that the value of $\dim_H K$ was previously obtained in \cite{Pey,BenNasr,BenNasr2} in the very special case that there is an integer $b\ge 2$ such that the law of $A$ assigns equal probabilities to subsets of cardinality~$b$ and probability 0 to the other ones). It is worth pointing out that  the origin of this different formula with respect to the deterministic case comes from the possibility  that $\mathbb{E} (N_i)<1$ for some $i$, which makes the situation quite versatile with respect to the deterministic Bedford-McMullen carpets.

The approach developed in \cite{GL94} to get \eqref{B-MM} is not based on a variational principle related to a natural class of measures supported on the attractor. To determine the sharp lower and upper bounds for $\dim_H K$, the authors of \cite{GL94}  adapted the approach used by Bedford in the deterministic case: the lower bound for $\dim_H K$ is obtained by taking the maximum of two values, namely $\dim_H \Pi_2(K)$  (where $\Pi_2$ still denotes the orthogonal projection on the $x_2$-axis) and the maximum of  the lower bounds for the Hausdorff dimensions of certain random subsets of $K$. Each such subset $E$ is obtained by considering the union of almost all the fibers $\Pi_2^{-1}(\{(0,x_2)\}$ with respect to the restriction to $\Pi_2(K)$ of some Bernoulli product measure. The Hausdorff dimensions of $\Pi_2(E)$ and that of the associated fibers are controlled from bellow. This yields a lower bound for $\dim_H E$ thanks to a theorem of Marstrand. The upper bound for $\dim_H K$ is obtained by using some effective coverings of~$K$. It turns out to be delicate to transfer these methods to the higher dimensional cases. Indeed, for the lower bound, the Hausdorff dimension of the 1-dimensional fibers mentioned above is obtained thanks to statistically self-similar branching measures in random environment, the dimension of which is relatively direct to get, and it yields the dimension of the fiber.  Using this approach in the higher dimensional case $k\ge 3$, we would have to consider the restriction of Bernoulli measures to $\Pi_i(K)$ for $2\le i\le k$. Then for $i\ge 3$ we would meet the much harder problem to estimate the Hausdorff dimension of fibers which are statistically self-affine Sierpinski sponges  in a random environment in $\mathbb R^{i-1}$, a problem not less difficult than the one we consider in this paper; one would have to compute $\dim_H \Pi_i(K)$ as well, a question  that we will naturally consider.  Also, for the upper bound, extending to higher dimensions the combinatorial argument used in  \cite{GL94} to get effective coverings seems impossible. However, we will see that it is  rather direct to get the  box-counting dimension of $K$ in the higher dimensional cases by adapting the arguments of  \cite{GL94} used for the two dimensional case.

We will develop the dimension theory for statistically self-affine Sierpinski sponges in $\R^k$, for any $k\ge 2$, by studying the statistiscally self-affine measures on $K$ (which are also called Mandelbrot measures on $K$). We will prove the following Ledrappier-Young type formula: given a Mandelbrot measure $\mu$ on $K$ (see Sections~\ref{Mandmeas} and~\ref{euclidean} for the definition),
\begin{align}
\nonumber\dim (\mu)&= \frac{1}{\log(m_1)}\dim_e(\mu)+\sum_{i=2}^k \Big (\frac{1}{\log(m_i)}-\frac{1}{\log(m_{i-1})} \Big )\dim_e(\mu\circ\Pi_i^{-1})\\
\label{LYr}&= \frac{1}{\log(m_1)}\dim_e(\mu)+\Big (\frac{1}{\log(m_i)}-\frac{1}{\log(m_{i-1})} \Big )\cdot \min \big (\dim_e(\mu), h_{\nu_i}(\Pi_i\circ f)\big ),
\end{align}
where $\dim_e(\mu)$ is the dimension entropy of $\mu$, and $\nu_i$ is the Bernoulli product measure  $\mathbb{E}(\mu\circ\Pi_i^{-1})$ (see Theorem~\ref{thmdim}). The fact that $\dim_e(\mu\circ\Pi_i^{-1})= \min (\dim_e(\mu), h_{\nu_i}(\Pi_i\circ f))$ follows from our previous study of projections of Mandelbrot measures in~\cite{BF2016}. To get \eqref{LYr}, we show that $\tau_\mu$, the $L^q$-spectrum of $\mu$, is differentiable at 1 with $\tau_\mu'(1)$ equal to the right hand side of \eqref{LYr}; this implies the exact dimensionality of $\mu$, with dimension equal to $\tau_\mu'(1)$.

Optimising \eqref{LYr} yields the sharp lower bound for the Hausdorff dimension of~$K$ (see  Theorem~\ref{optim});  the supremum is uniquely attained, and the optimisation problem is non-standard; the presence of the $k-1$ minima in the sum gives rise to $k$ possible simplifications of the formula separated by what  can be thought of as $k-1$ phase transitions according to the position of $\dim_e(\mu)$ with respect to the entropies $h_{\nu_i}(\Pi_i\circ f)$, hence $k$ distinct optimisation problems must be considered, of which the optima must be compared. This study will use the thermodynamic formalism as well as a version of von Neumann's min-max theorem, and the optimal Hausdorff dimension will be expressed as the ``weighted'' pressure of some deterministic potential (see Theorem~\ref{dim K}). Our sharp  upper bound for $\dim_H K$ proves that this  maximal Hausdorff dimension  of a Mandelbrot measure supported on $K$ yields $\dim_H K$. This bound is derived from a variational principle as well, namely we optimise over uncountably many types of  coverings of $K$, each of which provides an upper bound for $\dim_H K$ (see Theorem~\ref{VPUB});  when $k=2$, this does not reduce back to the argument developed in \cite{GL94}. As a by-product, we get an alternative approach to the proof by Kenyon and Peres \cite{KP} of the sharp upper bound in the deterministic higher dimensional case. One may wonder if the approach by Kenyon and Peres, which in the deterministic case uses a uniform control of the lower local dimension of the unique Bernoulli measure of maximal Hausdorff dimension on $K$, can be extended to the random case by using the unique Mandelbrot measure of maximal Hausdorff dimension on $K$. We met an essential difficulty in trying to follow this direction, except in special cases (see the discussion at the beginning of Section~\ref{UB}).

Our result regarding the box-counting dimension of $K$ is stated in Theorem~\ref{thmdimB}. Another difference between the deterministic or random dynamical Sierpinski sponges, and the random Sierpinski  sponges studied in this paper is that for the first two classes of objects, their images by the natural projections  $\Pi_i$, $2\le i\le k$, are lower dimensional objects of the same type (e.g. Sierpinski sponges project to 2-dimensional Sierpinski carpets via $\Pi_{k-1}$),  while this is not the case for the third one. Our approach will also provide $\dim_H \Pi_i(K)$ (via a variational principle) and $\dim_B \Pi_i(K)$ for the third class of random attractors (see Section~\ref{Projections}; we note that the case of  $\Pi_k(K)$ reduces to that  of $\Pi_2(K)$ when $k=2$ and $K$ is a statistically self-similar set, a situation which is covered by \cite{DekGri,Fal}). Finally, we determine the dimension of the conditional measures associated with the successive images of any  Mandelbrot measure by the  projections $\Pi_i$ (Section~\ref{condmeas}), which for each $i$ equals $\dim (\mu)-\dim({\Pi_i}\mu)$, hence the conservation dimension formula holds.

Since using symbolic spaces to encode the Euclidean situation is necessary, we will work on such spaces and endow them with adapted ultrametrics, so that the case of  random statistically self-affine Euclidean sponges  and their projections will be reducible to a particular situation of a more general framework on symbolic spaces and their factors.

Our framework and main results are presented in the next section.

\section{Main results on self-affine symbolic spaces, and application to the Euclidean case}

We start with defining the symbolic random statistically self-affine sponges, which will be studied in this paper.

\subsection{Symbolic random statistically self-affine sponges}
\label{S-2.1}
Let us first recall the notion of self-affine symbolic space.

Let $\mathbb N$,  $\mathbb N^*$, $\mathbb R_+$ and $\mathbb R_+^*$ stand for the sets of non-negative integers, positive integers, non-negatice real numbers, and positive real numbers respectively.

Let $k\geq 2$ be an integer. Assume that $(X_i, T_i)$ ($i=1,\ldots, k$) are
 one-sided full-shift spaces over finite alphabets $\A_i$ of cardinality $\ge 2$ and such that $X_{i+1}$ is a factor of $X_i$ with a one-block
factor map $\pi_i: \;X_i\to X_{i+1}$ for $i=1,\ldots, k-1$. That is,  $\pi_i$ is extended from a map (which is also denoted by $\pi_i$ for convenience)  from  $\A_i$ to $\A_{i+1}$ by  $$
\pi_i\left((x_n)_{n=1}^\infty\right)=\left({\pi}_i(x_n)\right)_{n=1}^\infty \mbox{  for all }
(x_n)_{n=1}^\infty\in X_i.$$
 For
convenience, we use $\pi_0$ to denote the identity map on $X_1$.
Define $\Pi_i:\;X_1\to X_{i}$ by $$\Pi_i=\pi_{i-1}\circ\pi_{i-2}\circ
\cdots \circ \pi_0$$ for $i=1,\ldots,k$.
Define $\mathcal A_i^*=\bigcup_{n\ge 0}\mathcal A_i^n$, where $\mathcal A_i^0$ consists of the empty word $\epsilon$. The maps $\pi_i$ and $\Pi_i$ naturally extend to $\mathcal A_i^*$ and $\mathcal A_1^*$ respectively,  that is,  $$\pi_i(x_1\cdots x_n)=\pi_i(x_1)\cdots \pi_i(x_n),\quad \Pi_i(x_1\cdots x_n)=\Pi_i(x_1)\cdots \Pi_i(x_n).
$$

If $u\in\mathcal A_i^*$, we denote by $[u]$ the cylinder made of those elements in $X_i$ which have $u$ as prefix. If $x\in X_i$ and $n\ge 0$, we denote by $x_{|n}$ the prefix of $x$ of length $n$.  For $x, y\in X_i$, let $x\land y$ denote the longest common prefix of $x$ and $y$.

Let $\vec{\boldsymbol\gamma}=(\gamma_1,\ldots, \gamma_k)\in \R_+^*\times (\R_+)^{k-1}$. Define  an ultrametric distance $d_{\vec{\boldsymbol\gamma}}$ on $X_1$ by
\begin{equation}
\label{e-da}
d_{\vec{\boldsymbol\gamma}}(x,y)=\max\left(e^{-\frac{|\Pi_i(x)\land \Pi_i(y)|}{\gamma_1+\cdots+\gamma_{i}}}:\;1\le i\le k\right),
\end{equation}
where
$$|u\land v|=\left\{ \begin{array}{ll}
0, & \mbox{ if } u_1\neq v_1,\\
\max\{n:\; u_j=v_j\mbox{ for } 1\leq j\leq n \} & \mbox{ if } u_1=v_1\\
\end{array}
\right.
$$
for $u=(u_j)_{j=1}^\infty, v=(v_j)_{j=1}^\infty\in X_i$.

The metric space $(X_1, d_{\vec{\boldsymbol\gamma}})$ is called a {self-affine symbolic space}. It is a natural model used to characterize the geometry of  compact invariant sets on the $k$-torus under a diagonal endomorphism \cite{Bed84, McMullen, KePe96,BF2012}.

If $f$ is a measurable mapping between a measured space $(X,\mathcal T,\nu)$ and a measurable space $(Y,\mathcal S)$, the image $\nu\circ f^{-1}$ of $\nu$ will be simply denote by $f\nu$.

Now, we can define  symbolic random statistically self-affine sponges. Let $A=(c_a)_{a\in\mathcal A_1}$ be a random variable taking values in $\{0,1\}^{\mathcal A_1}$. It encodes a random subset of $\mathcal A_1$, namely $\{a\in\mathcal A_1: \; c_a=1\}$, which we identify with $A$. Suppose that $\mathbb{E}(\sum_{a\in \mathcal A_1}c_a)>1$.
 Let $(A(u))_{u\in\mathcal A_1^*}$ be a sequence of independent copies of $A$. For all $n\in\N$, let
\begin{equation}\label{Kn}
\begin{split}
K_n& =\{x\in X_1: \; c_{x_i}(x_{|i-1})=1  \mbox{ for all }1\le i\le n\}\\
&=\bigcup_{u\in \mathcal A_1^n:\  \prod_{i=1}^nc_{u_i}(u_{|i-1})=1} [u].
\end{split}
\end{equation}

Due to our assumption on $A$, with positive probability, the set
$$
K=\bigcap_{n\ge 1} K_n,
$$
is the boundary of a non-degenerate Galton Watson tree with offspring distribution given by that of the random integer $\sum_{a\in \mathcal A_1}c_a$. The set $K$ satisfies the following symbolic statistically self-affine invariance property:
$$
K=\bigcup_{a\in A} a\cdot K^a,
$$
where
\begin{equation}\label{Ka}
K^a= \bigcap_{n\ge 1}\Big (K^a_n=\bigcup_{u\in \mathcal A_1^n: \ \prod_{i=1}^nc_{u_i}(au_{|i-1})=1} [u]\Big ).
\end{equation}
We call $K$  a symbolic statistically self-affine Sierpinski sponge. The link with the Euclidean case will be made in Section~\ref{euclidean}.

\medskip

Next we recall the definition and basic properties of Mandelbrot measures on $K$, and state our result on their exact dimensionality.

\subsection{Mandelbrot measures on $K$}\label{Mandmeas}
These measures will play an essential role in finding a sharp lower bound for $\dim_H K$. Let $W=(W_a)_{a\in \mathcal A_1}$ be a non-negative random vector defined simultaneously with $A$ and such that $\{W_a>0\}\subset \{c_a=1\}$ for all $a\in\mathcal A_1$.  Let
$$
T(q)=T_W(q)=-\log \mathbb E\sum_{a\in\mathcal A_1} W_a^q,\quad q\ge 0,
$$
and suppose that $T(1)=0$, i.e. $\mathbb E(\sum_{a\in\mathcal A_1} W_a)=1$. Let $((A(u),W(u))_{u\in\mathcal A_1^*}$ be a sequence of independent copies of $(A,W)$.

For each $u, v\in \mathcal A_1^*$ let
$$
Q^u(v)=\prod_{k=1}^{|v|} W_{v_k}(u\cdot v_{|k-1}),
$$
and simply denote  $Q^\epsilon(v)$ by $Q(v)$. Due to our assumptions on $W$, for each $u\in \mathcal A_1^*$, setting $Y_n(u)=\sum_{|v|=n} Q^u(v)$ the sequence
$$(Y_n(u),\sigma (W_a(uv):a\in\mathcal A_1, v\in\mathcal A_1^*))_{n\ge 1}$$ is  a non-negative martingale. Denote its limit almost sure limit by~$Y(u)$. The mapping
$$
\mu: [u]\mapsto Q(u)Y(u)
$$
defined on the set of cylinders $\{[u]:u\in  \mathcal A_1^*\}$ extends to a unique measure $\mu$ on $(X_1,\mathcal B(X_1))$. This measure was first introduced in \cite{mandelbrot}, and is called a Mandelbrot measure. The support of $\mu$ is the set
$$
K_\mu=\bigcap_{n\ge 0}\bigcup_{u\in \mathcal A_1^n: \ Q(u)>0} [u]\subset K,
$$
where the inclusion $K_\mu\subset K$ follows from the assumption $$\{W_a(u)>0\}\subset\{c_a(u)=1\}$$ for all $u\in\mathcal A_1^*$ and all $a\in\mathcal A_1$. Moreover, if $T'(1-)>0$, then $(Y_n(u))_{n\ge 1}$ is uniformly integrable (see \cite{KP,Biggins,DuLi}), and in this case  $K_\mu$ is a symbolic statistically self-affine set as well.  Also, $K_\mu=K$ almost surely if and only if  $\{c_a(u)=1\}\setminus \{W_a>0\}$ has probability~0 (see \cite[Proposition A.1]{BF2016}). If $T'(1-)\le 0$, either $\mu=0$ almost surely (in such case one says that $\mu$ is degenerate), or
$$\mathbb P(\exists \, a\in  \A_1 \mbox { such that } W_a=1\mbox{ and } W_{a'}=0 \text{ for  }a'\neq a)=1$$
  (see \cite{KP,Biggins,DuLi} as well), and in this later case $T'(1-)=0$, $K_\mu$ is a singleton and~$\mu$ is a Dirac measure \cite{DuLi}.

The measure $\mu$ is also the weak-star limit of the sequence $(\mu_n)_{n\ge 1}$ defined by distributing uniformly (with respect to the uniform measure on $X_1$) the mass $Q(u)$ on each $u\in\mathcal A_1^n$.  It is statistically self-affine in the sense that
$$
\mu(B)=\sum_{a\in\mathcal A_1} W_a\, \mu^a(\sigma (B\cap[a]))
$$
for every Borel set $B \subset X_1$, where $\mu^a$ is the copy of $\mu$ constructed on $K^a$ with the weights $(W(au))_{u\in\mathcal A_1^*}$.

We will make a systematic use of the notion of entropy dimension of  measures on $X_i$. Given a  finite Borel measure $\nu$ on $X_i$,   the entropy dimension  of $\nu$ is defined as
\begin{equation}
\label{e-entropy}
\dim_e(\nu)=\lim_{n\to \infty} -\frac{1}{n}\sum_{u\in \mathcal A_i^n} \nu ([u])\log \nu([u]),
\end{equation}
whenever the limit exists.  If $\nu$ is $T_i$-invariant, one has  $\dim_e(\nu)=h_\nu(T_i)$, the entropy of $\nu$ with respect to $T_i$ (see e.g. \cite{Bowen1975}).

Due to the results by Kahane and Peyri\`ere \cite{KP,K87}, when a Mandelbrot measure $\mu$ is non-degenerate (that is, when $\mathbb P(\mu\neq 0)>0$), with probability 1, conditional on $\mu\neq 0$, one has
$$
\lim_{n\to\infty} \frac{\log(\mu([x_{|n}]))}{-n}=T'(1-)=-\sum_{a\in\mathcal A_1} \mathbb{E}(W_a\log W_a),\quad \text{for $\mu$-a.e. $x$}.
$$
It then follows that $\dim_e(\mu)$ exists \cite{FanLauRao} and $\dim_e(\mu)=T'(1-)$.
In particular,
\begin{equation}\label{maj dime mu}
\dim_e(\mu)\le \log \mathbb E(\# A),
\end{equation}
with equality if and only if $\mu$ is the so called branching measure, i.e. it is obtained from the random vector $(\mathbf{1}_{A}(a) /\mathbb E(\# A))_{a\in\mathcal A_1}$ (see Remark~\ref{rem0} for a justification of this uniqueness).

Before stating our first result, we present a result deduced from \cite{BF2016} about the entropy dimension of ${\Pi_i}\mu$ (in \cite{BF2016} we proved  the exact dimensionality of the projection of a Mandelbrot measure on $X_1\times X_1$ to the first factor $X_1$, in the case that $X_1$ is endowed with~$d_{\vec{\boldsymbol\gamma}}$, where ${\vec{\boldsymbol\gamma}}=((\log\#\mathcal A_1)^{-1})$; but projecting from $X_1$ to $X_i$ and considering the entropy dimension do not affect the arguments):
\begin{thm}\cite[Theorem 3.2]{BF2016}\label{projmu} Let $\mu$ be a non-degenerate Mandelbrot measure on $K$. Suppose that $T(q)>-\infty$ for some $q>1$. With probability 1, conditional on $\mu\neq 0$, for all $2\le i\le k$, one has $\dim_e({\Pi_i}\mu)=\min (\dim_e(\mu),h_{\nu_i}(T_i))$, where $\nu_i$ is the Bernoulli product measure on $X_i$ obtained as $\nu_i=\mathbb E({\Pi_i}\mu)$.
\end{thm}

Hence, $\dim_e(\mu)$ and $h_{\nu_i}(T_i)$ compete in the determination of the entropy dimension of the $i^{\rm th}$ projection of $\mu$.

Recall that a locally finite Borel measure $\nu$ on a metric space $(X,d)$ is said to be exact dimensional with dimension $D$ if $$\lim_{r\to 0+} \frac{\log(\nu(B(x,r)))}{\log(r)}=D$$ for $\nu$-almost every $x$.  We denote the  number $D$ by $\dim(\nu)$ and call it the dimension of~$\nu$. Our result on the exact dimensionality of the Mandelbrot measure $\mu$ on $(X_1,d_{\vec{\boldsymbol\gamma}})$ is the following.

\begin{thm}[Exact dimensionality of $\mu$]\label{thmdim} Let $\mu$ be a non-degenerate Mandelbrot measure on $K$. Suppose that $T(q)>-\infty$ for some $q>1$.
With probability 1, conditional on $\mu\neq 0$, the measure $\mu$ is exact dimensional and
$\dim (\mu)=\dim^{\vec{\boldsymbol\gamma}}_e(\mu)$, where
\begin{equation}
\label{e-LY}
\dim^{\vec{\boldsymbol\gamma}}_e(\mu):=\sum_{i=1}^{k} \gamma_i \dim_e({\Pi_i}\mu)
=\gamma_1\dim_e(\mu)+\sum_{i=2}^{k} \gamma_i \min (\dim_e(\mu), h_{\nu_i}(T_i)).
\end{equation}
\end{thm}

This result will follow from the stronger fact that the $L^q$-spectrum of $\mu$ is differentiable at $1$ with derivative $\dim^{\vec{\boldsymbol\gamma}}_e(\mu)$ (see Theorem \ref{thm-3.1}).

\subsection{The Hausdorff and box counting dimensions of $K$}\label{dimsK}

To state our result on $\dim_H K$, we need to recall some elements of the weighted thermodynamic formalism.

For  $1\le i\le k$ and $b\in \mathcal A_i=\Pi_i(\mathcal A_1)$, let
\begin{equation}
\label{e-Nib}
N^{(i)}_b=\# \{a\in\mathcal A_1:\; [a]\subset  \Pi_i^{-1}([b]) \mbox{ and }[a]\cap K\neq\emptyset\}.
\end{equation}
Then set $$\widetilde{\mathcal A}_i=\{b\in\mathcal A_i: \; \mathbb{E}(N^{(i)}_b)>0\}.$$
Without loss of generality we assume that $\# \widetilde{\mathcal A}_k\ge 2$. Indeed, if $\widetilde{\mathcal A}_k$ is a singleton, then $X_k$ plays no role in the geometry of $K$ since $\Pi_k(K)$ is a singleton when $K\neq\emptyset$. As a consequence,  $\# \widetilde{\mathcal A}_i\ge 2$ for all $2\le i\le k$.

 For $1\le i\le k$, let $\widetilde X_i$ denote the one-sided symbolic space over the alphabet  $\widetilde \A_i$. If $\phi:\widetilde X_i\to\mathbb R$ is a continuous function on $\widetilde X_i$,  $\vec{\boldsymbol\beta} =(\beta_{i},\beta_{i+1},\ldots, \beta_k)\in\mathbb{R}_+^*\times\mathbb R_+^{k-i}$, and $\nu\in\mathcal M(X_i,T_i)$, let
\begin{equation}
\label{e-h-beta}
h_\nu^{\vec{\boldsymbol\beta}}(T_i):=\sum_{j=i}^{k} \beta_jh_{{\Pi_{i,j}}{\nu}}(T_j),
\end{equation}
where
\begin{equation}\label{piij}
\Pi_{i,j}=\pi_{j-1}\circ
\cdots \circ \pi_i \quad\text{if }j>i
\end{equation}
 and $\Pi_{i,i}$ is the identity map of $X_i$, and define the weighted pressure function
\begin{equation}
\label{e-weighted-pressure}
P^{\vec{\boldsymbol\beta}}(\phi,T_i)=\sup\Big \{\nu(\phi)+h_\nu^{\vec{\boldsymbol\beta}}(T_i): \; \nu\in \mathcal M(\widetilde X_i,T_i)\Big \},
\end{equation}
where $\nu(\phi)=\int_{\widetilde X_i}\phi\, \mathrm{d}\nu$. It is known (\cite{BF2012}) that if $\phi$ is H\"older-continuous, then the supremum in defining $P^{\vec{\boldsymbol\beta}}(\phi,T_i)$ is reached at a unique fully supported measure $\nu_\phi$, called the equilibrium state of $\phi$ for $P^{\vec{\boldsymbol\beta}}(\phi,T_i)$. Moreover, the mapping $\theta\mapsto P^{\vec{\boldsymbol\beta}}(\theta\phi,T_i)$ is differentiable, and
\begin{equation}\label{diff}
\frac{\mathrm{d}P^{\vec{\boldsymbol\beta}}(\theta \phi,T_i))}{\mathrm{d}\theta} =\int_{\widetilde X_i}\phi(x)\,\mathrm{d}\nu_{\theta\phi}(x)=\nu_{\theta\phi}(\phi).
\end{equation}


For $2\le i\le k$  let ${\vec{\boldsymbol\gamma}^i}=(\vec{\boldsymbol\gamma}^i_j)_{i\le j\le k}=(\gamma_1+\cdots +\gamma_i,\gamma_{i+1},\ldots, \gamma_k)$ and  let $\phi_i$ be the H\"older-continuous potential defined on $\widetilde X_i$ by
\begin{equation}\label{phiktheta}
\phi_i(x)=(\gamma_1+\cdots +\gamma_i)  \log \mathbb{E} (N^{(i)}_{x_1}).
\end{equation}
For this potential and $\vec{\boldsymbol \beta}={\vec{\boldsymbol\gamma}^i}$, we set
\begin{equation}
\label{e-p-i}
P_i(\theta)=P^{{\vec{\boldsymbol\gamma}^i}}(\theta \phi_i,T_i),  \quad \theta\in \R.
\end{equation}
We also define $P_{k+1}=P_k$.

The equilibrium state of $\theta\phi_i$ for $P^{{\vec{\boldsymbol\gamma}^i}}$ will be also called the equilibrium state of $P_i$ at $\theta$.

By \eqref{diff}, we have
\begin{equation}\label{diff2}
P_i'(\theta)=(\gamma_1+\cdots +\gamma_i)\sum_{b\in\widetilde{\mathcal A}_i}\nu_{\theta\phi_i}([b])\log \mathbb{E} (N^{(i)}_{b}) \quad (\theta\in\mathbb R).
\end{equation}

We now define some parameters involved in the next statement. In order to slightly simplify the exposition, we assume that all the $\gamma_i$ are positive. The general situation will be considered in  Section~\ref{lastsection}.

Set $I=\{2,\ldots,k\}$ (introducing this convention will simplify the discussion in Section~\ref{lastsection}),
\begin{equation}\label{tildethetai}
\widetilde\theta_i= \displaystyle \frac{\gamma_1+\cdots + \gamma_{i-1}}{\gamma_1+\cdots +\gamma_{i}}\quad\text{if }2\le i\le k,
\end{equation}
and define
$$
\widetilde I=  \left\{i\in I: \,  P_i'(1)\ge 0 \right \} .
$$

Moreover, define
 \begin{equation}
\label{e-i_0}
 i_0=\left\{ \begin{array}{ll}
 \min(\widetilde I) & \mbox{ if } \widetilde I\ne \emptyset\\
 k+1 & \mbox{ if } \widetilde I=\emptyset
 \end{array}
 \right.
 \end{equation}
 and
 \begin{equation}
 \label{e-theta_0}
  \theta_{i_0}=\left\{ \begin{array}{ll}
\min\left\{\theta\in  [\widetilde\theta_{i_0},1]:\; P_{i_0}'(\theta)\ge 0 \right \} & \mbox{ if } i_0\leq k\\
1 & \mbox{ if } i_0= k+1
 \end{array}
 \right..
\end{equation}


\begin{thm}[Hausdorff dimension of $K$]\label{dim K} With probability 1, conditional on $K\neq\emptyset$,
\begin{align}
\nonumber \dim_H K&=\sup \{\dim^{\vec{\boldsymbol\gamma}}_e(\mu): \; \mu\text{ is a positive Mandelbrot measure supported on  $K$}\}\\
\label{formula1} &=\begin{cases}
P_{i_0}(\theta_{i_0}) &\text{ if } i_0\le k\\
P_{k}(1) &\text{ if }i_0=k+1
\end{cases}\\
\nonumber&=\inf\Big \{P_i(\theta): \;  i\in I, \, \widetilde\theta_i\le \theta\le 1\Big \}.
\end{align}
Moreover, there is a unique Mandelbrot measure jointly defined with $K$ at which the supremum is attained, conditional on $K\neq\emptyset$.
\end{thm}
The Mandelbrot measure at which the supremum is attained will be specified in Section~\ref{OPT}. It is constructed from the equilibrium state of $P_{i_0}$ at $\theta_{i_0}$.


\begin{thm}[Box counting dimension of $K$]\label{thmdimB}
With probability 1, conditional on $K\neq\emptyset$,
$$
\dim_B K=\gamma_1\log \mathbb E(\# A) +\sum_{i=2}^k \gamma_i \min_{\theta\in[0,1]}\log\sum_{b\in \widetilde {\mathcal A}_i} \mathbb E(N^{(i)}_b)^\theta.
$$
\end{thm}

Next we give a necessary and sufficient condition for $\dim_H K= \dim_B K$. Define
\begin{equation}
\label{e-psi_i}
\psi_i: \theta\in[0,1]\mapsto  \log\sum_{b\in \widetilde {\mathcal A}_i} \mathbb E(N^{(i)}_b)^\theta \quad  (2\le i\le k).
\end{equation}
For each $2\le i\le k$, denote by $\widehat\theta_i$ the point in $[0,1]$ at which $\psi_i$ reaches its minimum if $\psi_i$ is not constant (i.e. there is $b\in  \widetilde{\mathcal A}_i$ such that $\mathbb E(N^{(i)}_b)\neq 1$), and $\widehat\theta_i=0$ otherwise.

We will need the following lemma to state and prove Corollary~\ref{cor1} about the necessary and sufficient condition for the equality $\dim_H K= \dim_B K$ to hold.
\begin{lem}
 For each $2\leq i\leq k$,  $\psi_i$ takes the  value $\log \mathbb E(\# A)$ at $\theta=1$. Moreover if $2\le i\le k-1$, then $\psi_{i}\ge \psi_{i+1}$, and  $\widehat \theta_i<1$ implies $\widehat \theta_{i+1}<1$.
\end{lem}
\begin{proof}
The first property is due to the relation $\mathbb{E}(N^{(i+1)}_{\widetilde b})=\sum_{b\in \widetilde {\mathcal A}_i: \, \pi_i(b)=\widetilde b}\mathbb{E} (N^{(i)}_b)$ for any $\widetilde b\in\widetilde{\mathcal A}_{i+1}$, and the second one is due both to this property and the subadditivity of $y\ge 0\mapsto y^\theta$.  The third property is a direct consequence of the first two properties. Indeed if $\widehat\theta_{i+1}=1$, then by the definition of $\widehat\theta_{i+1}$ and the convexity of $\psi_{i+1}$, we have for every $\theta \in [0,1)$,  $\psi_{i+1}(\theta)>\psi_{i+1}(1)=\log \mathbb E(\# A)$ and so   $$\psi_{i}(\theta)\geq \psi_{i+1}(\theta)>\log \mathbb E(\# A)=\psi_i(1),$$ which implies that $\widehat\theta_i=1$.
\end{proof}
\begin{cor} \label{cor1}
It holds that $\dim_H K= \dim_B K$ with probability 1, conditional on $K\neq\emptyset$, if and only if $\mathbb{E}(N_b^{(i)})$ does not depend on $b\in  \widetilde{\mathcal A}_i$ for all $i\in I$ such that $\widehat \theta_i<1$.
\end{cor}

\subsection{Dimensions of projections of $\mu$ and $K$}\label{Projections}

We still assume that all the $\gamma_i$ are positive and will  discuss the general case in Section~\ref{lastsection}. In the next statement $\mu$ is assumed to be a non-denenerate Mandelbrot measure such that  $K_\mu\subset K$ almost surely.
\begin{thm}
[Dimension of ${\Pi_i}\mu$]\label{thmdim2} Let $2\le i\le k$.  Suppose that $T(q)>-\infty$ for some $q>1$.
With probability 1, conditional on $\mu\neq 0$, the measure ${\Pi_i}\mu$ is exact dimensional and
$\dim ({\Pi_i}\mu)=\dim^{{\vec{\boldsymbol\gamma}}^i}_e({\Pi_i}\mu)$, where
$$
\dim^{{\vec{\boldsymbol\gamma}}^i}_e({\Pi_i}\mu):= \sum_{j=i}^k \vec{\boldsymbol\gamma}^i_j\dim_e({\Pi_j}\mu)=\sum_{j=i}^k  \vec{\boldsymbol\gamma}^i_j   \min (\dim_e(\mu), h_{\nu_j}(T_j)).
$$
\end{thm}

Next we consider $\dim_H\Pi_i(K)$. For this purpose,   let $i\in \{2,\ldots, k\}$. Define $ \overline \theta_i^i=0$ and $ \overline \theta_j^i=\widetilde \theta_j$ for $j=i+1, \ldots, k$ (cf.~\eqref{tildethetai} for the definition of $\widetilde \theta_j$). Set  $I_i=\{i,\ldots, k\}$ and define $$ \widetilde I_i=\{ j\in I_i:  \,P_j'(1)\ge 0\}.$$

 Then define  $$j_0:=j_0(i)=\left\{
\begin{array}{ll}  \min  ( \widetilde I_i) & \mbox { if } I_i\neq\emptyset\\
k+1 & \mbox{ otherwise}
\end{array}
\right.
.
$$
Also, set
\begin{equation}
\theta_{j_0}^i=\left\{ \begin{array}{ll} \min \left \{\theta\in [\overline\theta^i_{j_0},1]: P_{j_0}'(\theta)\ge 0\right \}& \mbox{  if } j_0\le k\\
1 & \mbox{ if }j_0=k+1
\end{array}
\right.
.
\end{equation}

\begin{thm}
[Hausdorff dimension of $\Pi_i(K)$]\label{dim PiK} Let $2\le i\le k$. Let $j_0=j_0(i)$ and $\theta^i_{j_0}$ be defined as above. Then, with probability 1, conditional on $K\neq\emptyset$,
\begin{align}
\nonumber\dim_H \Pi_i(K)&=\sup \{\dim^{{\vec{\boldsymbol\gamma}}^i}_e({\Pi_i}\mu): \, \mu\text{ is a positive Mandelbrot measure supported on ~$K$}\}\\
\label{formula2}&=\begin{cases}
P_{j_0}(\theta^i_{j_0})&\text{ if } j_0\le k\\
P_{j_0}(1)&\text{ if }j_0= k+1
\end{cases},
\end{align}
and the supremum is uniquely attained at a Mandelbrot measure jointly defined with $K$ if and only  if one of the following conditions is fulfilled:
 \begin{itemize}
 \item[(a)]
 ${j_0}>i$;
 \item[(b)]   $j_0=i$ and  $\theta^i_{j_0}>0$;
 \item[(c)] $j_0=i$, $\theta^i_{j_0}=0$ and  $P_{j_0}'(0)=0$.
 \end{itemize}
\end{thm}

\begin{rem}
We deduce from Theorems~\ref{dim K} and ~\ref{dim PiK} that: \begin{itemize}
\item[(i)] if $2\le i<i_0$ then $\dim_H \Pi_{i}(K)=\dim_H K$.
\item[(ii)] $\dim_H K=\dim_H \Pi_{i_0}(K)$ if and only if $i_0\le k$ and $P_{i_0}'(\theta_{i_0})=0$, or if $i_0= k+1$. Indeed, it is clear that the condition is sufficient. Now suppose that  $\dim_H K=\dim_H \Pi_{i_0}(K)$. If $i_0\le k$  and $P_{i_0}'(\theta_{i_0})>0$ then  $\theta_{i_0}=\widetilde \theta_{i_0}>0$. Moreover,  $j_0(i_0)=i_0$ and  $P_{i_0}'(\widetilde\theta_{i_0})>0$ implies that $\theta^{i_0}_{i_0}<\widetilde\theta_{i_0}$, so $\dim_H \Pi_{i_0}(K)=P_{i_0}(\theta^{i_0}_{i_0})<\dim_H K=P_{i_0}(\widetilde\theta_{i_0})$. Thus, if $i_0\le k$ then $P_{i_0}'(\theta_{i_0})=0$.
\item[(iii)] If $2\le i<j\le k$, then \begin{itemize}
\item[ (a)] if $2\le i<j<j_0(i)$, then  $\dim_H \Pi_{j}(K)=\dim_H \Pi_{i}(K)$;
\item[(b)] $\dim_H \Pi_{j_0(i)}(K)=\dim_H \Pi_{i}(K)$ if and only if $j_0(i)\le k$ and $P_{j_0(i)}'(\theta_{j_0})=0$, or $j_0(i)=k+1$. This can be proven like (ii).
\end{itemize}
\end{itemize}
\end{rem}

\begin{thm}
[Box counting dimension of $\Pi_i(K)$]\label{dimB PiK} Let $2\le i\le k$. With probability 1, conditional on $K\neq\emptyset$,
$$
\dim_B \Pi_i(K)=\sum_{j=i}^k \vec{\boldsymbol\gamma}^i_j\min_{\theta\in[0,1]}\log\sum_{b\in \widetilde{\mathcal A}_j} \mathbb E(N^{(j)}_b)^\theta.
$$
\end{thm}
\begin{cor}\label{cor2}Let $2\le i\le k$. It holds that $\dim_H \Pi_i(K)= \dim_B \Pi_i(K)$ with probability~1, conditional on $K\neq\emptyset$, if and only if either of the following three conditions hold:
\begin{enumerate}
\item $\widehat \theta_i=1$ and $\mathbb{E}(N_b^{(j)})$ does not depend on $b\in \widetilde{\mathcal A}_j$ for all $j\in I_i\setminus\{i\}$ such that $\widehat\theta_j<1$.

\item  $0<\widehat \theta_i<1$, or  $\widehat \theta_i=0$ and $\psi_i'(0)=0$. Moreover setting
$$j'_0=\max\{j\in I_i:  \mbox{ for all }j'\le j \mbox{ in }I_i, \mbox{ either } 0<\widehat \theta_{j'}<1, \mbox{ or } \widehat \theta_{j'}=0 \mbox{ and }\psi_{j'}(0)=0\},
$$
 then the follow hold:
\begin{itemize}
\item[(i)]  For all $ j\in I_i$ such that $j\le j'_0$,  for all $ b\in \widetilde{\mathcal A}_j$, $\Pi_{i,j}^{-1}(b)\cap \mathcal A_i$ is a singleton (in particular $\widehat  \theta_j=\widehat  \theta_i$);
\item[(ii)]  For all $ j\in I_i$ such that $j>j'_0$ one has $\widehat  \theta_j=0$ and $\psi_j'(0)>0$, and  $$\sum_{b'\in \widetilde{\A}_i:\; [b']\subset  \Pi_{i,j}^{-1}([b])} \mathbb{E}(N_{b'}^{(i)})^{\widehat  \theta_i}$$ does not depend on $b\in\widetilde{\mathcal A}_j$.
\end{itemize}

\item  For all $ j\in I_i$, $\widehat  \theta_j=0$,  $\psi_j'(0)>0$, and  $\#\Pi_{i,j}^{-1}( b)\cap\mathcal A_i$ does not depend on $b\in\widetilde{\mathcal A}_j$.
\end{enumerate}
\end{cor}

In the random case, Theorem~\ref{dim PiK}, Theorem~\ref{dimB PiK} and Corollary~\ref{cor2} are new except in the case $i=k$, which is reducible to the two dimensional case considered in   \cite{DekGri,Fal,BF2016}.

\subsection{Dimensions of conditional measures}\label{condmeas} Given a non-degenerate Mandelbrot measure $\mu$, conditional on $\mu\neq 0$, for each $2\le i\le k$, the measure $\mu$ disintegrates as the skewed product of ${\Pi_i}\mu (\mathrm{d}z)\,\mu^z(\mathrm{d}x)$, where $\mu^z$ is the conditional measure supported on $\Pi_i^{-1}(\{z\})\cap K$ for ${\Pi_i}\mu$-almost every $z$. We will prove  the exact dimensionality of the measures $\mu^z$ and the value for their dimensions.

Let us start with a consequence of \cite[Theorems 3.1 and 3.2]{BF2016}:

\begin{thm}\label{dimentr}Let $\mu$ be a non-degenerate Mandelbrot measure supported on $K$ and $2\le i\le k$. Let $\nu_i$ be the Bernoulli product measure equal to $\mathbb E({\Pi_i}\mu)$. With probability 1, conditional on $\mu\neq 0$, ${\Pi_i}\mu$ is absolutely continuous with respect to $\nu_i$ if $\dim_e(\mu)>h_{\nu_i}(T_i)$, otherwise ${\Pi_i}\mu$ and $\nu_i$ are mutually singular.
Moreover, if $T(q)>-\infty$ for some $q>1$, then for ${\Pi_i}\mu$-a.e.~$z\in \Pi_i(K)$ and $\mu^z$-a.e.~$x\in K$, $$\lim_{n\to\infty} \frac{\log (\mu^z([x_{|n}]))}{-n}=\dim_e(\mu)-\dim_e({\Pi_i}\mu),$$  which is equal to  $\dim_e(\mu)-h_{\nu_i}(T_i)$ if $\dim_e(\mu)>h_{\nu_i}(T_i)$ and 0 otherwise;  in particular
 the entropy dimension of $\mu^z$ exists and  equals $\dim_e(\mu)-\dim_e({\Pi_i}\mu)$.
\end{thm}
It is worth mentioning that the existence of the local entropy dimension for $\mu^z$ and the entropy dimension conservation formula comes from the study achieved in~\cite{FalJin} for the self-similar case, while the alternative between singularity and absolute continuity regarding ${\Pi_i}\mu$, as well as the value of $\dim_e({\Pi_i}\mu)$ and so that of $\dim_e(\mu^z)$ are obtained in \cite{BF2016}.

For the Hausdorff dimension of the conditional measures, we prove the following result:
\begin{thm}\label{thmfibers} Let $\mu$ be a non-degenerate Mandelbrot measure supported on $K$ and $2\le i\le k$. Let $\nu_i$ be the Bernoulli product measure equal to $\mathbb E({\Pi_i}\mu)$.  Suppose that $T(q)>-\infty$ for some $q>1$. With probability 1, conditional on $\mu\neq 0$:
\begin{enumerate}
\item If $\dim_e(\mu)\le h_{\nu_i}(T_i)$, then for  ${\Pi_i}\mu$-a.e.~$z\in \Pi_i(K)$, the measure $\mu^z$ is exact dimensional with Hausdorff dimension equal to 0.

\item If  $\dim_e(\mu)>h_{\nu_i}(T_i)$,  then for  ${\Pi_i}\mu$-a.e.~$z\in \Pi_i(K)$, the measure $\mu^z$ is exact dimensional with
\begin{equation}\label{dimfib}
\dim (\mu^z)=\gamma_1(\dim_e(\mu)-h_{\nu_i}(T_i))+\sum_{j=2}^{i-1}\gamma_j\big (\min(\dim_e(\mu),h_{\nu_j}(T_j))-h_{\nu_i}(T_i)\big).
\end{equation}
\item  In both the previous situations, the Hausdorff dimension conservation $$\dim(\mu)=\dim(\mu^z)+\dim({\Pi_i}\mu)$$ holds.
\end{enumerate}
\end{thm}

Naturally, there is a similar result for the conditional measures of ${\Pi_i}\mu$ projected on $X_j$, $2\le i < j\le k$.

\begin{thm}\label{thmfibers2} Suppose $k\ge 3$. Let $\mu$ be a non-degenerate Mandelbrot measure supported on $K$ and $2\le i<j\le k$. Suppose that $T(q)>-\infty$ for some $q>1$. With probability 1, conditional on $\mu\neq 0$, denote by $({\Pi_{i}}\mu)^{j,z}$ the conditional measure of ${\Pi_i}\mu$ associated with the projection $\Pi_{i,j}$, and defined for ${\Pi_{j}}\mu$-almost every $z$.
\begin{enumerate}
\item If $\dim_e(\mu)\le h_{\nu_j}(T_j)$, then for  ${\Pi_{j}}\mu$-a.e.~$z\in \Pi_j(K)$, the measure $({\Pi_{i}}\mu)^{j,z}$ is exact dimensional with Hausdorff dimension equal to 0.

\item If  $\dim_e(\mu)>h_{\nu_j}(T_j)$,  then for  ${\Pi_{j}}\mu$-a.e.~$z\in \Pi_j(K)$, the measure $({\Pi_{i}}\mu)^{j,z}$ is exact dimensional with
\begin{align}\label{dimfib2}
\dim (({\Pi_{i}}\mu)^{j,z}))=\sum_{j'=i}^{j-1}\vec{\boldsymbol\gamma}^i_{j'}\big (\min(\dim_e(\mu),h_{\nu_{j'}}(T_{j'}))-h_{\nu_j}(T_j))\big).
\end{align}
\item  In both the previous situations, the Hausdorff dimension conservation $$\dim({\Pi_{i}}\mu)=\dim(({\Pi_{i}}\mu)^{j,z}))+\dim ({\Pi_{j}}\mu)$$ holds for  $\Pi_{i,j}$-a.e.~$z\in \Pi_{i,j}(K)$.
\end{enumerate}
\end{thm}

\subsection{Applications to the Euclidean realisations of symbolic random statistically self-affine Sierpinski sponges}
\label{euclidean}
The link with Euclidean random sponges is the following: Given a sequence of integers $2\le m_k<  \cdots <  m_1$, if $\mathcal A_i=\prod_{j=i}^k\{0,\ldots,m_{i}-1\}$ for $1\le i\le k$, $\pi_{i}$ is the canonical projection from $\mathcal A_i$ to $\mathcal A_{i+1}$ for $1\le i\le k-1$,  $\gamma_1=1/\log(m_1)$, and $\gamma_i=1/\log(m_{i})-1/\log(m_{i-1})$, $2\le i\le k$, then the cylinders of generation $n$ of $X_1$ project naturally onto parallelepipeds of the family
$$\mathcal G_n=\left\{\prod_{i=1}^k [\ell_im_i^{-n},(\ell_i+1)m_i^{-n}]: \; 0\le \ell_i\le m_i^n-1\right\},$$ and $K$ projects on a  statistically self-affine Sierpinski sponge $\widetilde K$, also called Mandelbrot percolation set associated with $(A(u))_{u\in\mathcal A_1^*}$ in the cube $[0,1]^k$ endowed with the nested grids $(\mathcal G_n)_{n\ge 0}$.

It is direct to prove that all the results of the previous sections are valid if one replaces $K$ by $\widetilde K$, the  Mandelbrot measures by their natural projections on $\widetilde K$ (which are also called Mandelbrot measures), and $\Pi_i$ by the orthogonal projection from $\mathbb R^k$ to $\{0\}^{i-1}\times \mathbb R^{k-i+1}$.

If $\widetilde K$ is deterministic, then $i_0=2$, $\theta_2=\gamma_1/(\gamma_1+\gamma_2)$, the Mandelbrot measure of maximal Hausdorff dimension is a Bernoulli product measures, and we recover the result established by Kenyon and Peres in \cite{KP} (they work on $(\mathbb R/\mathbb Z)^k$ but it is equivalent); also, in this case the results on the dimension of conditional measures is a special case of the general result obtained by the second author on the dimension theory of self-affine measures \cite{Feng20..}. If $k=2$, the Euclidean version of Theorem~\ref{thmdim} yields the value of $\dim_H \widetilde K$ computed by Gatzouras and Lalley in \cite{GL94}.

Regarding the box counting dimension of $\widetilde K$, if $\widetilde K$ is deterministic, we just recover the result of \cite{KP}; in this case,  $\widehat \theta_i=0$ for all $2\le i\le k$. If $k=2$, we recover the result of Gatzouras and Lalley in \cite{GL94}.

\medskip

The paper is organized as follows. Section 3 is dedicated to the proof of Theorem~\ref{thmdim}, Section 4 to the proof of Theorem~\ref{dim K}, Section 5 to that of Theorem~\ref{thmdimB} and its corollary, Section 6 to those of the corresponding results for projections of Mandelbrot measures and~$K$, Section~\ref{pfcondmeas} to those on conditional measures, and the brief Section~\ref{lastsection} to the case when some $\gamma_i$ vanish.

\section{The Hausdorff dimension of $\mu$. Proof of Theorem~\ref{thmdim}}\label{sec3}

Let us start with a few definitions.

With the notation given in the introduction part, for any word $I\in\mathcal A_1^*$ and any integer $n\ge 0$, we denote by $\mu^I$ the measure defined on $X_1$ by
$$
\mu^I([J])=Q^I(J) Y(IJ) \quad \mbox{for } J\in \mathcal A_1^*
$$
and by $\mu^I_{n}$ the measure on $X_1$ obtained by distributing uniformly $Q^I(J)$ on any cylinder $J$ of the $n^{\rm th}$ generation. Also, we write $\mu_n=\mu_n^\epsilon$.

For $1\le i\le k$ and $n\in \N$,  let
$$
\ell_i(n)=\min \left\{p\in\N: \; p\ge (\gamma_1+\cdots+\gamma_i)\frac{n}{\gamma_1}\right\},
$$
and by convention set $\ell_0(n)=0$. It is easy to check that in the ultrametric space  $(X_1,d_{\vec{\boldsymbol\gamma}})$, the closed ball centered at $x$ of radius $e^{-\frac{n}{\gamma_1}}$ is given by
\begin{equation*}\label{Bxr}
B(x,e^{-\frac{n}{\gamma_1}})= \left\{y\in X_1: \; \Pi_i(y_{|\ell_i(n)})=\Pi_i({x}_{|\ell_i(n)}) \mbox{ for all } 1\leq i\leq k\right\}.
\end{equation*}
Let $\mathcal F_n$ be the partition of $X_1$ into closed balls of radius $e^{-\frac{n}{\gamma_1}}$. For any  finite Borel measure~$\nu$ on $X_1$,  the $L^q$-spectrum of $\nu$ can be defined as the concave mapping
$$
\tau_\nu:\; q\in\R\mapsto \liminf_{n\to\infty} -\frac{\gamma_1}{n}\log\sum_{B\in\mathcal F_n} \nu(B)^q,
$$
with the convention $0^q=0$.

It is known that since $(X_1,d_{\vec{\boldsymbol\gamma}})$ satisfies the Besicovich covering property, for $\nu$-almost every $x\in X_1$, one has $$\tau_\nu'(1^+)\le \underline{\dim}_{\rm{loc}}(\nu,x)\le \overline{\dim}_{\rm{loc}}(\nu,x) \le \tau_\nu'(1^-),$$ so that the existence of $\tau_\nu'(1)$ implies the exact dimensionality of $\nu$, with dimension equal to~$\tau_\nu'(1)$ (see, e.\,g., \cite{Ngai}).
Consequently, Theorem~\ref{thmdim}  follows from the following stronger result.

\begin{thm}
\label{thm-3.1}
Suppose that $T(q)>-\infty$ for some $q>1$. Conditional on $\mu\neq 0$, $\tau_\mu'(1)$ exists and equals $\dim^{\vec{\boldsymbol\gamma}}_e(\mu)$, where
$\dim^{\vec{\boldsymbol\gamma}}_e(\mu)$ is defined as in \eqref{e-LY}.
\end{thm}

Recall that for $2\le i\le k$, we defined $\nu_i$ as ${\Pi_i}\mathbb{E}(\mu)= \mathbb{E}({\Pi_i}\mu)$.
If $\nu$ is a Bernoulli product measure on $X_i$, we set
\begin{equation}
\label{e-cal-T}
\mathcal T_{\nu}(q)=-\log\sum_{b\in\mathcal A_i}\nu([b])^q\quad (q\ge 0).
\end{equation}
Then Theorem \ref{thm-3.1} follows from the following proposition.
\begin{pro}\label{pro1}
Suppose that $T(q)>-\infty$ for some $q>1$. Let $$i_0=\max\{2\le i\le k:\; T'(1)\le \mathcal T'_{\nu_i}(1)\},$$
 with the convention $\max (\emptyset)=1$. Then there exists $q_0>1$ and $c_0\ge 0$ such that for all $q\in (0,q_0]$, we have
\begin{equation}\label{Lq}
\mathbb{E}\Big (\sum_{B\in\mathcal F_n} \mu(B)^q\Big )=O\Big (\exp \Big ( \ell_{i_0}(n)(c_0(q-1)^2- T(q))-\sum_{i=i_0+1}^k (\ell_{i}(n)-\ell_{i-1}(n))\mathcal T_{\nu_i}(q)\Big )\Big )
\end{equation}
as $n\to\infty$, here we use the Landau big O notation. Moreover, $c_0$ can be taken equal to 0 if one restricts $q$ to belong to $(0,1]$ or if $T'(1)\neq \mathcal T'_{\nu_i}(1)$ for all $2\le i\le k$.
\end{pro}
Assume that Proposition~\ref{pro1} holds. Then, a standard argument (see, e.g. \cite[Lemma C]{BF2016}) yields that for any fixed $q\in (0,q_0]$, the following holds almost surely:
\begin{align*}
 \limsup_{n\to\infty}\frac{\displaystyle\log \sum_{B\in\mathcal F_n} \mu(B)^q}{n}&\le \limsup_{n\to\infty}\frac{\ell_{i_0}(n)(c_0(q-1)^2- T(q))-\sum_{i=i_0+1}^k (\ell_{i}(n)-\ell_{i-1}(n))\mathcal T_{\nu_i}(q)}{n}\\
&=\frac{\gamma_1+\cdots+\gamma_{i_0}}{\gamma_1} \,(c_0(q-1)^2- T(q))-\sum_{i=i_0+1}^k \frac{\gamma_i}{\gamma_1}\,  \mathcal T_{\nu_i}(q).
\end{align*}
Then, by the convexity of the two sides as functions of $q$, the inequality holds almost surely for all $q\in (0,q_0]$. Multiplying both sides by $-\gamma_1$  yields, conditional on $\mu\neq 0$,
$$
\tau_\mu(q)\ge -(\gamma_1+\cdots+\gamma_{i_0})c_0(q-1)^2+(\gamma_1+\cdots+\gamma_{i_0}) \, T(q)+\sum_{i=i_0+1}^k \gamma_i\,  \mathcal T_{\nu_i}(q).
$$
Since both sides of the above inequality are concave functions which coincide at $q=1$ and the right hand side is differentiable at 1, we necessarily have that $\tau'_\mu(1)$ does exist and is equal to the derivative at $1$ of the right hand side, namely $$(\gamma_1+\cdots+\gamma_{i_0}) \, T'(1)+\sum_{i=i_0+1}^k \gamma_i\,  \mathcal T'_{\nu_i}(1)=\dim^{\vec{\boldsymbol\gamma}}_e(\mu).$$
Hence Proposition \ref{pro1} implies Theorem \ref{thm-3.1}.
\medskip

In the remaining part of this section, we prove  Proposition~\ref{pro1}.  We still need  two additional lemmas.

\begin{lem}\label{lemclef}Suppose that $T(q)>-\infty$ for some $q>1$. Then, for all  $q\in (1,2)$ such that $T(q)>0$, there exists a constant $C_q>0$ such that  for all $2\le i\le k$, for all  $n\ge 1$ one has
$$
\max \left(\mathbb E\sum_{U\in\mathcal A_i^n} {\Pi_i}\mu([U])^q, \; \mathbb E \sum_{U\in\mathcal A_i^n} {\Pi_i}\mu_n([U])^q\right )\le C_q e^{-n\min (T(q), \mathcal T_{\nu_i}(q))}.
$$
\end{lem}

\begin{proof}
This is a direct consequence of \cite[Corollary 5.2]{BF2016}, in which the case $k=2$ is considered.
\end{proof}

\begin{lem}\cite{vB-E}\label{martdiff}
 Let $(L_j)_{j\geq 1}$ be a sequence of centered independent  real valued  random variables. For every finite  $I\subset\N$ and $q\in (1, 2]$,
$$\E\Big (\Big |\displaystyle\sum_{i\in I} L_{i}\Big|^q\Big ) \leq 2  \sum_{i\in I}  \E( \left|L_{i}\right|^q).$$
\end{lem}

\begin{proof}[Proof of Proposition~\ref{pro1}]

At first, note that the set of balls $\mathcal F_n$ is in bijection with the set $\prod_{i=1}^k \mathcal A_i^{\ell_i(n)-\ell_{i-1}(n)}$, since for any $x=(x_i)_{i=1}^\infty\in X_1$, if we set   $U_i=\Pi_i(x_{\ell_{i-1}(n)+1}\cdots x_{\ell_{i}(n)})$, $1\le i\le k$, and $U=(U_1,\ldots, U_k)$, then
\begin{eqnarray*}
B(x,e^{-\frac{n}{\gamma_1}})&=&\left\{y\in X_1: \;  \Pi_i(T_1^{\ell_{i-1}(n)}y)\in [U_i] \mbox{ for all } 1\le i\le k\right\}\\
&=&\bigcup_{(J_1,\ldots,J_k)\in \mathcal J_U}[J_1\cdots J_k],
\end{eqnarray*}
where
\begin{equation}\label{decomp}
\mathcal J_U:= \left\{(J_1,\ldots,J_k)\in \prod_{i=1}^{k}\A_1^{\ell_i(n)-\ell_{i-1}(n)}: \;  \Pi_i(J_i)=U_i \mbox{ for all }  1\le i\le k\right\}.
\end{equation}
For $1\le i\le k$,  set $\mathcal U^{(i)}_n= \prod_{j=i}^{k}\A_j^{\ell_j(n)-\ell_{j-1}(n)}$.  For $q\in \R_+$ we need to estimate from above the partition function
$$
Z_{q,n}:=\sum_{B\in\mathcal F_n} \mu(B)^q=\sum_{ U\in\mathcal U_{n}^{(1)}}\Big(\sum_{(J_1,J_2,\ldots,J_k)\in \mathcal J_U}\mu([J_1\cdots J_k])\Big )^q.
$$
For $1\le i\le k$ and $U^{(i)}=(U_{i},\cdots, U_k)\in \mathcal U^{(i)}_n$,  set
$$
\mathcal J_{U^{(i)}}=\left\{(J_{i},\ldots,J_k)\in \prod_{j=i}^{k}\A_1^{\ell_j(n)-\ell_{j-1}(n)}: \; \Pi_j(J_j)=U_j \mbox{ for all }i\le j\le k\right\}.
$$
Also, set $\mathcal U^{(k+1)}_n=\{\epsilon\}$ and $\mathcal J_{\epsilon}=\{\epsilon\}$.

Then, for $1\le i\le k$ and $(J_1,\ldots, J_i)\in \prod_{j=1}^i \mathcal A_1^{\ell_i(n)-\ell_{i-1}(n)}$,  define the random variable
$$
Z_{q,n}(J_1\cdots J_i)=\sum_{U^{(i+1)}\in  \mathcal U^{(i+1)}_n}\left(\sum_{(J_{i+1},\ldots,J_k)\in \mathcal J_{U^{(i+1)}}} \mu^{J_1\cdots J_{i}} ([J_{i+1}\ldots J_k])\right)^q.
$$
Notice that $Z_{q,n}(\epsilon)=Z_{q,n}$ and  $Z_{q,n}(J_1\cdots J_k)= Y(J_1\cdots J_k)^q$.

Due to the branching property associated with the measures $\mu^J$, $J\in\mathcal A_1^*$, for all $0\le i\le k-1$ we have
$$
 Z_{q,n}(J_1\cdots J_i)= \sum_{U^{(i+1)}\in  \mathcal U^{(i+1)}_n}\left(\sum_{\substack{J_{i+1}\in \mathcal A_{1}^{\ell_{i+1}(n)-\ell_{i}(n)} \\ \Pi_{i+1}(J_{i+1})=U_{i+1}}} \mu^{J_1\cdots J_i}_{\ell_{i+1}(n)-\ell_{i}(n)}([J_{i+1}]) \, S_{U^{(i+2)}}(J_1\cdots J_iJ_{i+1})\right)^q,
 $$
where $U^{(i+2)}=(U_{i+2},\ldots,U_k)\in \mathcal{U}^{(i+2)}_n$, and
$$
S_{U^{(i+2)}}(J_1\cdots J_iJ_{i+1})=\sum_{(J_{i+2},\ldots,J_k)\in \mathcal J_{U^{(i+2)}}} \mu^{J_1\cdots J_iJ_{i+1}} ([J_{i+2}\ldots J_k]).
$$
Notice that the random variables $S_{U^{(i+2)}}(J_1\cdots J_iJ_{i+1})$, where $J_{i+1}\in \mathcal A_{1}^{\ell_{i+1}(n)-\ell_{i}(n)}$ and $\Pi_i(J_{i+1})=U_{i+1}$, are independent and identically distributed, and independent of the $\sigma$-algebra generated by  the $\mu^{J_1\cdots J_i}_{\ell_{i+1}(n)-\ell_{i}(n)}(J_{i+1})$. Setting $$L(J_{i+1})=S_{U^{(i+2)}}(J_1\cdots J_{i+1})-\mathbb E(S_{U^{(i+2)}}),$$
 where $\mathbb E(S_{U^{(i+2)}})$  stands for the common value of the  $S_{U^{(i+2)}}(J_1\cdots J_{i+1})$ expectations, we have, for $q>1$:
\begin{align*}
&\mathbb{E}\left (\sum_{\substack{J_{i+1}\in \mathcal A_{i+1}^{\ell_{i+1}(n)-\ell_{i}(n)} \\ \Pi_{i+1}(J_{i+1})=U_{i+1}}} \mu^{J_1\cdots J_i}_{\ell_{i+1}(n)-\ell_{i}(n)}([J_{i+1}]) \, S_{U^{(i+2)}}(J_1\cdots J_iJ_{i+1})\right)^q\\
&\le 2^{q-1} \mathbb{E}\left (\sum_{\substack{J_{i+1}\in \mathcal A_{i+1}^{\ell_{i+1}(n)-\ell_{i}(n)} \\ \Pi_{i+1}(J_{i+1})=U_{i+1}}} \mu^{J_1\cdots J_i}_{\ell_{i+1}(n)-\ell_{i}(n)}([J_{i+1}]) \right)^q \, \mathbb E(S_{U^{(i+2)}})^q\\
&\quad\quad +2^{q-1}\mathbb{E}\left |\sum_{\substack{J_{i+1}\in \mathcal A_{i+1}^{\ell_{i+1}(n)-\ell_{i}(n)} \\ \Pi_{i+1}(J_{i+1})=U_{i+1}}} \mu^{J_1\cdots J_i}_{\ell_{i+1}(n)-\ell_{i}(n)}([J_{i+1}]) \, L(J_{i+1})\right|^q.
\end{align*}
Assuming that $q\in (1,2]$, we can apply Lemma~\ref{martdiff} to the second term conditional on the $\sigma$-algebra generated by  the $\mu^{J_1\cdots J_i}_{\ell_{i+1}(n)-\ell_{i}(n)}(J_{i+1})$ and get
\begin{align*}
&\mathbb{E}\left |\sum_{\substack{J_{i+1}\in \mathcal A_{i+1}^{\ell_{i+1}(n)-\ell_{i}(n)} \\ \Pi_{i+1}(J_{i+1})=U_{i+1}}} \mu^{J_1\cdots J_i}_{\ell_{i+1}(n)-\ell_{i}(n)}([J_{i+1}]) \, L(J_{i+1})\right|^q\\
&\le 2^{q} \mathbb{E}\left( \sum_{\substack{J_{i+1}\in \mathcal A_{i+1}^{\ell_{i+1}(n)-\ell_{i}(n)} \\ \Pi_{i+1}(J_{i+1})=U_{i+1}}} \mu^{J_1\cdots J_i}_{\ell_{i+1}(n)-\ell_{i}(n)}([J_{i+1}])^q\right )\, \mathbb{E} (|L|^q)\\
&\le 2^{q} \mathbb{E}\left( \sum_{\substack{J_{i+1}\in \mathcal A_{i+1}^{\ell_{i+1}(n)-\ell_{i}(n)} \\ \Pi_{i+1}(J_{i+1})=U_{i+1}}} \mu^{J_1\cdots J_i}_{\ell_{i+1}(n)-\ell_{i}(n)}([J_{i+1}])\right )^q\, \mathbb{E} (|L|^q) \quad \text{(using superadditivity)}\\
&=2^{q} \mathbb{E}\left ({\Pi_{i+1}}\mu^{J_1\cdots J_i}_{\ell_{i+1}(n)-\ell_{i}(n)}([U_{i+1}])^q\right )\, \mathbb{E} (|L|^q),
\end{align*}
where $\mathbb{E} (|L|^q)=\mathbb{E} (|S_{U^{(i+2)}}-\mathbb E(S_{U^{(i+2)}})|^q)\le 2^q \mathbb{E} (S_{U^{(i+2)}}^q)$, and  $\mathbb{E} (S_{U^{(i+2)}}^q)$ is the common value of the  $\mathbb E(S_{U^{(i+2)}}(J_1\cdots J_{i+1})^q)$.  Incorporating the last inequality in the previous one, we get
\begin{align*}
&\mathbb{E}\left (\sum_{\substack{J_{i+1}\in \mathcal A_{i+1}^{\ell_{i+1}(n)-\ell_{i}(n)} \\ \Pi_{i+1}(J_{i+1})=U_{i+1}}} \mu^{J_1\cdots J_i}_{\ell_{i+1}(n)-\ell_{i}(n)}(J_{i+1}) \, S_{U^{(i+2)}}(J_1\cdots J_{i+1})\right)^q\\
&\le 2^{3q} \mathbb{E}\left ({\Pi_{i+1}}\mu^{J_1\cdots J_i}_{\ell_{i+1}(n)-\ell_{i}(n)}([U_{i+1}])^q\right )\mathbb{E}\left(S_{U^{(i+2)}}^q\right).
\end{align*}
Then, taking an arbitrary element $\widetilde J_{i+1}$ in $\mathcal A_1^{\ell_{i+1}(n)-\ell_i(n)}$, we obtain
\begin{align*}
&\mathbb{E}(Z_{q,n}(J_1\cdots J_i))\\
&\le 2^{3q}  \sum_{U_{i+1}\in \mathcal A_{i+1}^{\ell_{i+1}(n)-\ell_i(n)}} \mathbb{E}\big ({\Pi_{i+1}}\mu^{J_1\cdots J_i}_{\ell_{i+1}(n)-\ell_{i}(n)}([U_{i+1}])^q\big ) \sum_{U^{(i+2)}\in  \mathcal U^{(i+2)}_n}\, \mathbb E(S_{U^{(i+2)}}(J_1\cdots J_i\widetilde J_{i+1})^q)\\
&=2^{3q}  \sum_{U_{i+1}\in \mathcal A_{i+1}^{\ell_{i+1}(n)-\ell_i(n)}} \mathbb{E}\big ({\Pi_{i+1}}\mu^{J_1\cdots J_i}_{\ell_{i+1}(n)-\ell_{i}(n)}([U_{i+1}])^q\big )\,  \mathbb{E}(Z_{q,n}(J_1\cdots J_i \widetilde J_{i+1})).
\end{align*}
Since $(\mu_p)_{p\ge 1}$ and $(\mu^{J_1\cdots J_i}_p)_{p\ge 1}$ are identically distributed this yields
\begin{align*}
&
\mathbb{E}(Z_{q,n}(J_1\cdots J_i))\\
&\le 2^{3q}  \sum_{U_{i+1}\in \mathcal A_{i+1}^{\ell_{i+1}(n)-\ell_i(n)}} \mathbb{E}\big ({\Pi_{i+1}}\mu_{\ell_{i+1}(n)-\ell_{i}(n)}([U_{i+1}])^q\big )\,  \mathbb{E}(Z_{q,n}(J_1\cdots J_i \widetilde J_{i+1}))
.
\end{align*}
It follows that
$$
\mathbb{E}(Z_{q,n})\le 2^{3qk} \, \mathbb E(Y^q) \prod_{i=1}^k \mathbb{E} \left (\sum_{U_i\in\mathcal A_i^{\ell_{i}(n)-\ell_{i-1}(n)}}{\Pi_i}\mu_{\ell_{i}(n)-\ell_{i-1}(n)}([U_i])^q\right ).
$$
Let $q_1\in (1,2]$ such that $T(q)>0$  for all $q\in (1,q_1]$ (remember that $T(1)=0$ and $T'(1)>0$). Then, for all $q\in (1,q_1]$, the previous estimate combined with  Lemma~\ref{lemclef} yields
\begin{align*}
\mathbb{E}(Z_{q,n})&\le 2^{3qk} C_q^{k-1} \mathbb{E}(Y^q)\,\mathbb{E} \Big (\sum_{U_1\in\mathcal A_1^{n}}\mu_{n}([U_1])^q\Big )\\
&\qquad\times \ \exp\Big( -\sum_{i=2}^k (\ell_i(n)-\ell_{i-1}(n))\min (T(q),\mathcal T_{\nu_i}(q))\Big )\\
&=2^{3qk} C_q^{k-1} \mathbb{E}(Y^q)\,\exp\Big( -n T(q)-\sum_{i=2}^k (\ell_i(n)-\ell_{i-1}(n))\min (T(q),\mathcal T_{\nu_i}(q))\Big ).
\end{align*}

Finally, recall that $i_0=\max\{2\le i\le k: T'(1)\le \mathcal T_{\nu_i}'(1)\}$ (note that for each $i$ the number $\mathcal T_{\nu_i}'(1)$ is the measure-theoretic entropy of $\nu_i$ so that the sequence $(\mathcal T_{\nu_i}'(1))_{1\le i\le k}$ is non-increasing). Since $T$ and the functions $\mathcal T_{\nu_i}$ are analytic near 1 and coincide at $1$, for all $2\le i\le i_0$ there exist $q_{0,i}\in (1,q_1]$ and $c_i\ge 0$ such that for all $q\in (1,q_{0,i}]$ one has $$\min (T(q),\mathcal T_{\nu_i}(q))\ge T(q)-c_i(q-1)^2,$$ with $c_i=0$ if $T'(1)< \mathcal T_{\nu_i}'(1)$. Taking $c_0=\max\{c_i:\; 2\le i\le i_0\}$ and $q_0=\min\{q_{0,i}:\; 2\le i\le i_0\}$ yields~\eqref{Lq}.

\medskip

Suppose now that $q\in (0,1]$. We start with giving general estimate of $\mathbb{E}(Z_{q,n}(J_1\cdots J_i))$. Using the subadditivity of $x\in\mathbb R_+\mapsto x^q$ we have
$$
 Z_{q,n}(J_1\cdots J_i)\le \sum_{U^{(i+1)}\in  \mathcal U^{(i+1)}_n}\sum_{\substack{J_{i+1}\in \mathcal A_{1}^{\ell_{i+1}(n)-\ell_{i}(n)} \\ \Pi_{i+1}(J_{i+1})=U_{i+1}}} \mu^{J_1\cdots J_i}_{\ell_{i+1}(n)-\ell_{i}(n)}([J_{i+1}])^q \, S_{U^{(i+2)}}(J_1\cdots J_iJ_{i+1})^q,
 $$
so
\begin{align*}
\mathbb{E}(Z_{q,n}(J_1\cdots J_i))&\le \sum_{U_{i+1}\in \mathcal A_{i+1}^{\ell_{i+1}(n)-\ell_i(n)}}\sum_{\substack{J_{i+1}\in \mathcal A_{1}^{\ell_{i+1}(n)-\ell_{i}(n)} \\ \Pi_{i+1}(J_{i+1})=U_{i+1}}}\mathbb{E} (\mu^{J_1\cdots J_i}_{\ell_{i+1}(n)-\ell_{i}(n)}(J_{i+1})^q) \\
&\qquad\qquad\qquad\qquad \cdot \mathbb{E}\Big (\sum_{U^{(i+2)}\in  \mathcal U^{(i+2)}_n}\, S_{U^{(i+2)}}(J_1\cdots J_i\widetilde J_{i+1})^q\Big )\\
&=\mathbb{E}\Big (\sum_{J_{i+1}\in \mathcal A_{1}^{\ell_{i+1}(n)-\ell_{i}(n)}}\mu^{J_1\cdots J_i}_{\ell_{i+1}(n)-\ell_{i}(n)} (J_{i+1})^q\Big ) \mathbb{E}(Z_{q,n}(J_1\cdots J_i \widetilde J_{i+1}))\\
&=\mathbb{E}\Big (\sum_{J_{i+1}\in \mathcal A_{1}^{\ell_{i+1}(n)-\ell_{i}(n)}}\mu_{\ell_{i+1}(n)-\ell_{i}(n)} (J_{i+1})^q\Big ) \mathbb{E}(Z_{q,n}(J_1\cdots J_i \widetilde J_{i+1}))\\
&=\exp(-(\ell_{i+1}(n)-\ell_{i}(n)) T(q)) \mathbb{E}(Z_{q,n}(J_1\cdots J_i \widetilde J_{i+1})).
\end{align*}

Starting from $\mathbb{E}(Z_{q,n})=\mathbb{E}(Z_{q,n}(\epsilon))$ and iterating $i_0$ times the previous estimate we get
\begin{align*}
\mathbb{E}(Z_{q,n})&\le \Big (\prod_{i=1}^{i_0} \exp(-(\ell_{i}(n)-\ell_{i-1}(n)) T(q))\Big )\,  \mathbb{E}(Z_{q,n}(\widetilde J_1\cdots\widetilde J_{i_0}))\\
&=\exp (-\ell_{i_0}(n) T(q) )\,  \mathbb{E}(Z_{q,n}(\widetilde J_1\cdots\widetilde J_{i_0})).
\end{align*}

On the other hand, setting $\widetilde J=\widetilde J_1\cdots\widetilde J_{i_0}$ and $\lambda(n)=\ell_{k-1}(n)-\ell_{i_0+1}(n)$, we can write
\begin{align*}
&Z_{q,n}(\widetilde J)\\
&= \sum_{U^{(i_0+1)}\in  \mathcal U^{(i_0+1)}_n}\left (\sum_{\substack{J'=J_{i_0+1}\cdots,J_{k-1}\\ \Pi_j(J_{j})=U_{j}, \ \forall\,  i_0+1\le j\le k-1}} \mu_{\lambda(n)}^{\widetilde J}([J']) \sum_{J_{k}:\;  \Pi_{k}(J_{k})=U_{k}}\mu^{\widetilde J J'}([J_k])\right )^q\\
&=\sum_{U^{(i_0+1)}\in  \mathcal U^{(i_0+1)}_n}\left (\sum_{\substack{J'=J_{i_0+1}\cdots J_{k-1}\\ \Pi_j(J_{j})=U_{j}, \ \forall\,  i_0+1\le j\le k-1}} \mu_{\lambda(n)}^{\widetilde J}([J'])\, \nu_k([U_k])\, X(\widetilde JJ') \right )^q,
\end{align*}
where
$$
X(\widetilde JJ')=\sum_{J_{k}:\;  \Pi_{k}(J_{k})=U_{k}}w^{\widetilde JJ'}(J_k), \text{ with }w^{\widetilde JJ'}(J_k)=\begin{cases}
\displaystyle \frac{ \mu^{\widetilde JJ'}([J_k])}{\nu_{k}([U_{k}])}&\text{if }\nu_{k}([U_{k}])>0,\\
0&\text{otherwise}
\end{cases}.
$$
We can now use the independence of the random variables $X(\widetilde JJ')$ with respect to the $\sigma$-algebra generated by the $\mu_{\lambda(n)}^{\widetilde J}([J'])$, conditioned with respect to this $\sigma$-algebra and use Jensen's inequality to get
$$
\mathbb{E} (Z_{q,n}(\widetilde J))\le \sum_{U^{(i_0+1)}\in  \mathcal U^{(i_0+1)}_n}\mathbb{E}\left (\sum_{\substack{J'=J_{i_0+1},\ldots, J_{k-1}\\ \Pi_j(J_{j})=U_{j}, \ \forall\,  i_0+1\le j\le k-1}} \mu_{\lambda(n)}^{\widetilde J}([J'])\, \nu_k([U_k])\, \mathbb{E}(X(\widetilde JJ'))\right )^q.
$$
But by construction we have $\nu_k([U_k])=\mathbb{E}\Big (\sum_{J_k:\; \Pi_{k}(J_k)=U_k} \mu^{\widetilde J J'}([J_k])\Big )$, hence $\mathbb{E}(X(\widetilde JJ'))=1$. Setting
$$
R=\sum_{(U_{i_0+1},\ldots,U_{k-1})\in \prod_{i=i_0+1}^{k-1}\mathcal A_i^{\ell_i(n)-\ell_{i-1}(n)}}\mathbb{E}\left (\sum_{\substack{J'=J_{i_0+1}\cdots J_{k-1}\\ \Pi_j(J_{j})=U_{j}, \ \forall\,  i_0+1\le j\le k-1}} \mu_{\lambda(n)}^{\widetilde J}([J'])\right )^q,
$$
this yields
$$
\mathbb{E} (Z_{q,n}(\widetilde J))\le R\cdot \Big (\sum_{U_k\in \mathcal A_k^{\ell_k(n)-\ell_{k-1}(n)}}\nu_k([U_k])^q\Big )=R\cdot \exp (-(\ell_k(n)-\ell_{k-1}(n))\mathcal T_{\nu_k}(q)).
$$

We can apply to $R$ the same type of estimate as that for $\mathbb{E} (Z_{q,n}(\widetilde J))$, the only change being that $\mu^{\widetilde JJ'}([J_k])=\mu^{\widetilde JJ_{i_0+1}\cdots J_{k-1}}([J_k])$ must be replaced by $\mu^{\widetilde JJ_{i_0+1}\cdots J_{k-2}}_{\ell_{k-1}(n)-\ell_{k-2}(n)}(J_{k-1})$, and one now must use the fact that $\nu_{k-1}=\mathbb{E}\left( {\Pi_{k-1}}\mu^{\widetilde JJ_{i_0+1}\cdots J_{k-2}}_{\ell_{k-1}(n)-\ell_{k-2}(n)}\right)$. Iterating we get
$$
\mathbb{E} (Z_{q,n}(\widetilde J))\le \exp \Big (-\sum_{i=i_0+1}^k (\ell_{i}(n)-\ell_{i-1}(n))\mathcal T_{\nu_i}(q)\Big ),
$$
and finally
$$
\mathbb{E}(Z_{q,n})\le \exp\Big  (-\ell_{i_0}(n) T(q) -\sum_{i=i_0+1}^k (\ell_{i}(n)-\ell_{i-1}(n))\mathcal T_{\nu_i}(q)\Big ).
$$
This completes the proof of the Proposition.
\end{proof}

\section{The Hausdorff dimension of $K$. Proof of Theorem~\ref{dim K}}

We have to  optimise the weighted entropy $\dim^{\vec{\boldsymbol\gamma}}_e(\mu)$ over the set of non-degenerate Mandelbrot measures $\mu$ supported on $K$; this will provide us with a sharp lower bound for $\dim_H K$. To do so, it is convenient to first relate $\dim_e(\mu)$ to $h_{\nu_i}(T_i)$ for all $1\le i\le k$, where $\nu_i=\mathbb{E}(\Pi\mu)$. This is the purpose of Section~\ref{skewed}. Then we identify  at which point the maximum of weighted entropy dimension  of Mandelbrot measures supported on $K$ is reached. This constitutes Section~\ref{OPT}. Section~\ref{LB} quickly derives the sharp lower bound for $\dim_H K$. Finally, in Section~\ref{UB} we develop a kind of variational principle to get the sharp upper bound for $\dim_H K$.

\subsection{Mandelbrot measure as a kind of skewed product and decomposition of entropy dimension}\label{skewed} Let $\mu$ be a non-degenerate Mandelbrot measure jointly constructed with $K$ and such that $K_\mu\subset K$ almost surely. As in Section~\ref{Mandmeas}, we denote by $W$ the random vector used to generate $\mu$. By construction, for any $1\le i\le k$, the measure $\nu_i=\mathbb E({\Pi_i}\mu)$ is the Bernoulli product measure on $X_i$ associated with the probability vector $p^{(i)}=(p^{(i)}_b)_{b\in \mathcal A_i}$, where
$$
p^{(i)}_b=\sum_{a\in\mathcal A_1: \; [a]\subset \Pi_i^{-1}([b])} \mathbb{E}(W_a),
$$
and one has $\nu_i={\pi_{i-1}}\nu_{i-1}$ for $i\ge2$.   We also define, for $b\in\mathcal A_i$,
$$
V_b^{(i)}=(V_{b,a}^{(i)})_{a\in\mathcal A_1:\; [a]\subset  \Pi_i^{-1}([b]) }=
\begin{cases}\nu_i([b])^{-1}(W_a)_{a\in\mathcal A_1:\; [a]\subset \Pi_i^{-1}([b]) } &\text{ if $\nu_i([b])>0$}\\
0&\text{ otherwise}
\end{cases},
$$
so that for all $a\in \mathcal A_1$, for all $1\le i\le k$,  we have  the multiplicative decomposition
$$
W_a=\nu_i(\Pi_i[a]) \cdot V^{(i)}_{\Pi_i(a),a}.
$$
For $1\le i\le k$ and $b\in \mathcal A_i$, set
\begin{equation}\label{TVi}
T_{V^{(i)}_b}(q)= -\log  \mathbb{E}\left( \sum_{a\in\mathcal A_1:\; [a]\subset \Pi_i^{-1}([b]) } (V_{b,a}^{(i)})^q\right) \quad (q\ge 0),
\end{equation}
with the conventions $0^0=0$ and $\log (0)=-\infty$. In particular
\begin{equation}\label{TVi0}
T_{V^{(i)}_b}(0)=-\log \mathbb E(N^{(i)}_b(W)),
\end{equation}
where
$$
 N^{(i)}_b(W)=\#\{a\in\mathcal A_1:\; [a]\subset  \Pi_i^{-1}([b]),\; W_a>0\}.
$$
One can check that
\begin{equation}\label{T,Ti}
e^{-T_W(q)}=\sum_{\substack{b\in\mathcal A_i\\\nu_i([b])>0}} \nu_i([b])^q e^{-T_{V^{(i)}_b}(q)},
\end{equation}
from what it follows, after differentiating at $1$, that
\begin{equation}\label{decomph}
\dim_e(\mu)=h_{\nu_i}(T_i)+\dim_e(\mu|\nu_i),
\end{equation}
where
$$
h_{\nu_i}(T_i)=-\sum_{b\in\mathcal A_i} \nu_i([b])\log\nu_i([b])
$$
is the entropy of the $T_i$-invariant measure $\nu_i$ and
$$
\dim_e(\mu|\nu_i):=\sum_{\substack{b\in\mathcal A_i\\\nu_i([b])>0}} \nu_i([b])  T'_{V^{(i)}_b}(1)=\sum_{\substack{b\in\mathcal A_i\\\nu_i([b])>0}} \nu_i([b]) \left(-\sum_{a\in\mathcal A_1:\; [a]\subset \Pi_i^{-1}([b]) } \mathbb{E} (V_{b,a}^{(i)}\log V_{b,a}^{(i)})\right )
$$
must be thought of as the relative entropy dimension of $\mu$ given $\nu_i$ whenever this number is non-negative.

Among the Mandelbrot measures supported on $K$, special ones will play a prominent role. We introduce them now.

Recall that for  $1\le i\le k$ and $b\in \mathcal A_i=\Pi_i(\mathcal A_1)$, we defined
$$
N^{(i)}_b=\# \{a\in\mathcal A_1:\; [a]\subset  \Pi_i^{-1}([b]),\; [a]\cap K\neq\emptyset\}
$$
and we also defined the set $\widetilde{\mathcal A}_i=\{b\in\mathcal A_i: \, \mathbb{E}(N^{(i)}_b)>0\}$.

For $a\in\mathcal A_1$ such that $[a]\subset  \Pi_i^{-1}([b])$
let
$$
\widetilde V^{(i)}_{b,a}= \begin{cases}
(\mathbb{E}(N^{(i)}_b))^{-1}&\text{ if } b\in\widetilde{\mathcal A}_i \text{ and }[a]\cap K\neq\emptyset\\
0&\text{ otherwise}
\end{cases}.
$$
If $\nu_i$ is a Bernoulli product measure on $X_i$, and $ W_a$ is taken equal to $\widetilde W_a=\nu_i([b]) \widetilde V^{(i)}_{b,a}$ for all $a\in\mathcal A_1$ such that $[a]\subset  \Pi_i^{-1}([b])$, and if $T_{\widetilde W}'(1)>0$, we get a new Mandelbrot measure that we denote by $\mu_{\nu_i}$. By construction, $\nu_i=\mathbb E({\Pi_i}\mu_{\nu_i})$, and
\begin{equation}\label{relatentr}
\dim_e(\mu_{\nu_i}|\nu_i)=\sum_{b\in \mathcal A_i}\nu_i([b])\log \mathbb{E}(N_b^{(i)}).
\end{equation}

The following remarks point out some properties which will play key roles in our proof of Theorem~\ref{dim K}.

\begin{rem}\label{rem0} Given a non-degenerate Mandelbrot measure $\mu$ supported on $K$ and $1\le i\le k$, for each $b\in\mathcal A_i$ such that $\nu_i([b])>0$, the function $T_{V^{(i)}_b}$ is concave, takes value 0 at 1 and due to \eqref{TVi0} one has $T_{V^{(i)}_b}(0)=-\log\mathbb E(N_b^{(i)}(W))\ge -\log\mathbb E(N_b^{(i)})$, so $T_{V^{(i)}_b}$ is bounded from below by the linear function $T_{\widetilde V^{(i)}_b}:q\mapsto (q-1)\log\mathbb E(N_b^{(i)})$. Consequently, $ T_{V^{(i)}_b}'(1)\le T_{{\widetilde V}^{(i)}_b}'(1)=\log\mathbb E(N_b^{(i)})$. It then follows from \eqref{decomph} that

\begin{eqnarray*}
T_{\widetilde W}'(1)&=& h_{\nu_i}(T_i)+\sum_{b\in \A_i:\; \nu_i([b])>0} \nu_i([b]) T_{\widetilde V^{(i)}_b}'(1)\\
&\geq&  h_{\nu_i}(T_i)+\sum_{b\in \A_i:\; \nu_i([b])>0} \nu_i([b]) T_{V^{(i)}_b}'(1)\\
&= &T_{W}'(1),
\end{eqnarray*}
with equality only if $T_{V^{(i)}_b}$ equals $T_{{\widetilde V}^{(i)}_b}$ whenever $\nu_i([b])>0$. In particular, $T_{V^{(i)}_b}$ must be affine. Remembering \eqref{TVi}, and differentiating twice shows, after an application of Cauchy-Schwarz inequality, that this implies that ${\widetilde V}^{(i)}_{b,a}$ equals $\mathbf{1}_{\R_+^*}(W_a)/\mathbb{E}(N_b^{(i)}(W))$ almost surely. Consequently, $T_{V^{(i)}_b}$ equals $T_{{\widetilde V}^{(i)}_b}$ if and only if $V^{(i)}_b=\widetilde{V}^{(i)}_b$ almost surely, for all $b$ such that $\nu_i([b])>0$, i.e. $W=\widetilde W$ almost surely. In particular, either $T_{\widetilde W}'(1)=T_{W}'(1)$ and $\mu_{\nu_i}$ is non-degenerate since $\mu_{\nu_i}=\mu$ almost surely, or $T_{\widetilde W}'(1)>T_{W}'(1)\ge 0$ so that $\mu_{\nu_i}$ is non-degenerate (recall that since $\mu$ is non-degenerate, $T_{W}'(1)= 0$ only in the case that $\mu$ is a Dirac measure).
\end{rem}

{
Recall that for $i=1,\ldots, k$, $X_i=\A_i^\N$ and $\widetilde{X}_i=(\widetilde{\A}_i)^\N$, where $$\widetilde{\A}_i:=\{b\in \A_i:\; \E(N_b^{(i)})>0\}.$$
}

\begin{rem}\label{rem1}
The reader will also check that when $\mu_{\nu_i}$ is non-degenerate, for all $1\le i'\le i-1$, denoting $\mathbb E({\Pi_{i'}}\mu_{\nu_i})$ by $\nu_{i'}$, one also  has $\widetilde W_a=\nu_{i'}(b') \widetilde V^{(i')}_{b',a}$ for all $b'\in \mathcal A_{i'}$ and $a\in \mathcal A_1$ such that $[a]\subset  \Pi_{i'}^{-1}([b])$. Consequently, $\mu_{\nu_{i'}}=\mu_{\nu_i}$ and
$$
\dim_e(\mu_{\nu_i})=h_{\nu_i}(T_i)+\sum_{b\in \mathcal A_i}\nu_i([b])\log \mathbb{E}N_b^{(i)}=h_{\nu_{i'}}(T_{i'})+\sum_{b'\in \mathcal A_{i'}}\nu_{i'}([b'])\log \mathbb{E}(N_{b'}^{(i')})
$$
(with the convention $0\times (-\infty)=0$). Also, if $\nu_i$ is fully supported on $\widetilde X_i$, then $K_{\mu_{\nu_i}}=K$ almost surely since by construction a component $\widetilde W_a$ of the random vector $\widetilde W$ associated with $\mu_{\nu_i}$ as in Section~\ref{skewed}  vanishes exactly when $a\not \in A(\omega)$.
\end{rem}

\begin{de}\label{Bi}
We denote by $\mathcal B_i$ the set of Bernoulli product measures on $X_i$ which are supported on $\widetilde X_i$. Also, if $\rho_i$ is a Bernoulli product measure defined on $\widetilde X_i$, we also denote by $\rho_i$ the Bernoulli product measure on $X_i$ defined by assigning the mass $\rho_i(B\cap\widetilde X_i)$ to any Borel subset $B$ of $X_i$.
\end{de}

\begin{rem}\label{rem2}
Another important fact, which follows from the previous remarks, is that given $1\le i\le k$ and $\rho_i\in\mathcal B_i$ such that $h_{\rho_i}(T_i)+\sum_{b\in \widetilde {\mathcal A}_i}\rho_i([b])\log \mathbb{E}N_b^{(i)}>0$, then $\mu_{\rho_i}$ is the unique Mandelbrot measure $\mu$ jointly defined with $K$ and such that if $\rho=\mathbb{E}(\mu)$, then
$h_\rho(T_1)+\sum_{b\in \widetilde {\mathcal A}_1}\rho([b])\log \mathbb{E}N_b^{(1)}=h_{\rho_i}(T_i)+\sum_{b\in \mathcal A_i}{\rho_i}([b])\log \mathbb{E}N_b^{(i)}$.

On the other hand,  it  also holds that if $\rho_i\in \mathcal B_i$, then
$$
\sup\left\{ h_\rho(T_1)+\sum_{b\in \widetilde {\mathcal A}_1}\rho([b])\log \mathbb{E}N_b^{(1)} : \rho\in \mathcal B_1,\ \Pi_i\rho=\rho_i\right\}= h_{\rho_i}(T_i)+\sum_{b\in \mathcal A_i}{\rho_i}([b])\log \mathbb{E}N_b^{(i)},
$$
and  the supremum is uniquely attained. See e.g. \cite{LeWa77} for a proof (or use \cite[Lemma 1.1]{Bowen1975} for a direct proof).
\end{rem}

\subsection{An optimisation problem}\label{OPT}
 The following result invokes notions introduced to state Theorem~\ref{dim K} and in the previous Section~\ref{skewed}.  Recall Definition~\ref{Bi}.

\begin{thm}\label{optim}
Let
\begin{align*}
M^{\vec{\boldsymbol\gamma}}=\max \{\dim^{\vec{\boldsymbol\gamma}}_e(\mu): \, \mu\text{ is a positive Mandelbrot measure supported on $K$}\}.
\end{align*}

One has
$$
M^{\vec{\boldsymbol\gamma}}=\begin{cases}
P_{i_0}(\theta_{i_0})&\text{ if } i_0\le k\\
P_k(1)&\text{ if }i_0= k+1
\end{cases},
$$
and the maximum is uniquely attained at $\mu_{\nu_{i_0}}$, where $\nu_{i_0}$ is the equilibrium state of $P_{i_0}$ at~$\theta_0$ and  $\mu_{\nu_{i_0}}$ is the Mandelbrot measure associated to $\nu_{i_0}$ as in Section~\ref{skewed}.
\end{thm}
Before proving Theorem~\ref{optim}, we first recall a version of von Neumann's min-max theorem and establish two useful corollaries of it.

\subsubsection{\bf A generalised version of von Neumann's min-max theorem, and some applications}
In this part, we  state a generalised version of von Neumann's min-max theorem, which was proved by  Ky Fan \cite{Fan1953}.

We first introduce some definitions.
Let $f$ be a real function defined on the product  set $X\times Y$ of two arbitrary sets $X$, $Y$ (not necessarily topologized).
$f$ is said to be convex on~$X$, if for any two elements $x_1,\; x_2\in X$ and $t\in [0,1]$, there exists an element $x_0\in X$ such that $f(x_0,y)\leq tf(x_1, y)+(1-t) f(x_2,y)$ for all $y\in Y$. Similarly $f$ is said to be concave  on $Y$, if for any two elements $y_1,\; y_2\in X$ and $t\in [0,1]$, there exists an element $y_0\in Y$ such that $f(x,y_0)\geq tf(x, y_1)+(1-t) f(x,y_2)$ for all $x\in X$.

\begin{thm}\cite[Theorem 2]{Fan1953}.
\label{thm-1.1}
Let $X$ be a compact Hausdorff space and $Y$ an arbitrary set (not topologized). Let $f$ be a real-valued function on $X\times Y$ such that for every $y\in Y$, $f(x,y)$ is lower semi-continuous on $X$. If $f$ is convex on $X$ and concave on Y, then
\begin{equation}
\min_{x\in X} \sup_{y\in Y} f(x,y)=\sup_{y\in Y} \min_{x\in X} f(x,y).
\end{equation}
\end{thm}

Now we give a specific application of Theorem \ref{thm-1.1}.
Let $Y$ be a compact convex subset of a topological vector space.  Let $\ell\in \N^*$ and $g_1,\ldots, g_\ell$ be real concave functions defined on $Y$.
Set $$\Delta_\ell=\left\{(q_1,\ldots, q_\ell)\in \R^\ell:\; q_i\geq 0 \mbox{ and }\sum_{i=1}^\ell q_i=1\right\}.$$
Define $P:\; \Delta_\ell\to \R$ by
$$
P(q_1,\ldots, q_\ell)=\sup\left\{\sum_{i=1}^\ell q_ig_i(y):\; y\in Y\right\}.
$$
Clearly $P$ is a convex function on $\Delta_\ell$.    As a consequence of Theorem \ref{thm-1.1}  we have the following.
\begin{cor}
\label{cor-1}
Under the above setting, we have
\begin{equation}
\label{e-2}
\min_{(q_1,\ldots, q_\ell)\in \Delta_\ell} P(q_1,\ldots, q_\ell)=\sup_{y\in Y} \min\{g_1(y),\ldots, g_\ell(y)\}.
\end{equation}
\end{cor}
\begin{proof}
Define $f:\; \Delta_\ell\times Y\to \R$ by
$$f({\bf q}, y)=\sum_{i=1}^\ell q_ig_i(y) \quad \mbox{ for } {\bf q}=(q_1,\ldots, q_\ell)\in \Delta_\ell,\; y\in Y.
$$
Clearly, $f$ is convex and continuous in ${\bf q}=(q_1,\ldots, q_\ell)$ and concave in $y$.
Notice that for every $y\in Y$, $$\min\{g_1(y),\ldots, g_\ell(y)\}=\min_{(q_1,\ldots, q_\ell)\in \Delta_\ell}\sum_{i=1}^\ell q_ig_i(y)=\min_{(q_1,\ldots, q_\ell)\in \Delta_\ell} f((q_1,\ldots, q_\ell), y).$$
Applying Theorem 1.1 to $f$ (in which we take $X=\Delta_\ell$) yields \eqref{e-2}.
\end{proof}

Next we rewrite \eqref{e-2} as
\begin{equation}
\label{e-3}
\min_{{\bf q}\in \Delta_\ell} \sup_{y\in Y}\; {\bf q}\cdot {\bf g}(y)=
\sup_{y\in Y} \min_{{\bf q}\in \Delta_\ell}\;  {\bf q}\cdot {\bf g}(y)=\sup_{y\in Y}\min\{g_1(y),\ldots, g_\ell(y)\},
\end{equation}
where ${\bf q}=(q_1,\ldots, q_\ell)$ and ${\bf g}(y):=(g_1(y), \ldots, g_\ell(y))$.
\begin{cor}
\label{cor-2}
Assume in addition that $g_1,\ldots, g_\ell$ are strictly concave and continuous (or upper semi-continuous) on $Y$. Then  the following properties hold:
\begin{itemize}
\item[(i)] For any ${\bf q}\in \Delta_\ell$, there is a unique $y=y_{\bf q}\in Y$ which attains the supremum in the variational principle
\begin{equation}
\label{e-3}P({\bf q})=\sup_{y\in Y}\; {\bf q}\cdot {\bf g}(y).
\end{equation}
\item[(ii)] Write $L=\min_{{\bf q}\in \Delta_\ell}P({\bf q})$. Then there is a unique $y\in Y$ such that
\begin{equation}
\label{e-4} \min\{g_1(y),\ldots, g_\ell(y)\}=L.
\end{equation}
Moreover if $P$ takes this minimum at a point ${\bf q}^*=(q_{1}^*,\ldots, q_{\ell}^*)\in\Delta_\ell$, then $y_{\bf q^*}=y$ and furthermore, $q_j^*=0$ if $g_j(y)>L$.
\end{itemize}
\end{cor}
\begin{proof}
Part (i) follows from the assumptions that $g_i$ are upper semi-continuous and strictly concave. Specifically, the first assumption ensures that the supermum in \eqref{e-3}
is attainable. The second one ensures that it is uniquely attained.

To see (ii), first notice that the uniqueness of  $y$ follows from the fact that $L$ equals $\sup_{y\in Y}\min\{g_1(y),\ldots, g_\ell(y)\}$, and from the strict concavity of $\min\{g_1,\ldots,g_\ell\}$, which follows from the strict concavity of  $g_i$.  To prove the remaining statements of part (ii),
suppose that $L=P({\bf q^*})$ for some ${\bf q}^*=(q_{1}^*,\ldots, q_{\ell}^*)\in\Delta_\ell$.  By (i), there is a unique $y_{\bf q^*}\in Y$ so that $P(\bf q^*)={\bf q^*}\cdot {\bf g}(y_{\bf q^*})$.  Recall that $y\in Y$ is the unique element so that \eqref{e-4} holds. Clearly ${\bf q^*}\cdot {\bf g}(y)\geq L$. We prove that $y_{{\bf q^*}}=y$ by contradiction.  If $y\neq y_{\bf q^*}$, then $g_i((y+y_{\bf q^*})/2)>(g_i(y)+g_i(y_{\bf q^*}))/2$ for all $i$ so
$$
P({\bf q^*})\geq {\bf q^*}\cdot {\bf g}((y+y_{\bf q^*})/2)> ({\bf q^*}\cdot {\bf g}(y) +{\bf q^*}\cdot {\bf g}(y_{\bf q^*}))/2\geq L,
$$
leading to a contradiction. Hence $ y=y_{\bf q^*}$. Finally, since $${\bf q^*}\cdot {\bf g}(y)=L=\min  \{g_1(y),\ldots, g_\ell(y)\},$$
it follows that  $q_j^*=0$ if $g_j(y)>L$.
\end{proof}

\subsubsection{\bf Proof of Theorem~\ref{optim}}

For $1\le i\le k$, we identify ${\mathcal B}_i$, the set of Bernoulli product measures on $X_i$ which are supported on $\widetilde X_i$,  with ${\mathcal P}(\widetilde{A}_i)$, the compact convex set of probability measures on $\widetilde{A}_i$.  An element $\nu_i\in  {\mathcal B}_i$ is said to be {\em fully supported on $\widetilde{A}_i$} if $\nu_i(b)\neq 0$ for every $b\in \widetilde{A}_i$.  The projection $\pi_i$ originally defined from $\mathcal A_i$ to $\mathcal A_{i+1}$ maps $\widetilde{A}_i$ to $\widetilde{A}_{i+1}$ ($i=1,\ldots, k-1$). For $\nu_i\in {\mathcal B_i}$ let $H(\nu_i)$ denote the Shannon entropy of $\nu_i$, which coincides with $h_{\nu_i}(T_i)$ but the notation $H(\nu_i)$ will be lighter.

Let  $\widetilde{\phi}:\; \widetilde{A}_1\to \R$ be  defined by
$$
\widetilde{\phi}(a)=\log {\Bbb E}(N_a^{(1)}),\quad q\in \widetilde{A}_1.
$$
We identify $\widetilde \phi$ with the potential defined on $\widetilde X_1$ by $x\mapsto \widetilde \phi(x_1)$.

Write $\mathcal B_1^*:=\{\nu\in {\mathcal B}_1:\; H(\nu)+\nu(\widetilde{\phi})>0\}$. The proof of Theorem \ref{optim} is based on Remarks~\ref{rem0}  to~\ref{rem2}, Corollary~\ref{cor-2}
as well as the following two facts:

\begin{lem}
\label{lem-2}
\begin{eqnarray}
M^{\vec{\boldsymbol\gamma}}&=&\max \left\{ \gamma_1(H(\nu)+\nu(\widetilde{\phi}))+\sum_{i=2}^k \gamma_i \min\{ H(\nu)+\nu(\widetilde{\phi}), H(\Pi_i\nu)\}: \; \nu\in {\mathcal B}_1^*\right\} \nonumber\\
 &=&\max \left\{ \gamma_1(H(\nu)+\nu(\widetilde{\phi}))+\sum_{i=2}^k \gamma_i \min\{ H(\nu)+\nu(\widetilde{\phi}), H(\Pi_i\nu)\}: \; \nu\in {\mathcal B}_1\right\} \label{e-3.1}.
\end{eqnarray}
\end{lem}
\begin{proof}
At first, note that for any non-degenerate non-atomic Mandelbrot measure $\mu$, $\dim_e(\mu)=H(\nu)+\nu(\widetilde{\phi})>0$, where $\nu={\Bbb E}(\mu)$. Then, notice that if $\nu\in {\mathcal B}_1\backslash {\mathcal B}_1^*$, we have $H(\nu)+\nu(\widetilde{\phi})\leq 0$ and so
$$\gamma_1(H(\nu)+\nu(\widetilde{\phi}))+\sum_{i=2}^k \gamma_i \min\{ H(\nu)+\nu(\widetilde{\phi}), H(\Pi_i\nu)\}\leq 0.$$
However we know that $M^{\vec{\boldsymbol\gamma}}>0$ (which follows from our assumption  $\mathbb{E} (\# A)>1$ and by considering the branching measure, i.e. the Mandelbrot measure on $K$ associated with the random vector  $W=(\mathbf{1}_{A}(a) /\mathbb E(\# A))_{a\in\mathcal A_1}$),  so the maximum must be taken over~${\mathcal B}_1^*$.
\end{proof}

Now we are ready to prove Theorem \ref{optim}.
\begin{proof}[Proof of Theorem \ref{optim}]
For $i=1,\ldots, k$, define $g_i:\; \mathcal B_1\to \R$ by
\begin{equation}\label{gi}
g_i(\nu)=(\gamma_1+\ldots+\gamma_i) (H(\nu)+\nu(\widetilde{\phi}))+\sum_{j=i+1}^k \gamma_j H(\Pi_j \nu).
\end{equation}
Then $g_i$ are strictly concave functions on $\mathcal B_1$  (identified with ${\mathcal P}(\widetilde{A}_1)$).   Clearly for each $\nu\in \mathcal B_1$,
$$\gamma_1(H(\nu)+\nu(\widetilde{\phi}))+\sum_{i=2}^k \gamma_i \min\{ H(\nu)+\nu(\widetilde{\phi}), H(\Pi_i\nu)\}=\min_{1\leq i\leq k} g_i(\nu).$$
Hence by \eqref{e-3.1}, we obtain
\begin{equation}
\label{e-3.2}
M^{\vec{\boldsymbol\gamma}}=\sup_{\nu\in \mathcal B_1} \min\{g_1(\nu),\ldots, g_k(\nu)\}.
\end{equation}
Define $P:\; \Delta_k\to \R$ by
\begin{equation}
\label{e-3.3}
P(q_1,\ldots, q_k)=\sup_{\nu\in {\mathcal B}_1} \sum_{i=1}^k q_i g_i(\nu).
\end{equation}
Applying Corollaries \ref{cor-1}-\ref{cor-2}, we see that
$$M^{\vec{\boldsymbol\gamma}}=\min_{(q_1,\ldots, q_k)\in \Delta_k} P(q_1,\ldots, q_k)$$
and that the supremum in \eqref{e-3.2} is uniquely attained at some element of $\mathcal B_1$, say  $\nu$; moreover,   $\nu$ is the unique equilibrium state
for $P({\bf q}^*)$, provided that $P$ takes its minimum at  ${\bf q}^*=(q_1^*,\ldots, q_k^*)$.

Fix such a ${\bf q}^*$.  Since $\nu$ is the equilibrium state  for $P({\bf q}^*)$, it must be fully supported on $\widetilde{A}_1$ (see \cite{BF2012}).
By the strong concavity of the entropy function $H$, it follows that
\begin{equation}
\label{e-3.4}
H(\Pi_i\nu)\ge H(\Pi_{i+1}\nu)\quad \mbox{ for }i=1,\ldots, k-1;
 \end{equation}
and the equality holds only if $\pi_i: \widetilde{\A}_i\to \widetilde{\A}_{i+1}$ is injective.
So either one of the following 3 cases occurs: (i) $H(\nu)+\nu(\widetilde{\phi})\geq H(\Pi_2 \nu)$; (ii) there exists a unique $i\in \{2,\ldots, k-1\}$ so that
$H(\Pi_{i+1}\nu)\leq H(\nu)+\nu(\widetilde{\phi})< H(\Pi_i\nu)$; (iii) $H(\nu)+\nu(\widetilde{\phi})< H(\Pi_k \nu)$.

If Case (i) occurs, then $g_1(\nu)\leq g_2(\nu)\le g_3(\nu)\le \cdots\le g_k(\nu)$, and by Corollary \ref{cor-2}(ii), we have $q_j^*=\ldots=q_{k}^*=0$ where $j$ is the smallest index $i$ such that $g_i(\nu)>g_2(\nu)$, if any; if there is no such an index, set  $j=k+1$.
So if $j=3$, then  $P({\bf q}^*)=P(q_1^*, q_2^*,0,\ldots,0)$. Below we assume that $j>3$. Since $g_2(\nu)=\ldots=g_{j-1}(\nu)$, it follows that
\begin{equation}
\label{e-F-1}
H(\Pi_2\nu)=H(\Pi_3\nu)=\ldots=H(\Pi_{j-1}\nu),
\end{equation}
and
\begin{equation}
P({\bf q}^*)=P(q_1^*, q_2^*,\ldots, q_{j-1}^*,0\ldots,0).
\end{equation}
By \eqref{e-F-1} we see that the mappings $\pi_2:\; \widetilde{\A}_2\to \widetilde{\A}_{3}$, \ldots, $\pi_{j-2}: \widetilde{\A}_{j-2}\to \widetilde{\A}_{j-1}$ are injective, which implies that
\begin{equation}
\label{e-F-2}
H(\Pi_2\mu)=H(\Pi_3\mu)=\cdots=H(\Pi_{j-1}\mu) \quad \mbox{ for all } \mu\in \mathcal B_1.
\end{equation}
Now $P(q_1^*,\cdots,q_{j-1}^*,0,\ldots,0)=\sup_{\mu\in \mathcal B_1} (q_1^*g_1(\mu)+\cdots+q_{j-1}^*g_{j-1}(\mu))$.
According to \eqref{e-F-2}, we have
\begin{align*}
g_1(\mu)&=\gamma_1(H(\mu)+\mu(\widetilde{\phi}))+(\gamma_2+\cdots+\gamma_{j-1})H(\Pi_2\mu)+\sum_{p=j}^k \gamma_pH(\Pi_p\mu),\\
g_2(\mu)&=(\gamma_1+\gamma_2)(H(\mu)+\mu(\widetilde{\phi}))+(\gamma_3+\cdots+\gamma_{j-1})H(\Pi_2\mu)+\sum_{p=j}^k \gamma_pH(\Pi_p\mu),\\
&\cdots\\
g_{j-1}(\mu)&=(\gamma_1+\cdots+\gamma_{j-1})(H(\mu)+\mu(\widetilde{\phi}))+\sum_{p=j}^k \gamma_pH(\Pi_p\mu).
\end{align*}
It follows that for $(q_1,\ldots, q_{j-1})\in \Delta_{j-1}$,
\begin{align*}
\sup_{\mu\in \mathcal B_1} \sum_{i=1}^{j-1} q_ig_i(\mu)&=\sup_{\mu\in \mathcal B_1} \sum_{t=1}^{j-1}q_t(\gamma_1+\cdots+\gamma_t)(H(\mu)+\mu(\widetilde\phi))+\sum_{s=1}^{j-2}q_s(\gamma_{s+1}+\cdots+\gamma_{j-1})H(\Pi_2\mu)\\
&\qquad\quad+\sum_{p=j}^k\gamma_pH(\Pi_p\mu)\\
&=\sup_{\eta\in \mathcal B_2} \sum_{t=1}^{j-1}q_t(\gamma_1+\cdots+\gamma_t)(H(\eta)+\eta(\widetilde\phi_2))+\sum_{s=1}^{j-2}q_s(\gamma_{s+1}+\cdots+\gamma_{j-1})H(\eta)\\
&\qquad\quad+\sum_{p=j}^k\gamma_pH(\Pi_{2,p}\eta)\qquad \mbox{(where $\Pi_{2,p}:=\pi_{p-1}\circ\cdots\circ  \pi_{2}$)}\\
&=\sup_{\eta\in \mathcal B_2} \sum_{t=1}^{j-1}q_t(\gamma_1+\cdots+\gamma_t)\eta(\widetilde\phi_2)+(\gamma_1+\cdots+\gamma_{j-1})H(\eta)+
\sum_{p=j}^k\gamma_pH(\Pi_{2,p}\eta),
\end{align*}
where in the second equality, we used the fact that $$\sup\{H(\mu)+\mu(\widetilde{\phi}):\; \mu\in \mathcal B_1,\; \Pi_2\mu=\eta\}=H(\eta)+\eta(\widetilde{\phi}_2)$$ for each $\eta\in \mathcal B_2$ (see Remark \ref{rem2}).
Hence $P(q_1,\ldots, q_{j-1},0,\ldots,0)$  depends only on $\sum_{t=1}^{j-1}q_t(\gamma_1+\cdots+\gamma_t)$. Clearly, there exists $2<j'\leq j$ such that
$$
\gamma_1+\cdots+\gamma_{j'-2}\leq \sum_{t=1}^{j-1}q_t(\gamma_1+\cdots+\gamma_t)\leq \gamma_1+\cdots+\gamma_{j'-1}
$$
so there exists  $(p_{j'-2}, p_{j'-1})\in \Delta_2$ such that $$\sum_{t=1}^{j-1}q_t(\gamma_1+\cdots+\gamma_t)=p_{j'-2}(\gamma_1+\cdots+\gamma_{j'-2})+p_{j'-1}(\gamma_1+\cdots+\gamma_{j'-1})$$
and so
$$
P(q_1,\ldots, q_{j-1},0,\ldots,0)=P(\underbrace{0,\ldots,0,}_{(j'-3)\text{ times}} p_{j'-2}, p_{j'-1},0,\ldots,0).
$$
Applying this fact to the case when $(q_1,\ldots, q_{j-1})=(q_1^*,\ldots, q_{j-1}^*)$, we see that there exist $2<j'\leq j$ and  $(p_{j'-2}, p_{j'-1})\in \Delta_2$ so that
$$
P({\bf q}^*)=P(q_1^*,\ldots, q_{j-1}^*,0,\ldots, 0)=P(\underbrace{0,\ldots,0,}_{(j'-3)\text{ times}} p_{j'-2}, p_{j'-1},0,\ldots,0).
$$

 If Case (ii) occurs, then either (a) $H(\Pi_{i+1}\nu)< H(\nu)+\nu(\widetilde{\phi})< H(\Pi_i\nu)$, or (b)
 $H(\Pi_{i+p+1}\nu)<H(\Pi_{i+p}\nu)=\ldots =H(\Pi_{i+1}\nu)=H(\nu)+\nu(\widetilde{\phi})<H(\Pi_i\nu)$ for some $p\geq 1$.
 In the subcase (a),  $g_1(\nu)>\ldots>g_i(\nu)<g_{i+1}(\nu)\leq \ldots \leq g_k(\nu)$,
so by Corollary \ref{cor-2}(ii), $q_{j}^*=0$ for $j\neq i$, i.e. ${\bf q}^*=(0,\ldots,0, q_i^*, 0,\ldots,0)$ with $q_i^*=1$.
In the subcase (b), we have  $$q_1(\nu)>\ldots>g_i(\nu)=g_{i+1}(\nu)=\cdots=g_{i+p}(\nu)<g_{i+p+1}(\nu)\leq \ldots \leq g_k(\nu),$$
so by Corollary \ref{cor-2}(ii), $q_{j}^*=0$ if $j<i$ or $j>i+p$, i.e.
 $$
 P({\bf q}^*)=P(\underbrace{0,\ldots, 0}_{(i-1)\text{ times}}, q_{i}^*,\ldots, q_{i+p}^*,0,\ldots,0).
 $$
Following an argument similar as in the first case, we can show that there exist $1\leq p'\leq p$ and $(q_{i+p'-1}, q_{i+p'})\in \Delta_2$ so that
$$
 P({\bf q}^*)=P(0,\ldots,0, q_{i}^*,\ldots, q_{i+p}^*,0,\ldots,0)=P(\underbrace{0,\ldots, 0}_{(i+p'-2)\text{ times}}, q_{i+p'-1}, q_{i+p'},0,\ldots,0).
$$

Finally if Case (iii) occurs, then $g_1(\nu)>g_2(\nu)>\cdots>g_k(\nu)$.  By Corollary \ref{cor-2}(ii), $q_{j}^*=0$ for $j<k$.
So $P({\bf q}^*)=P(0,\ldots, 0,1)$.

The discussions in the above paragraph describe the possible forms of $P({\bf q}^*)$. As a direct consequence, we have
\begin{equation}
\label{e-3.0}
M^{\vec{\boldsymbol\gamma}}=\min \{R_1,\ldots, R_{k-1}, P(0,\ldots,0,1)\},
\end{equation}
where
\begin{equation}
\label{e-3.5}
R_i:=\min\{P(q_1,\ldots, q_k):\; (q_1,\ldots, q_k)\in \Delta_k,\; q_j=0 \mbox{  for } j\in \{1,\ldots, k\}\backslash \{i, i+1\}\}
\end{equation}
for $i=1,\ldots, k-1$. In what follows, we further investigate the values of $R_i$.

Fix $i\in \{1,\ldots, k-1\}$. By definition,
\begin{equation}
\label{e-3.6}
R_i=\inf_{t\in [0,1]}P(\underbrace{0,\ldots, 0}_{(i-1)\text{ times}}, 1-t, t, 0,\ldots, 0).
\end{equation}
For each $\nu_{i+1}\in {\mathcal B}_{i+1}$,
$$
\sup\{H(\nu)+\nu(\widetilde{\phi}):\;  \nu\in \mathcal B_1,\  \Pi_{i+1}\nu=\nu_{i+1}\}=H(\nu_{i+1})+\nu_{i+1}(\widetilde{\phi}_{i+1}),
$$
where  $\widetilde{\phi}_{i+1}$ is the potential on $\widetilde{X}_{i+1}$ defined by  $\widetilde{\phi}_{i+1}(x)=\log {\Bbb E}(N_{x_1}^{(i+1)})$ (see Remark~\ref{rem2}). This, combined with the definition of $P$, yields (noting that $\widetilde{\phi}_{i+1}={\phi}_{i+1}/(\gamma_1+\cdots+\gamma_{i+1})$, where ${\phi}_{i+1}$ was defined in \eqref{phiktheta})
\begin{align*}
P&(\underbrace{0,\ldots, 0}_{(i-1)\text{ times}}, 1-t, t, 0,\ldots, 0)\\
&=\sup_{\nu\in \mathcal B_1}\left\{(1-t)\left[(\gamma_1+\ldots+\gamma_i)(H(\nu)+\nu(\widetilde{\phi}))+\sum_{j=i+1}^k\gamma_jH(\Pi_j\nu)\right]\right.\\
&\qquad\qquad + \left. t \left[(\gamma_1+\ldots+\gamma_{i+1})(H(\nu)+\nu(\widetilde{\phi}))+\sum_{j=i+2}^k\gamma_jH(\Pi_j\nu)\right]\right\}\\
&=\sup_{\nu_{i+1}\in {\mathcal B}_{i+1}} \Big[(\gamma_1+\ldots+\gamma_{i}+t\gamma_{i+1})\nu_{i+1}(\widetilde{\phi}_{i+1})+ (\gamma_1+\ldots+\gamma_{i+1})H(\nu_{i+1})\\
& \qquad\qquad\qquad  +\sum_{j=i+2}^k \gamma_j H(\Pi_{i+1, j} \nu_{i+1})\Big]\\
&=P_{i+1} \left( \frac{\gamma_1+\ldots+\gamma_{i}+t\gamma_{i+1}}{\gamma_1+\ldots+\gamma_{i+1}} \right),
\end{align*}
where $\Pi_{i+1,j}=\pi_{j-1}\circ\cdots\circ  \pi_{i+1}$, and $P_{i+1}$ is defined as in~\eqref{e-p-i}. So by \eqref{e-3.6} we have
$$R_i=\min\{P_{i+1}(t):\; (\gamma_1+\ldots+\gamma_i)/(\gamma_1+\ldots+\gamma_{i+1})\leq  t\leq 1\},\quad i=1,\ldots, k-1.$$
Notice that $P(0,\ldots, 0, 1)=P_{k}(1)$. Recalling that $\widetilde\theta_{i}= (\gamma_1+\ldots+\gamma_{i-1})/(\gamma_1+\ldots+\gamma_{i})$ for $2\le i\le k$, the previous discussion yields
\begin{equation}\label{Mgamma}
M^{\vec{\boldsymbol\gamma}}=\min\Big (\min_{2\leq i\leq k} \inf  \{ P_{i}(\theta): \theta\in [\widetilde\theta_{i},1]\}, P_k(1)\Big ).
\end{equation}

Finally we prove that
$$M^{\vec{\boldsymbol\gamma}}=\left\{\begin{array}{ll}
P_{i_0}(\theta_{i_0}) & \mbox{  if } i_0\leq k\\
P_{k}(1) & \mbox{  if } i_0= k+1
\end{array}
\right.,
$$
where we recall that $i_0$ and $\theta_{i_0}$ where defined just before the statement of Theorem~\ref{dim K}.

Below we divide our discussions into 3 steps :

{\sl Step 1}.
First assume that $i_0=2$, i.e., $P_2'(1)\geq 0$. By the convexity of $P_2$ we see that $P_2'(\theta)$ is non-decreasing, so there are only two possibilities:  (1) $P_2'(\widetilde \theta_2)\geq 0$ (in which case, $\theta_{i_0}=\widetilde \theta_2$) ; or (2) $P_2'(\theta_{i_0})=0$.

Recall $P_2'(\theta)=(\gamma_1+\gamma_2) \cdot \nu_{\theta \phi_2}(\phi_2)$, where $\nu_{\theta \phi_2}\in \mathcal B_2$ denotes the equilibrium state of the function $P_2$ at $\theta$.

If (1) occurs, then by the above derivative formula, we have
$$
\nu_{\widetilde\theta_2 \phi_2}(\phi_2)=(\gamma_1+\gamma_2)\nu_{\widetilde\theta_2 \phi_2}(\widetilde{\phi}_2)\geq 0.
$$
To simplify the notation, we write $\nu_2:=\nu_{\widetilde\theta_2 \phi_2}$.  Let $\nu$ be the unique element in $\mathcal B_1$ such that $\Pi_2\nu=\nu_2$ and
$H(\nu)+\nu(\widetilde{\phi})=H(\nu_2)+\nu_2(\widetilde{\phi}_2)$. Since $\nu_2(\widetilde{\phi}_2)\geq 0$,  we have
\begin{equation}
\label{e-3.8}H(\nu)+\nu(\widetilde{\phi})\geq H(\nu_2)\ge H(\Pi_3\nu)\ge \ldots\ge H(\Pi_k\nu).
\end{equation}
In the meantime, \begin{align*}
P_2(\widetilde \theta_2)&=\gamma_1\nu_2(\widetilde{\phi}_2)+(\gamma_1+\gamma_2)H(\nu_2)+\sum_{j=3}^k \gamma_j H(\Pi_{2,j} \nu_2)\\
&=\gamma_1(H(\nu)+\nu(\widetilde{\phi}))+\sum_{j=2}^k\gamma_jH(\Pi_j\nu))=g_1(\nu).
\end{align*}
By \eqref{e-3.8}, we have $g_1(\nu)\leq g_2(\nu)\le \cdots\le g_k(\nu)$. So by \eqref{e-3.2},
$$P_2(\widetilde \theta_2)=g_1(\nu)=\min\{g_1(\nu), \ldots, g_k(\nu)\}\leq M^{\vec{\boldsymbol\gamma}}.$$
But $P_2(\widetilde \theta_2)\geq \inf_{{\bf q}\in \Delta_k}P({\bf q})=M^{\vec{\boldsymbol\gamma}}$. So we get $P_2(\widetilde \theta_2)=M^{\vec{\boldsymbol\gamma}}$.

Next assume that Case (2) occurs.  We still let $\nu_2$ denote the equilibrium state of $P_2$ at~$\theta_{i_0}$. The derivative formula gives that $\nu_2(\widetilde{\phi}_2)=0$ and so \begin{equation}
\label{e-3.9}
P_2(\theta_{i_0})=(\gamma_1+\gamma_2)H(\nu_2)+\sum_{j=3}^k \gamma_j H(\Pi_{2,j} \nu_2).
\end{equation}
Let $\nu$ be the unique element in $\mathcal B_1$ such that $\Pi_2\nu=\nu_2$ and
$H(\nu)+\nu(\widetilde{\phi})=H(\nu_2)+\nu_2(\widetilde{\phi}_2)=H(\nu_2)$. Then  $g_1(\nu)= g_2(\nu)\le \cdots\le g_k(\nu)$. By \eqref{e-3.9} and the fact that $\nu_2(\widetilde{\phi}_2)=0$,
$$P_2(\theta_{i_0})=\gamma_1\nu_2(\widetilde{\phi}_2)+(\gamma_1+\gamma_2)H(\nu_2)+\sum_{j=3}^k \gamma_j H(\Pi_{2,j} \nu_2)=g_1(\nu).$$
Similar to the discussion for Case (1), we also obtain $P_2(\theta_{i_0})=M^{\vec{\boldsymbol\gamma}}$.

It follows from Remark~\ref{rem0}, as well as Corollary~\ref{cor-2}(ii)
that in any of the cases discussed above, $\mu_{\nu_2}$ is the unique Mandelbrot measure $\mu$ supported on $K$ such that $\dim_H(\mu)=M^{\vec{\boldsymbol\gamma}}$,  conditional on $K\neq\emptyset$.

{\sl Step 2}. Assume that $2<i_0\leq k$. Then $P_{i_0-1}'(1)<0$ and $P_{i_0}'(1)\geq 0$.  However, notice that $P_{i_0-1}(1)=P_{i_0}(\widetilde\theta_{i_0})$ (both of them are equal to  $$P(\underbrace{0,\ldots, 0}_{(i_0-2)\text{ times}}, 1, 0, 0,\ldots, 0).$$
Let $\nu_{i_0-1}$ and $\nu_{i_0}$ denote the equilibrium states of $P_{i_0-1}$ at $1$ and $P_{i_0}$ at $\widetilde\theta_{i_0}$, respectively. Assume at the moment that $\pi_{i_0-1}\nu_{i_0-1}=\nu_{i_0}$; we will check it at the end of the proof.  Now from $P_{i_0-1}'(1)<0$ it follows that $\nu_{i_0-1}(\widetilde{\phi}_{i_0-1})<0$. Next we consider two possible cases:  (a) $P_{i_0}'(\widetilde\theta_{i_0})>0$; (b) $P_{i_0}'(\widetilde\theta_{i_0})\leq 0$.
If Case (a) occurs, then $\nu_{i_0}(\widetilde{\phi}_{i_0})>0$ and $\theta_{i_0}=\widetilde\theta_{i_0}$.  Let $\nu$ be the unique element in $\mathcal B_1$ such that $\Pi_{i_0}\nu=\nu_{i_0}$ and
$H(\nu)+\nu(\widetilde{\phi})=H(\nu_{i_0})+\nu_{i_0}(\widetilde{\phi}_{i_0})$. We will check that  $H(\nu)+\nu(\widetilde{\phi})=H(\nu_{i_0-1})+\nu_{i_0-1}(\widetilde{\phi}_{i_0-1})$ as well.
Hence $H(\nu_{i_0})<H(\nu)+\nu(\widetilde{\phi})<H(\nu_{i_0-1})$. It follows that
$$g_1(\nu)\ge \ldots\ge g_{i_0-1}(\nu)>g_{i_0}(\nu)\le \ldots\le g_k(\nu).$$
However $P_{i_0}(\widetilde\theta_{i_0})=g_{i_0}(\nu)$ so $P_{i_0}(\widetilde\theta_{i_0})=\min\{g_i(\nu):\; 1\leq i\leq k\}\leq M^{\vec{\boldsymbol\gamma}}$. Hence we have $P_{i_0}(\theta_{i_0})=P_{i_0}(\widetilde\theta_{i_0})= M^{\vec{\boldsymbol\gamma}}$.

Suppose Case (b) holds. Then $P_{i_0}'(\theta_{i_0})=0$. Following the similar argument as in Case (2) of Step 1, we see that $P_{i_0}(\theta_{i_0})=M^{\vec{\boldsymbol\gamma}}$ by considering $\nu$ defined as in Case (a). Note that $\nu_{\theta_{i_0}\phi_{i_0}}(\widetilde\phi_{i_0})=0$. We have $P_{i_0}(\theta_{i_0})=g_{i_0}(\nu) $ and $g_1(\nu)\ge g_2(\nu)\ge \cdots\ge  g_{i_0-1}(\nu)=g_{i_0}(\nu)\le g_{i_0+1}(\nu)\le \cdots\le g_k(\nu)$. So $P_{i_0}(\theta_{i_0})= M^{\vec{\boldsymbol\gamma}}$.

Like in the first step, in any case, $\mu_{\nu_{i_0}}$ is the unique Mandelbrot measure $\mu$ supported on $K$ such that $\dim_H(\mu)=M^{\vec{\boldsymbol\gamma}}$, conditional on $K\neq\emptyset$.

{\sl Step 3}. Assume that $i_0= k+1$. Then $P_i'(1)<0$ for every $i\in \{2,\ldots, k\}$.  It follows that for each $i\in \{1,\ldots, k-1\}$,
 $P(\underbrace{0,\ldots, 0}_{(i-1)\text{ times}}, 1-t, t, 0,\ldots, 0)$ is strictly decreasing in $t$. Hence by \eqref{e-3.0}-\eqref{e-3.6},
 $M^{\vec{\boldsymbol\gamma}}=P(0,\ldots, 0, 1)=P_{k}(1)$. Let $\nu_k$ be the equilibrium state of $P_k$ at $1$. and $\nu$ be the unique element of $\mathcal B_1$ such that $\Pi_k\nu=\nu_k$ and  $H(\nu)+\nu(\widetilde\phi)=H(\nu_k)+\nu_k(\widetilde\phi_k)$. Then $P_{k}(1)=g_k(\nu)$. We conclude as in the two first steps.

The proof of Theorem \ref{optim} will be complete after we have checked that in Step 2 we have $\pi_{i_0-1}\nu_{i_0-1}=\nu_{i_0}$ and  $H(\nu)+\nu(\widetilde{\phi})=H(\nu_{i_0-1})+\nu_{i_0-1}(\widetilde{\phi}_{i_0-1})$.

To see that these equalities do hold, we look at the optimisation problem consisting in determining the equilibrium state $\nu_{i_0-1}$ of $P_{i_0-1}$ at $1$, by conditioning on the knowledge of  $\pi_{i_0-1}\nu_{i_0-1}$: using the relation $\phi_{i_0-1}= \widetilde \gamma_{i_0-1} \widetilde\phi_{i_0-1}$, we have
\begin{align*}
&P_{i_0-1}(1)\\
&=\sup_{\rho_{i_0-1}\in\mathcal B_{i_0-1}} \widetilde \gamma_{i_0-1} \big (\rho_{i_0-1}(\widetilde \phi_{i_0-1})+H(\rho_{i_0-1})\big )+\sum_{i=i_0}^k \gamma_i H(\Pi_{i_0-1,i}\rho_{i_0-1})\\
&= \sup_{\rho_{i_0}\in\mathcal B_{i_0}} \widetilde \gamma_{i_0}H(\rho_{i_0})+\sum_{i=i_0+1}^k \gamma_i H(\Pi_{i_0,i}\rho_{i_0})\\
&\quad +\sup_{(p(\cdot|\hat b))_{\hat b\in \widetilde{\mathcal A}_{i_0}}}\sum_{\widehat b\in\mathcal A_{i_0}} \rho_{i_0}([\widehat b]) \cdot   \widetilde \gamma_{i_0-1}\sum_{\substack{ b\in \widetilde{\mathcal A}_{i_0-1}\\ \pi_{i_0-1}(b)=\widehat b}} -p(b|\widehat b)\log p(b|\widehat b)+ p(b|\widehat b) \log \mathbb{E}(N^{(i_0-1)}_b),
\end{align*}
where $(p(b|\widehat b))_{b\in \mathcal A_{i_0-1}, \pi_{i_0-1}(b)=\widehat b}$ is the probability vector such that  $p(b|\widehat b)=\rho_{i_0-1}(b)/\rho_{i_0}(\widehat b)$ when $\rho_{i_0}(\widehat b)>0$, and any probability vector otherwise. Optimising over the families $(p(b|\widehat b))_{b\in \mathcal A_{i_0-1}, \pi_{i_0-1}(b)=\widehat b}$ yields
$$
p(b|\widehat b)=\frac{\mathbb{E}(N^{(i_0-1)}_b)}{\mathbb{E}(N^{(i_0)}_{\widehat b})}
$$
when $\rho_{i_0}(\widehat b)>0$, and
\begin{align*}
P_{i_0-1}(1)&=\sup_{\rho_{i_0}\in\mathcal B_{i_0}} \widetilde \gamma_{i_0-1} \rho_{i_0}(\widetilde \phi_{i_0})+ \widetilde \gamma_{i_0}H(\rho_{i_0})+\sum_{i=i_0+1}^k \gamma_i H(\Pi_{i_0,i}\rho_{i_0})\\
&= \sup_{\rho_{i_0}\in\mathcal B_{i_0}}  \rho_{i_0}( \widetilde \theta_{i_0}\phi_{i_0})+ \widetilde \gamma_{i_0}H(\rho_{i_0})+\sum_{i=i_0+1}^k \gamma_i H(\Pi_{i_0,i}\rho_{i_0}) (=P_{i_0}(\widetilde \theta_{i_0})),
\end{align*}
from which the desired equalities follow.
\end{proof}

\subsection{Lower bound for the Hausdorff dimension of $K$}\label{LB}

The sharp lower bound comes from the optimisation problem solved  in Section~\ref{OPT}. Consider the unique Mandelbrot measure $\mu=\mu_{\nu_{\theta_{i_0}\phi_{i_0}}}$  obtained there. By construction the measure $\mu$ is fully supported on $K$ conditional on $K\neq\emptyset$, because $\nu_{\theta_{i_0}\phi_{i_0}}$ is fully supported on $\widetilde X_{i_0}$ if $i_0\le k$ and $\widetilde X_{k}$ otherwise, and due to the last claim of Remark~\ref{rem1}. Also, the assumptions of Theorem~\ref{thmdim} are fulfilled for $\mu$, and the Hausdorff dimension of $\mu$ provides the desired  lower bound  for $\dim_H K$.

\subsection{Upper bound for the Hausdorff dimension of $K$}\label{UB}

Let us start by discussing a first possible attempt to show that $\dim_H K\le M^{\vec{\boldsymbol\gamma}}$. We could expect to use the measure $\mu=\mu_{\nu_{\theta_{i_0}\phi_{i_0}}}$ of maximal Hausdorff dimension $M^{\vec{\boldsymbol\gamma}}$ and show that $\underline\dim_{\rm loc}(\mu,x)\le M^{\vec{\boldsymbol\gamma}}$ everywhere on $K$; this is the approach used by McMullen  \cite{McMullen}  as well as Kenyon and Peres \cite{KePe96} in the deterministic case;  that would make it possible to conclude quite quickly. In the random situation, we can show that this approach via the lower local dimension works  in the case when $N^{(2)}_b\ge 2$ almost surely for all $b\in \widetilde {\mathcal{A}}_2$; say in this case that $K$ is of type I.  This requires quite involved  moments estimates for martingales in varying environments. Notice that in this case  we have $i_0=2$ and  $\theta_2=\gamma_1/(\gamma_1+\gamma_2)$. The type I makes it possible to treat the case of a slightly more general type of examples, still quite close to the deterministic case:   $i_0=2$, $\theta_2=\gamma_1/(\gamma_1+\gamma_2)$,  and it is possible to approximate $K$ by a sequence $(K^{(p)})_{p\in\mathbb N}$ of random Sierpinski sponges of type I in the sense that $K\subset K^{(p)}$ for all $p\in\mathbb N$, $\bigcap_{p\in\mathbb N}K^{(p)}=K$, and  $\lim_{p\to\infty} \dim_H K^{(p)}=M^{\vec{\boldsymbol\gamma}}$. A sufficient condition to be in this situation is that $\Psi_2(0)<\Psi_2( \gamma_1/(\gamma_1+\gamma_2))$, where $\Psi_2(\theta)=\sum_{b\in \mathcal A_2} \mathbb E (N^{(2)}_b)^\theta$ (this condition obviously holds for examples of type I).

Thus, regarding the lower local dimension approach, it remains open whether or not  in general it holds that  $\underline\dim_{\rm loc}(\mu,x)\le M^{\vec{\boldsymbol\gamma}}$ everywhere on $K$; moreover,  the sufficient condition just stated to get the sharp upper bound for $\dim_H K$ is not at all satisfactory.

The alternative is to examine the strategy that Gatzouras and Lalley adopted for the two dimensional case. Their approach is inspired by Bedford's treatment of the deterministic two dimensional case, and it uses effective coverings of the set $K$ to find the sharp upper bound for $\dim_H K$. These coverings are closely related to a combinatoric argument due to Bedford. But this argument turns out to be hard to extend to higher dimensional cases. Below, we use a different, though related, combinatoric argument, which yields  nice effective coverings as well,  but works in any dimension. Also, in the deterministic case and in any dimension, it provides an alternative to the argument  using a uniform bound for the lower local dimension of~$\mu$. However, and interestingly, our  argument uses a slight generalisation of a key combinatoric lemma  established by Kenyon and Peres to get this uniform bound.


We now provide a general upper bound for  $\dim_H K$, expressed through a variational principle.
\begin{thm}\label{VPUB} With probability 1, conditional on $K\neq\emptyset$,
$$
\dim_H K\le \inf\Big \{P_i(\theta): \;  i\in I, \; \widetilde\theta_i\le \theta\le 1\Big \}.
$$
\end{thm}
The fact that this is a sharp upper bound for $\dim_H K$ follows from \eqref{Mgamma} and the inequality $\dim_H K\ge M^{\vec{\boldsymbol\gamma}}$ obtained in the previous section.

\medskip

Before proving Theorem~\ref{VPUB}, we need to introduce some new definitions, and to make some preliminary observations.

Let $2\le i\le k$ and $ \widetilde\theta_i\le \theta\le 1$. For  $\rho_i \in\mathcal B_i$ set
\begin{equation}\label{Dit}
D_{i,\theta}(\rho_i)=\widetilde \gamma_{i} \theta \sum_{b\in\mathcal A_i} \rho_i([b])\log \mathbb E(N^{(i)}_b) +  \widetilde \gamma_{i}h_{\rho_i}(T_i)+ \sum_{j=i+1}^k \gamma_j  h_{{\Pi_{i,j}}\rho_i}(T_j)
\end{equation}
(again with the convention $0\times (-\infty)=0$.

For $\rho=(\widetilde \rho_i,\rho_{i},\ldots,\rho_k)\in\mathcal B_i\times \prod_{j=i}^k\mathcal B_j$, set
\begin{equation}\label{Dit'}
\widetilde D_{i,\theta}(\rho)=\widetilde \gamma_{i} \theta \sum_{b\in\mathcal A_i} \widetilde \rho_i([b])\log \mathbb E(N^{(i)}_b) +  \widetilde \gamma_{i}h_{\rho_i}(T_i)+ \sum_{j=i+1}^k \gamma_j  h_{\rho_j}(T_j).
\end{equation}
%


For each $2\le i\le k$. We endow the set $\mathcal B_i\times \prod_{j=i}^k\mathcal B_j$ with the distance
$$
d_i(\rho,\rho')=\max\left (\max_{b\in\mathcal A_i} |\widetilde \rho_i([b])-\widetilde \rho'_i([b])|,\;  \max_{i\le j\le k} \max_{b\in\mathcal A_j} |\rho_j([b])-\rho'_j([b])|\right ),
$$
which makes it a compact set.  Let
\begin{equation}\label{Ri}
\mathcal R_i=\left\{\rho\in  \mathcal B_i\times \prod_{j=i}^k\mathcal B_j:\;  \widetilde D_{i,\theta}(\rho)\le P_i(\theta)\right\}.
\end{equation}
The set $\mathcal R_i$ is compact. For any $\epsilon\in(0,1)$,  $\mathcal R_i$ can be covered by a finite collection of open balls $\{\widering B(\rho^{(m)},\epsilon)\}_{1\le m\le M(\epsilon)}$. Moreover, if $\epsilon\le \min \{(\#\mathcal A_j)^{-1}:\; 1\le j\le k\}$, we can assume that for all $m$ the components of each probability vector $\rho^{(m)}$ are not smaller than $\epsilon/2$.

For $x\in X_1$, $2\le j\le k$, and $n\in\mathbb N^*$ we define  $\rho_j(x, n)$ to be  the Bernoulli product measure on $X_j$ associated with the probability vector whose components are the frequences of occurrence of the different elements of $\mathcal A_j$ in $\Pi_j(x_{|n})$, namely the vector $$\left(n^{-1}\#\{1\le m\le n:\, \Pi_j(x_m)=b\right)_{b\in \mathcal A_j}.$$
 Also, let
\begin{equation}\label{rhoxn}
\rho(x,n)= (\widetilde \rho_i,\rho_{i},\ldots,\rho_k),
\end{equation}
where
\begin{align*}
\widetilde \rho_i&=\rho_i\left(x, \lceil \theta \ell_i(n)\rceil\right ),\ \rho_i=\rho_i(x,\ell_{i}(n)),\\
\rho_j&=\rho_{j}\left(T_1^{\ell_{j-1}(n)}x,\ell_{j}(n)-\ell_{j-1}(n)\right), \quad j=i+1, \ldots,  k.
\end{align*}

Now, for any $n\in\mathbb N^*$ and $U=(U_i,\ldots,U_k)\in \mathcal A_i^{\ell_i(n)}\times \prod_{j=i+1}^k \mathcal A_j^{\ell_j(n)-\ell_{j-1}(n)}$, we can define $\rho (U)= (\widetilde \rho_i(U),\rho_{i}(U),\ldots,\rho_k(U))$ as equal to $\rho(x,n)$, for any $x\in X_1$ such that $\Pi_i(x_1\cdots x_{\ell_i(n)})=U_i$ and $\Pi_j(x_{\ell_{j-1}(n)+1}\cdots x_{\ell_{j}(n)})=U_j$ for all $i+1\le j\le k$. Note that $\rho_i(U)$ depends on $U_i$ only, so we also denote  it by $\rho_i(U_i)$.

Then, for each $1\le m\le M(\epsilon)$ and $n\in\mathbb N^*$ we set
$$
\mathcal R_i(\epsilon, m, n)=\left\{U\in   \mathcal A_i^{\ell_i(n)}\times \prod_{j=i+1}^k \mathcal A_j^{\ell_j(n)-\ell_{j-1}(n)}: \rho (U)\in \widering B(\rho^{(m)},\epsilon)\right \}.
$$
It is standard to observe that if $U_i\in  \mathcal A_i^{\ell_i(n)}$ is such that $|\rho_i(U_i)([b])-\rho^{(m)}_i([b])|\le \epsilon$ for all $b\in \mathcal A_i$ then
\begin{align*}
\rho_i^{(m)}([U_i])=\prod_{b\in \mathcal A_i} \rho_i^{(m)}([b])^{\ell_i(n) \rho_i(U)([b])}
&\ge \prod_{b\in \mathcal A_i} \rho_i^{(m)}([b])^{\ell_i(n) \rho^{(m)}_i([b])} \prod_{b\in \mathcal A_i} \rho_i^{(m)}([b])^{\ell_i(n) \epsilon}\\
&\ge \exp \left (- \ell_i(n) \left(h_{\rho_i^{(m)}}(T_i)+   \epsilon\log(2/\epsilon)\right)\right ),
\end{align*}
Consequently, the cardinality of the set $\mathcal U_{i,\epsilon, m, n}$ of such $U_i$ is bounded from above by $\exp \big ( \ell_i(n) \big(h_{\rho_i^{(m)}}(T_i)+  \epsilon\log(1/\epsilon)\big)\big )$.

Similarly, for each $i+1\le j\le k$,  the cardinality of the set $\mathcal U_{j,\epsilon, m, n}$ of  those $U_j\in \mathcal A_j^{\ell_j(n)-\ell_{j-1}(n)}$ such that $|\rho_j(U_j)([b])-\rho^{(m)}_j([b])|\le \epsilon$ for all $b\in \mathcal A_j$ is bounded from above  by  $\exp \big ( (\ell_j(n)-\ell_{j-1}(n)) (h_{\rho_j^{(m)}}(T_j)+\epsilon \log(2/\epsilon))\big )$. Since by definition of $\mathcal R_i(\epsilon, m, n)$ we have $\mathcal R_i(\epsilon, m, n)\subset \prod_{i=j}^k \mathcal U_{j,\epsilon, m, n}$, the previous observations yield
\begin{align}\nonumber &\#\mathcal R_i(\epsilon, m, n)\\
\nonumber&\le \prod_{j=i}^{k}(\#\mathcal U_{j,\epsilon, m, n})\\
\label{cardRim}&\le \exp \big (\ell_k(n)\epsilon  \log(2/\epsilon)\big ) \exp \Big (\ell_i(n) h_{\rho_i^{(m)}}(T_i)+\sum_{j=i+1}^k (\ell_j(n)-\ell_{j-1}(n))h_{\rho_j^{(m)}}(T_j)\Big ).
\end{align}

We also notice that if we endow $X_i$ with the metric
\begin{equation*}
d_{\vec{\boldsymbol\gamma},i}(x,y)=\max\left(e^{-\frac{|\Pi_{i,j}(x)\land \Pi_{i,j}(y)|}{\gamma_1+\cdots+\gamma_{j}}}:\; j=i, \ldots,  k\right),
\end{equation*}
the balls of radius $e^{-\frac{n}{\gamma_1}}$ in $X_1$ project to the balls of the same radius in $X_i$, which are  parametrized by the elements of $\mathcal A_i^{\ell_i(n)}\times \prod_{j=i+1}^k \mathcal A_j^{\ell_j(n)-\ell_{j-1}(n)}$, in the sense that such a ball takes the form
$$
B_U=\left\{y\in X_i: \; y_1\cdots y_{\ell_i(n)}=U_i, \; \Pi_{i,j}(y_{\ell_{j-1}(n)+1}\cdots y_{\ell_j(n)}) =U_j \mbox{ for } j= i+1,\ldots, k\right\}
$$
for some $U$ in $\mathcal A_i^{\ell_i(n)}\times \prod_{j=i+1}^k \mathcal A_j^{\ell_j(n)-\ell_{j-1}(n)}$. Moreover, given such a ball $B_U$, $\Pi_i^{-1}(B_U)\cap K$ is covered by, say, a family $\mathcal B(U)$ of $n_U$ balls of radius $e^{-\frac{n}{\gamma_1}}$ which intersect $K$. Each of the $N^{(i)}_{{U_i}_{|\ell_{i-1}(n)}}$ cylinders of generation $\ell_{i-1}(n)$ in $X_1$ which intersects $K$ and project to $[{U_i}_{|\ell_{i-1}(n)}]$ in $X_i$ via $\Pi_i$ intersects only one  such ball.  Indeed, for such a cylinder $[V_1\cdots V_{\ell_{i-1}}(n)]$, the data $\Pi_j([V_{\ell_{j-1}(n)+1}\cdots V_{\ell_{j}(n)}])$, $1\le j\le i-1$, and $B_U$ determine a unique ball $B$ of $X_1$ such that $\Pi_i(B)=B_U$.  This implies $n_U\le N^{(i)}_{{U_i}_{|\ell_{i-1}(n)}}$. Consequently, for every integer $\ell$ between $\ell_{i-1}(n)$ and $\ell_i(n)$, we also have $n_U\le N^{(i)}_{{U_i}_{|\ell}}$. In particular,
\begin{equation}\label{nU}
n_U\le N^{(i)}_{{U_i}_{|\lceil \theta \ell_i(n)\rceil}}.
\end{equation}

The following lemma, whose proof we postpone  to the end of this section, will play an essential role to find effective coverings of $\Pi_i(\widetilde X_1)$, and then of $K$. Let us mention at the moment that in this lemma $(1)\Rightarrow (2)\Rightarrow (3)$. ALso, recall the definition \eqref{Ri} of $\mathcal R_i$.
\begin{lem}\label{comblem} For all $x\in \widetilde X_1$:
\begin{enumerate}
\item $\liminf_{n\to\infty} \widetilde D_{i,\theta} (\rho(x,n))-D_{i,\theta}(\rho_i(x,n))\le 0$;

\item  $\liminf_{n\to\infty} \widetilde D_{i,\theta} (\rho(x,n))\le P_i(\theta)$;

\item  there exists $\rho\in \mathcal R_i$ and an increasing sequence of integers $(n_j)_{j\in \mathbb N}$ such that $\rho(x,n_j)$ converges to $\rho$ as $j\to\infty$.
\end{enumerate}
\end{lem}

\begin{proof}[Proof of Theorem~\ref{VPUB}] It follows from Lemma~\ref{comblem}(3) that given $\epsilon>0$, for all $x\in \widetilde X_1$, there exists $1\le m\le M(\epsilon)$ such that  $\Pi_i(x)$ belongs to $\bigcup_{U\in\mathcal R_i(\epsilon, m, n)}B_U$ for infinitely many integers $n$. As a result,  for all $N\in\mathbb N^*$,  we get the following  covering of $K$:
$$
K\subset \bigcup_{n\ge N} \bigcup_{m=1}^{M(\epsilon)} \bigcup_{U\in\mathcal R_i(\epsilon, m, n)}\bigcup_{B\in \mathcal B(U)} B.
$$
Thus, given $s>0$, the pre-Hausdorff measure $\mathcal H^s_{e^{-\frac{n}{\gamma_1}}}$ of $K$ is bounded as follows:
$$
 \mathcal H^s_{e^{-\frac{N}{\gamma_1}}}(K) \le \sum_{n\ge N} \sum_{m=1}^{M(\epsilon)}\sum_{U\in\mathcal R_i(\epsilon, m, n)} N^{(i)}_{{U_i}_{|\lceil \theta \ell_i(n)\rceil}} e^{-\frac{n}{\gamma_1}s}.
 $$
Consequently, denoting by $(U_i)_\ell$ the $\ell$-th letter of $U_i$,
\begin{align*}
\mathbb{E}\Big ( \mathcal H^s_{e^{-\frac{N}{\gamma_1}}}(K)\Big )& \le \sum_{n\ge N} e^{-\frac{n}{\gamma_1}s} \sum_{m=1}^{M(\epsilon)}\sum_{U\in\mathcal R_i(\epsilon, m, n)} \mathbb{E}(N^{(i)}_{{U_i}_{|\lceil \theta \ell_i(n)\rceil}}) \\
&=\sum_{n\ge N} e^{-\frac{n}{\gamma_1}s} \sum_{m=1}^{M(\epsilon)}\sum_{U\in\mathcal R_i(\epsilon, m, n)}\prod_{\ell=1}^{\lceil \theta \ell_i(n)\rceil} \mathbb{E} (N^{(i)}_{(U_i)_\ell})
\end{align*}
and using the definition of $\widetilde\rho_i(U)$ to re-express the right hand side of the last inequality, we obtain
$$
\mathbb{E}\Big ( \mathcal H^s_{e^{-\frac{N}{\gamma_1}}}(K)\Big )\le \sum_{n\ge N} e^{-\frac{n}{\gamma_1}s} \sum_{m=1}^{M(\epsilon)}\sum_{U\in\mathcal R_i(\epsilon, m, n)}\exp \Big (\lceil \theta \ell_i(n)\rceil \sum_{b\in\mathcal A_i} \widetilde \rho_i(U)([b])\log \mathbb E(N^{(i)}_b) \Big ).
$$
Now we use the fact that $U\in \mathcal R_i(\epsilon, m, n)$ means that $d(\rho(U), \rho^{(m)})\le \epsilon$, to get a constant $C_i$ independent of $m$, $U$ and $n$ such that
$$
\sum_{b\in\mathcal A_i} \widetilde \rho_i(U)([b])\log \mathbb E(N^{(i)}_b)\le  C_i\epsilon+ \sum_{b\in\mathcal A_i} \widetilde \rho_i^{(m)}([b])\log \mathbb E(N^{(i)}_b).
$$
We then obtain:
\begin{align*}
\mathbb{E}\Big ( \mathcal H^s_{e^{-\frac{N}{\gamma_1}}}(K)\Big )&\le \sum_{n\ge N} e^{-\frac{n}{\gamma_1}s} e^{\lceil \theta \ell_i(n)\rceil C_i\epsilon}\\&\quad \quad\quad\quad\quad\quad\quad\cdot   \sum_{m=1}^{M(\epsilon)} (\#\mathcal R_i(\epsilon, m, n)) \exp \big (\lceil \theta \ell_i(n)\rceil \sum_{b\in\mathcal A_i} \widetilde \rho^{(m)}_i([b])\log \mathbb E(N^{(i)}_b) \big ).
\end{align*}
Using \eqref{cardRim}, the fact that $|\ell_j(n)-\frac{\gamma_1+\ldots+\gamma_j}{\gamma_1}n|\le 1$ for all $1\le j\le k$, as well as the definition of $\widetilde D_{i,\theta}(\rho^{(m)})$, we deduce that there exists a constant $\widehat C_i$ such that for all $1\le m\le M(\epsilon)$:
\begin{align*}
&(\#\mathcal R_i(\epsilon, m, n)) \exp \big (\lceil \theta \ell_i(n)\rceil \sum_{b\in\mathcal A_i} \widetilde \rho^{(m)}_i([b])\log \mathbb E(N^{(i)}_b) \big )  \\&\le \widehat C_i \exp \big (\ell_k(n)\epsilon \log(2/\epsilon)\big ) \exp \Big (\frac{n}{\gamma_1} \widetilde D_{i,\theta}(\rho^{(m)})\Big )\\
&\le \widehat C_i \exp \big (\ell_k(n)\epsilon  \log(2/\epsilon)\big ) \exp \Big (\frac{n}{\gamma_1}P_i(\theta)\Big ) \quad(\text{recall that $\rho^{(m)}\in\mathcal R_i$}).
\end{align*}
Upon taking $C_i=\widehat C_i$ big enough, we conclude that
$$
\mathbb{E}\Big ( \mathcal H^s_{e^{-\frac{N}{\gamma_1}}}(K)\Big )\le C_iM(\epsilon)\sum_{n\ge N} \exp \Big (-\frac{n}{\gamma_1}\big (s- P_i(\theta)-C_i\epsilon \log(2/\epsilon)\big)\Big ).
$$
If $s>P_i(\theta)+C_i\epsilon \log(1/\epsilon)$, this yields
$
\mathbb{E}\Big (\sum_{N\ge 1} \mathcal H^s_{e^{-\frac{N}{\gamma_1}}}(K)\Big )<\infty$, so
 $\lim_{N\to\infty} \mathcal H^s_{e^{-\frac{N}{\gamma_1}}}(K)=0$ and $\dim_H K\le s$ almost surely. Since this holds for any  given small enough $\epsilon>0$ , we get $\dim_H K\le P_i(\theta)$ almost surely.

The previous upper bound is easily seen to hold simultaneously for all $2\le i\le k$ and $\widetilde\theta_i\le  \theta\le 1$ since its holds simultaneously for all $2\le i\le k$ and rational $\widetilde\theta_i\le \theta\le 1$, and the mappings $\theta\mapsto P_i(\theta)$ are continuous. This yields Theorem~\ref{VPUB}.\end{proof}

\begin{proof}[Proof of Lemma~\ref{comblem}] That $(1)\Rightarrow (2) $ follows from the fact that $P_i(\theta)=\max\{D_{i,\theta}(\nu_i):\nu_i\in\mathcal B_i\}$, and $(2)\Rightarrow (3) $ is immediate.

Let $x\in X_1$. To prove (1), we are going to show that there exists $J\in\mathbb N^*$, as well as $J$ bounded sequences $u_j:\mathbb N^*\to \R$  such that $\lim_{n\to\infty}u_j(n+1)- u_j(n)=0$, and $J$ couples $(\alpha_j,\beta_j)\in \mathbb R_+^*$ such that for all $n\in\mathbb N^*$,
\begin{equation}\label{rewriting}
\widetilde D_{i,\theta} (\rho(x,n))-D_{i,\theta}(\rho_i(x,n))\le \epsilon_n+\sum_{j=1}^J u_j(\lceil \beta_j n\rceil)-u_j(\lceil \alpha_j n\rceil),
\end{equation}
with $\lim_{n\to\infty} \varepsilon_n=0$. The desired conclusion is then a direct application of \cite[Lemma 5.4]{FH}, which is a slight extension of the combinatorial lemma used by Kenyon and Peres \cite[Lemma 4.1]{KP}.

To prove \eqref{rewriting},  noting that ${\Pi_{i,j}}\rho_i(x,n)= \rho_j(x,n)$ for all $i\le j\le k$, and using the respective definitions of $\widetilde D_{i,\theta}$ and $D_{i,\theta}$, we can write, after defining the sequences $v_i(n)= \widetilde \gamma_{i} \theta \sum_{b\in\mathcal A_i} \rho(x,n)([b])\log \mathbb E(N^{(i)}_b)$ and $w_j(n)={\gamma}_{j} h_{\rho_j (x,n)}(T_j)$:
\begin{align*}
\widetilde D_{i,\theta} (\rho(x,n))-D_{i,\theta}(\rho_i(x,n))
&=v_i(\lceil \theta \ell_i(n)\rceil)-v_i(n)+\sum_{j=i}^k w_j(\ell_j(n))-w_j(n)\\
&\quad+\sum_{j=i+1}^k {\gamma}_{j}(h_{\rho_j (T_1^{\ell_{j-1}(n)}x,\ell_j(n)-\ell_{j-1}(n))}(T_j)-h_{\rho_j (x,\ell_j(n))}(T_j)).
\end{align*}
Note that each $u\in \{v_i, w_i,\ldots,w_k\}$ is bounded and  does satisfy  $\lim_{n\to\infty}u(n+1)- u(n)=0$. Also, we have
\begin{align}
\nonumber& h_{\rho_j (T_1^{\ell_{j-1}(n)}x,\ell_j(n)-\ell_{j-1}(n))}(T_j)-h_{\rho_j (x,\ell_j(n))}(T_j)\\
\label{comph}&\le \frac{\ell_{j-1}(n)}{\ell_j(n)-\ell_{j-1}(n)} (h_{\rho_j (x,\ell_j(n))}(T_j)-h_{\rho_j (x,\ell_{j-1}(n))}(T_j)).
\end{align}
To see this, note that denoting $\rho_j (x,\ell_j(n))$, $\rho_j (x,\ell_{j-1}(n))$ and $\rho_j (T_1^{\ell_{j-1}(n)}x,\ell_j(n)-\ell_{j-1}(n))$ by $\nu_j$, $\nu'_j$ and $\nu_j''$ respectively,  \eqref{comph} is equivalent to
\begin{equation}\label{comph2}
(\ell_j(n)-\ell_{j-1}(n))\, h_{\nu_j''}(T_j)+\ell_{j-1}(n) \,h_{\nu'_j}(T_j)\le \ell_j(n) \,h_{\nu_j}(T_j).
\end{equation}
However, by definition, for all $b\in\mathcal A_j$ we have
$$
\nu_j([b])= \frac{\ell_{j-1}(n)}{\ell_{j}(n)} \nu'_j ([b])+\frac{\ell_{j}(n)-\ell_{j-1}(n)}{\ell_{j}(n)} \nu''_j ([b]).
$$
Thus, \eqref{comph2}, and  then \eqref{comph}, follow from the concavity of $x\ge 0\mapsto -x\log(x)$ and the fact that   $h_\nu(T_j)=-\sum_{b\in \mathcal A_j} \nu([b])\log([b])$ for $\nu\in\{\nu_j, \nu'_j,\nu_j''\}$.

Setting $\alpha_j= \frac{\widetilde \gamma_{j}}{\gamma_1}$,  \eqref{comph} implies that
\begin{align*}
&{\gamma}_{j}(h_{\rho_j (T_1^{\ell_{j-1}(n)}x,\ell_j(n)-\ell_{j-1}(n))}(T_j)-h_{\rho_j (x,\ell_j(n))}(T_j))\\
&\le \frac{\widetilde \gamma_{j-1}}{\gamma_j} (w_j(\lceil \alpha_j n\rceil)-w_j(\lceil \alpha_{j-1} n\rceil)) +o(1).
\end{align*}
Moreover, $v_i(\lceil \theta \ell_i(n)\rceil)-v_i(n)= v_i(\lceil \theta \alpha_i n\rceil)-v_i(n)+o(1)$. Finally \eqref{rewriting} holds.
\end{proof}

\section{The box counting dimension of $K$. Proofs of Theorem~\ref{thmdimB} and Corollary~\ref{cor1}}

Recall that for all $u\in\N^*$ and  $u\in \mathcal A_1^*$, $(K^u,K^u_n)$ denotes the copie of $(K,K_n)$ generated by the random sets $(A(u\cdot a))_{a\in\mathcal A_1^*}$ (see \eqref{Ka}).

\begin{proof}[Proof of Theorem~\ref{thmdimB}] Here again, without loss of generality we assume that all the $\gamma_i$ are positive.

We will use in an essential way the result established in \cite[Section 4]{GL94}, which deals with the case where $k=2$, $m_1=e^{-\gamma_1}$ and $m_2=e^{-(\gamma_1+\gamma_2)}$ are integers, and with the Euclidean realisation of $K$.  It is worth noting that this result is strongly based on a result by Dekking on the asymptotic behaviour of the survival probability of a branching process in a random environment \cite{Dek}.

We first need to describe the balls of radius $e^{-\frac{n}{\gamma_1}}$ which intersect~$K$. For $n\in\mathbb N^*$, we saw that the set $\mathcal F_n$ of balls in $X_1$ of radius $e^{-\frac{n}{\gamma_1}}$ equals the set $\{B_U:U=(U_1,\ldots, U_k)\in \prod_{i=1}^k\mathcal A_i^{\ell_i(n)-\ell_{i-1}(n)}\}$, where
$$
B_U=\{y\in X_1:\, \Pi_i(T_1^{\ell_{i-1}(n)}(y))_{|\ell_i(n)-\ell_{i-1}(n)}=U_i, \ \forall\, 1\le i\le k\}.
$$
Thus $B_U\cap K\neq\emptyset$ if and only if the  event
$$
E_U=\left\{\exists\, (u_i)_{1\le i\le k}\in \prod_{i=1}^k \mathcal A_1^{\ell_i(n)-\ell_{i-1}(n)}: \ \text{both }F^{U}_k(u_1,\ldots, u_k)\text{ and } K^{u_1u_2\cdots u_k}\neq\emptyset\text{ hold}\right\}
$$
holds, where for all $1\le i\le k$
$$
F^{(U_1,U_2,\ldots,U_i)}_i(u_1,\ldots, u_i)=\left\{\begin{array}{cl}
 &\Pi_j(u_j)=U_j, \ \forall\, 1\le j\le i,\\
 &[u_j]\cap K_{\ell_j(n)-\ell_{j-1}(n)}^{u_1u_2\cdots u_{j-1}}\neq\emptyset , \ \forall\,  1\le j\le i\\
 \end{array}\right\}
$$
(note that necessarily $u_1=U_1$).

For $2\le i\le k$ and $(U_1,\ldots, U_{i})\in \prod_{j=1}^{i}\mathcal A_j^{\ell_j(n)-\ell_{j-1}(n)}$, we set
$$
E_i(U_1,\ldots, U_i)=\left \{\exists\, (u_j)_{1\le j\le i}\in \prod_{j=1}^i \mathcal A_j^{\ell_j(n)-\ell_{j-1}(n)}: \ F^{(U_1,U_2,\ldots,U_i)}_i(u_1,\cdots, u_i)\text{ holds }\right \}
$$
(note that $E_1(U_1)$ is simply the event $\{[U_1]\cap K_n\neq\emptyset\}$). We deduce from \cite[Section 4]{GL94} that conditional on $K\neq\emptyset$, we have
$$
\lim_{n\to\infty}\frac{ \log\#\{(U_1,U_2)\in \mathcal A_1^n\times \mathcal A_2^{\ell_2(n)-n}:\, E_2(U_1,U_2)\text{ holds}\}}{n}= \log (\mathbb{E}(\# A))+\frac{\gamma_2}{\gamma_1}\psi_2(\widehat \theta_2).
$$
This result mainly comes from the fact that $\lim_{n\to\infty}\frac{ \log\#\{U_1\in \mathcal A_1^n:\,  E_1(U_1)\text{ holds}\}}{n}= \log (\mathbb{E}(\# A))>0$, and given $U_1\in \mathcal A_1^n$ such that $[U_1]\cap K_n\neq\emptyset$, the number of those $U_2\in  \mathcal A_2^{\ell_2(n)-n}$ such that $E_2(U_1,U_2)$ holds is a random variable  $Z_{2,\ell_2(n)-n}(U_1)$,  so that the random variables $Z_{2,\ell_2(n)-n} (U_1)$ are independent and identically distributed,  $\lim_{n\to\infty}\frac{ \log\mathbb{E}(Z_{2,\ell_2(n)-n})(U_1)}{n}= \frac{\gamma_2}{\gamma_1}\psi_2(\widehat \theta_2)>0$ and, conditional on $K^{U_1}\neq\emptyset$, $\lim_{n\to\infty}\frac{ \log Z_{2,\ell_2(n)-n}(U_1)}{n}= \frac{\gamma_2}{\gamma_1}\psi_2(\widehat \theta_2)$ almost surely.

Now for $2\le i\le k$ set
$$
s_i= \log (\mathbb{E}(\# A))+\sum_{j=2}^{i}\frac{\gamma_i}{\gamma_1}\psi_j(\widehat \theta_j).
$$
Suppose that $k\ge 3$, and for some $3\le i\le k$ we have proven that conditional on $K\neq\emptyset$, it holds that
\begin{equation}\label{box}
\lim_{n\to\infty} \frac{\displaystyle \log \#\{(U_1,\ldots, U_{i-1})\in \prod_{j=1}^{i-1}\mathcal A_j^{\ell_j(n)-\ell_{j-1}(n)}: \, E_{i-1}(U_1,\ldots, U_{i-1}) \text{ holds}\}}{n}= s_{i-1}.
\end{equation}
Given $(U_1,\ldots, U_{i-1})\in \prod_{j=1}^{i-1}\mathcal A_j^{\ell_j(n)-\ell_{j-1}(n)}$, and fixed  associated $(u_1,\ldots,u_{i-1})$ such that $F^{(U_1,U_2,\ldots,U_{i-1})}_{i-1}(u_1,\cdots, u_{i-1})$ holds, following the arguments of \cite{GL94}, the cardinality of the set of those words $U_i\in  \mathcal A_i^{\ell_i(n)-\ell_{i-1}(n)}$ such that there exists $u_i\in  \mathcal A_1^{\ell_i(n)-\ell_{i-1}(n)}$ such that $F^{(U_1,U_2,\ldots,U_i)}_i(u_1,\cdots, u_i)$ holds, is a random variable $Z_{i,\ell_i(n)-\ell_{i-1}(n)}(U_1,\ldots,U_{i-1})$ so that the  $Z_{i,\ell_i(n)-\ell_{i-1}(n)}(U_1,\ldots,U_{i-1})$ are independent, and identically distributed. Moreover, setting $\widetilde\ell_i(n)=\ell_i(n)-\ell_{i-1}(n)$, one has both $\lim_{n\to\infty}\frac{ \log\mathbb{E}(Z_{i,\tilde\ell_i(n)})(U_1,\ldots,U_{i-1})}{n}= \frac{\gamma_i}{\gamma_1}\psi_i(\widehat \theta_i)>0$ and,  conditional on $K^{u_1\cdots u_{i-1}}\neq\emptyset$,  $\lim_{n\to\infty}\frac{ \log Z_{i,\tilde \ell_i(n)}(U_1,\ldots,U_{i-1})}{n}= \frac{\gamma_i}{\gamma_1}\psi_i(\widehat \theta_i)$ almost surely.  Then, again the same reasoning as in~\cite{GL94} for the case $k=2$ with the roles of $\mathcal A_1^n$ and $\mathcal A_2^{\ell_2(n)-n}$ now respectively played by $ \prod_{j=1}^{i-1}\mathcal A_j^{\ell_j(n)-\ell_{j-1}(n)}$ and $\mathcal A_i^{\ell_i(n)-\ell_{i-1}(n)}$ shows that \eqref{box} holds for $i$ as well. Consequently, applying this to $i=k$, conditional on $K\neq\emptyset$, we get for all $n\ge 1$ an integer  $\mathcal N_n$ such that $\lim_{n\to\infty}\frac{\log \mathcal N_n} {n}= s_k$,  as well as  $ \mathcal N_n$ elements $U=(U_1,\ldots,U_k)\in \prod_{i=1}^{k}\mathcal A_i^{\ell_i(n)-\ell_{i-1}(n)}$ and associated $(u_1=U_1,u_2,\ldots,u_k)\in  \prod_{i=1}^{k}\mathcal A_1^{\ell_i(n)-\ell_{i-1}(n)}$ such that $F^{(U_1,\ldots,U_k)}_k(u_1,\ldots,u_k)$ hold. The events $\{K^{u_1u_2\cdots u_k}\neq\emptyset\}$ being independent, with the same probability $\mathbb P(K\neq\emptyset)$, and independent of the events $F^{(U_1,U_2,\ldots,U_k)}_k(u_1,u_2,\ldots,u_k)$, using \cite[Section 4]{GL94} again yields
$$
\lim_{n\to\infty} \frac{\log\#\{U\in \prod_{i=1}^k\mathcal A_i^{\ell_i(n)-\ell_{i-1}(n)}: \, E_U\text{ holds}\}}{n}=\log (\mathbb{E}(\# A))+\sum_{i=2}^{k}\frac{\gamma_i}{\gamma_1}\psi_i(\widehat \theta_i),
$$
which, after dividing  by $\gamma_1^{-1}$, is precisely $\lim_{n\to\infty} \frac{\log \#\{B\in\mathcal F_n: B\cap K\neq\emptyset\}}{-\log(e^{-\frac{n}{\gamma_1}})}$, i.e. $\dim_B K$.
\end{proof}
Next we state, using  our notation, a fact established in the proof of \cite[Corollary 3.5]{BF2016}, which is a variational approach to the dimension of projections of fractal percolation sets in a symbolic space $X_1\times X_1$ to one of its two natural factors.

\begin{pro}\label{prop4.1} Let $2\le i\le k$. With probability 1, conditional on $K\neq\emptyset$,
$$
\max\Big\{ \min (\dim_e(\mu),h_{\nu_i}(T_i)) : \mu\text{ is a Mandelbrot measure supported on $K$}\Big\}= \psi_i(\widehat\theta_i),
$$
where $\nu_i$ stands for the expectation of ${\Pi_i}(\mu)$. Moreover the maximum is uniquely reached if and only if $\widehat \theta_i >0$ or $\widehat \theta_i =0$ and $\psi_i'(\widehat \theta_i)=0$. In any case, when the maximum is reached, one has $\nu_i=\nu_{i,\widehat\theta_i}$, where
$$
\nu_{i,\theta}([b])=\mathbb{E}(N_b^{(i)})^\theta/\sum_{b'\in\mathcal A_i}\mathbb{E}(N_{b'}^{(i)})^\theta .
$$
Also, if $\widehat\theta_i>0$, or $\widehat \theta_i =0$ and $\psi_i'(\widehat \theta_i)=0$, then $\dim_e(\mu)=h_{\nu_i}(T_i)$ for the unique $\mu$ at which the maximum is reached, and if $\widehat \theta_i =0$ and $\psi_i'(\widehat \theta_i)>0$, then  $\dim_e(\mu)>h_{\nu_i}(T_i)$ for $\mu$ at which the maximum is reached.
\end{pro}

\begin{proof}[Proof of Corollary~\ref{cor1}] It follows from \eqref{maj dime mu} and Proposition~\ref{prop4.1} Êthat  for $\dim_H K=\dim_B(K)$ to hold almost surely, conditional on $K\neq\emptyset$, the Mandelbrot measure $\mu$ of maximal dimension supported on $K$ must satisfy $\dim_e(\mu)=\log \mathbb E(\# A)$ and $\min (\dim_e(\mu),h_{\nu_i}(T_i))= \psi_i(\widehat \theta_i)$ for all $i\in I$.

The condition  $\dim_e(\mu)=\log \mathbb E(\# A)$ implies that $\mu$ is the so called branching measure, i.e. it is obtained from the random vector $(\mathbf{1}_{A}(a) /\mathbb E(\# A))_{a\in\mathcal A_1}$. The other condition implies that for all $2\le i\le k$, we have $\nu_i=\nu_{i,\widehat\theta_i}$. Since $\mu$ is the branching measure, this implies that $\mathbb E(N_b^{(i)})^{\widehat \theta_i}/ e^{\psi_i(\widehat \theta_i)}=\mathbb E(N_b^{(i)})/\mathbb E(\# A)$ for all $b\in \widetilde{\mathcal A}_i$, hence $\mathbb E(N_b^{(i)})^{\widehat \theta_i -1}$ does not depend on $b\in \widetilde{\mathcal A}_i$. This is a non trivial condition only if $\widehat\theta_i<1$. This proves the necessity of the condition given in the statement.

Now assume that $\mathbb E(N_b^{(i)})$ does not depend on $b\in \widetilde{\mathcal A}_i$ for all $i\in I$ such that $\widehat\theta_i<1$. Suppose first that there is no $i\in I$ such that $\widehat\theta_i<1$, i.e. $\widehat\theta_i=1$ for all $i\in I$. By the remark made above, the branching measure $\mu$ does satisfy $\dim_e^{\vec{\boldsymbol\gamma}}(\mu)=(\gamma_1+\cdots+\gamma_k)\dim_e(\mu)= \dim_B(K)$.  Next, suppose that $\widehat\theta_i<1$ for some $i\in I$. Again, consider the branching measure $\mu$. Since $\mathbb E(N_b^{(i)})$ does not depend on $b\in \widetilde{\mathcal A}_i$ we do have  $\nu_i=\nu_{i,\widehat\theta_i}$, so that $\min (\dim_e(\mu),h_{\nu_i}(T_i))=h_{\nu_i}(T_i)=h_{\nu_{i,\widehat \theta_i}}(T_i)=\psi_i(\widehat\theta_i)$. This yields again $\dim^{\vec{\boldsymbol\gamma}}_e (\mu)=\dim_B K$.
\end{proof}

\section{Projections of $K$ and $\mu$ to factors of $X_1$. Proofs of Theorems~\ref{thmdim2},~\ref{dim PiK} and~\ref{dimB PiK}, and Corollary~\ref{cor2}}

The proofs will be sketched.

\begin{proof}[Sketch of the proof of  Theorem~\ref{thmdim2}] For all $n\ge 1$, denote by $\mathcal F_n^i$ the set of balls of $X_i$ of radius $e^{-\frac{n}{\gamma_1}}$. Let $j_0=\max\{i\le j\le k: T'(1)\le \mathcal T'_{\nu_j}(1)\}$, with the convention $\max (\emptyset)=i-1$. Computations similar to those used to prove Theorem~\ref{thmdim} yield  $q_0>1$ and $c_0\ge 0$ such that for all $q\in (0,q_0]$ we have
\begin{equation*}
\mathbb{E}\Big (\sum_{B\in\mathcal F_n^i} {\Pi_i}\mu(B)^q\Big )=O\big (\exp  (-t(j_0,q,n)) \big )\quad \text{as $n\to\infty$,}
\end{equation*}
where
$$
t(j_0,q,n)=
\begin{cases}
\ell_{j_0}(n)(T(q)-c_0(q-1)^2)+\sum_{j=j_0+1}^k (\ell_{j}(n)-\ell_{j-1}(n))\mathcal T_{\nu_j}(q)&\text{if } j_0\ge i\\
\ell_{i}(n) \mathcal T_{\nu_i}(q)+\sum_{j=i+1}^k (\ell_{j}(n)-\ell_{j-1}(n))\mathcal T_{\nu_j}(q)&\text{otherwise}.
\end{cases}
$$
This is enough to get the differentiability of $\tau_{{\Pi_i}\mu}$ at $1$ with  $\tau_{{\Pi_i}\mu}'(1)$ equal to $\dim^{{{\vec{\boldsymbol\gamma}}^i}}_e({\Pi_i}\mu)$, and conclude.
\end{proof}

\begin{proof}[Sketch of the proof of Theorem~\ref{dim PiK}] We start with the lower bound.

If we proceed as in Section~\ref{optim}, we have to consider
$$
M^{\vec{\boldsymbol\gamma}^i}=\max \left\{ \sum_{j=i}^k {\boldsymbol\gamma}^i_j\min\{ H(\nu)+\nu(\widetilde{\phi}), H(\Pi_i\nu)\}: \; \nu\in {\mathcal B}_1\right\}.
$$
Define $\widetilde g_i:\; \mathcal B_1\to \R$ ($\mathcal B_1$ being still identified with $\mathcal P(\widetilde{\mathcal A}_1)$) by
\begin{equation*}
\widetilde g_i(\nu)=(\gamma_1+\ldots+\gamma_i) H(\Pi_i\nu)+\sum_{j=i+1}^k \gamma_j H(\Pi_j \nu)=h^{\vec{\boldsymbol\gamma}^i}_{\Pi_i\nu}(T_i).
\end{equation*}
Note that $\widetilde g_i$ is strictly concave only when considered as a function of $\Pi_i\nu$.

We have, remembering the definition \eqref{gi} of the mappings $g_j$,
$$
M^{\vec{\boldsymbol\gamma}^i}=\sup_{\nu\in \mathcal B_1} \min (\widetilde g_i(\nu),g_i (\nu),\ldots, g_k(\nu)).
$$
To express this supremum in terms of the pressure functions $P_j(\cdot)$, $i\le j\le k$,  we have to consider the mapping  $\widetilde P:\Delta_{k-i+2}\to\R$ defined by
$$
\widetilde P(\widetilde q_i,q_i,\ldots, q_k)=\sup_{\nu\in \mathcal B_1}\widetilde q_i \widetilde g_i(\nu)+q_ig_i(\nu)+\ldots+q_k g_k(\nu).
$$
Corollary \ref{cor-1} yields
$$M^{\vec{\boldsymbol\gamma}^i}=\min_{(\widetilde q_i,q_i,\ldots, q_k)\in \Delta_{k-i+2}} \widetilde P(\widetilde q_i,q_i,\ldots, q_k).
$$
Let $q^*=(\widetilde q_i^*,q_i^*,\ldots, q_k^*)$ at which the minimum is attained, and $\nu\in\mathcal B_1$ such that $\widetilde P(q^*)=\widetilde q_i^* \widetilde g_i(\nu)+q_i^*g_i(\nu)+\ldots+q_k ^*g_k(\nu)$. Due to Remark~\ref{rem2}, one has
$$\widetilde P(q^*)=\sup_{\nu_i\mathcal B_i} \widetilde q_i ^*\widetilde f_i(\nu_i)+q_i^*f_i(\nu_i)+\ldots+q_k^* {\color{red} f}_k(\nu_i),$$
where $\widetilde f_i(\nu_i)=h^{\vec{\boldsymbol\gamma}^i}_{\nu_i}(T_i)$ and
$$
f_p(\nu_i)= (\gamma_1+\cdots+\gamma_p)(H(\nu_i)+\nu_i(\widetilde \phi_i))+\sum_{j=p+1}^k \gamma_j H(\Pi_{i,j}\nu_i),\quad p=i,i+1,\ldots, k.
$$

Moreover, the supremum is reached at a unique element $\nu_i$ of $\mathcal B_i$, which is fully supported. As in the minimization of $P(q_1,\ldots,q_k)$ in the proof of Theorem~\ref{optim}, either one of the following 3 cases occurs: (i) $H(\nu_i)+\nu_i(\widetilde{\phi}_i)\geq H(\nu_i)$; (ii) there exists a unique $j\in \{i,\ldots, k-1\}$ so that $H(\Pi_{i,j+1}\nu_i)\leq H(\nu_i)+\nu_i(\widetilde{\phi}_i)< H(\Pi_{i,j}\nu_i)$; (iii) $H(\nu_i)+\nu_i(\widetilde{\phi}_i)< H(\Pi_{i,k}\nu_i)$. In cases $(ii)$ and $(iii)$, we necessarily have $\widetilde q_i^*=0$, and we conclude by using the same approach as in the study of $P$ that either there exists $j'\in \{j,\ldots, k-1\}$ and $(q_{j'}, q_{j'+1})\in\Delta_2$ such that $\widetilde P(\widetilde q_i^*,q_i^*,\ldots, q_k^*)=\widetilde P(\underbrace{0,\ldots, 0}_{ (j'-i+1)\text{ times}}, q_{j'}, q_{j'+1},0,\ldots,0)$, or $q^*=(0,\ldots,0,1)$. Also, if $j'>i$, then 
$$
\widetilde P(\underbrace{0,\ldots, 0}_{ (j'-i+1)\text{ times}}, q_{j'}, q_{j'+1},0,\ldots,0)=P(\underbrace{0,\ldots, 0}_{(j'-i+1)\text{ times}}, q_{j'}, q_{j'+1},0,\ldots,0).
$$ 
Now, suppose that we are in case $(i)$. Let $j$ be the smallest index $j'$ in $\{i,\ldots, k\}$ such that $f_{j'}(\nu_i)>\widetilde P(\widetilde q_i^*,q_i^*,\ldots, q_k^*)$, and let $j=k+1$ otherwise. We have $q^*_{j'}=0$ for all $j\le j'\le k$. If $j\le i+1$, we get $\widetilde P(\widetilde q_i^*,q_i^*,\ldots, q_k^*)=\widetilde P(\widetilde q_i^*,q_i^*, 0,\ldots,0)$. If $j>i+1$, then $\pi_{j'}$ is injective for all $i\le j'\le j-2$. Set $\widetilde\gamma_s=\gamma_1+\cdot+\gamma_s$.  For each $(\widetilde q_i,q_i,\ldots,q_{j-1})\in \Delta_{j-i+1}$, we have
$$
\widetilde P(\widetilde q_i,q_i,\ldots, q_{j-1},0,\ldots, 0)
=\sup_{\mu_i\in \mathcal B_i} \big(\sum_{ t=i}^{j-1}q_t\widetilde\gamma_t\big) \mu_i(\phi_i)+ \widetilde \gamma_iH(\mu_i)+\sum_{s=j}^k\gamma_s H(\Pi_{i,s}\mu_i),
$$
where we have used that $H(\mu_i)=H(\Pi_{i,j'}\mu_i)$ for all $i< j'\le j-1$ since $\pi_{j'}$ is injective for all $i\le j'\le j-2$. Thus $\widetilde P(\widetilde q_i,q_i,\ldots, q_{j-1},0,\ldots, 0)$ depends only on $\sum_{t=i}^{j-1}q_t\widetilde\gamma_t$. Set  $\beta:=\sum_{ t=i}^{j-1}q_t^*\widetilde\gamma_t\in [0, \widetilde\gamma_{j-1}]$. Either $\beta\in [0, \widetilde\gamma_i]$ or $\beta\in [\widetilde\gamma_{j'}, \widetilde\gamma_{j'+1}]$ for some $i\leq j'\leq j-2$.  If the first case occurs, then there exists $(\widetilde q_i,q_i)\in \Delta_2$ such that $q_i\widetilde\gamma_i=\beta$, so
$\widetilde P(\widetilde q_i^*,q_i^*,\ldots, q_k^*)=\widetilde P(\widetilde q_i,q_i, 0,\ldots,0)$. If the second case occurs, then there exists $(q_{j'}, q_{j'+1})\in \Delta_2$ such that $\beta=q_{j'}\widetilde\gamma_{j'}+q_{j'+1}\widetilde\gamma_{j'+1}$, so
$$
\widetilde P(\widetilde q_i^*,q_i^*,\ldots, q_k^*)=\widetilde P(\underbrace{0,\ldots, 0}_{(j'-i+1) \text{ times}}, q_{j'},q_{j'+1}, 0,\ldots,0).
$$

Then, noting that $\inf_{t\in[0,1]}\widetilde P(t,1-t,0,\ldots,0)=\inf_{t\in[0,1]}P_i(t)$, we get
$$
M^{\vec{\boldsymbol\gamma}^i}=\min\Big (\inf  \{ P_{i}(\theta): \theta\in [0,1]\},\min_{i+1\leq j\leq k} \inf  \{ P_{j}(\theta): \theta\in [\widetilde\theta_{i},1]\}, P_k(1)\Big ).
$$
Moreover, by adapting the discussion of the proof of Corollary~\ref{cor-2}, we can get that  either $M^{\vec{\boldsymbol\gamma}^i}$ equals $P_{i}(\theta)$ with $\theta\in [0,1]$ and in this case $M^{\vec{\boldsymbol\gamma}^i}$ is necessarily attained by $\nu\in\mathcal B_1$ such that $\Pi_i\nu$ is the equilibrium state of $P_i$ at $\theta$, or there is a unique $\nu\in\mathcal B_1$ at which $M^{\vec{\boldsymbol\gamma}^i}$ is attained; in the later case, the unique maximizing measure $\nu$ is the same as the maximizing measure for $M^{\vec{\boldsymbol\gamma}}$  in Theorem~\ref{optim} in case  $i_0\ge i+1$.

If $P_i'(1)\ge 0$, then, using the notations of the statement of Theorem~\ref{dim PiK}, we have $j_0=i$, and we denote by $\nu_i$ the equilibrium state of $P_i$ at $\theta_i^i$. Suppose that $P_i'(\theta_i^i)=0$, which holds automatically if $\theta_i^i>0$. This means that $\nu_i(\widetilde\phi_i)=0$.   Let $\nu$ be the unique element of $\mathcal B_1$ such that $\Pi_i\nu=\nu_{i}$ and $H(\nu)+\nu(\widetilde\phi)=H(\nu_i)+\nu_i(\widetilde\phi_i)=H(\nu_i)$. We have $P_i(\theta_i^i)=\widetilde g_i(\nu)= g_i(\nu)\le g_{i+1}(\nu)\le \cdots\le g_k(\nu)$, which implies $M^{\vec{\boldsymbol\gamma}^i}=P_i(\theta_i^i)$, and due to the observation made in the previous paragraph and Remark~\ref{rem0}, the measure
$\mu_\nu=\mu_{\nu_{i_0}}$ is the unique Mandelbrot measure $\mu$ such that the dimension of $\Pi_i\mu$ equals $M^{\vec{\boldsymbol\gamma}^i}$, conditional on $K\neq\emptyset$; moreover, one has $\dim_H(\mu_\nu)=\dim_H(\Pi_i\mu_\nu)$. If $\theta_i^i=0$ and $P_i'(\theta_i^i)>0$, we still have $M^{\vec{\boldsymbol\gamma}^i}=P_i(\theta_i^i)=\widetilde g_i(\nu)= g_i(\nu)\le g_{i+1}(\nu)\le \cdots\le g_k(\nu)$, but for any Mandelbrot measure~$\mu$ supported on $K$,  if $\rho=\mathbb{E}(\mu)$ satisfies $H(\rho)+\rho(\widetilde\phi)> H(\nu_i)$, then $\dim_H(\Pi_i\mu)=M^{\vec{\boldsymbol\gamma}^i}$. There are infinitely many such measures. Indeed, $\mu_{\nu_i}$ itself satisfies this property since $P_i'(0)=\nu_i(\widetilde\phi_i)=0>0$. Then one gets other examples of measures by considering, as defined in Section~\ref{skewed}, random vectors $W$ written as the ``skewed '' product of $(\nu_i([b]))_{b\in\mathcal A_i}$ with random vectors $(V_b)_{b\in \mathcal A_i}$ obtained as slight perturbions of  the vector $(\widetilde V_b)_{b\in \mathcal A_i}$ used to construct the measure $\mu_{\nu_i}$.  If $P_i'(1)<0$, the discussion ends like in the proof of Theorem~\ref{optim}.

\medskip

For the upper bound for the Hausdorff dimension, we prove that
$$
\dim_H \Pi_i(K)\le \inf\{P_j(\theta): \theta\in [0,1] \text{ if }j=i, \text{ and $\theta\in[\widetilde \theta_j,1]$ if $i<j\le k$}\},
$$
which in view of the lower bound is enough to conclude. To show the previous inequality, we extend the definitions of $D_{i,\theta}$ and $\widetilde D_{i,\theta}$ (see \eqref{Dit} and \eqref{Dit'}) to $\theta\in[0,1]$ and for $j=i$ we redefine the vector  $\rho(x,n)$ of \eqref{rhoxn} by taking $\widetilde \rho_i=\rho_i(x,\ell_i(n))$. It is readily seen from the proof of Lemma~\ref{comblem} that the conclusions of this lemma is still valid with these new definitions of $\widetilde D_{i,\theta}$ and $\rho(x,n)$.


Now, arguing similarly as in the proof of Theorem~\ref{VPUB}, for each $j\in \{i,\ldots,k\}$, for each $U=(U_j,\cdots,U_k)$ in $\mathcal A_j^{\ell_j(n)}\times \prod_{j'=j+1}^k \mathcal A_{j'}^{\ell_{j'}(n)-\ell_{j'-1}(n)}$, $\Pi_{i,j}^{-1}(B_U)\cap \Pi_i(K)$ is covered by, say, a family $\mathcal B(U)$ of $n_U$ balls of radius $e^{-\frac{n}{\gamma_1}}\subset X_i$ which intersect $\Pi_i(K)$.

Suppose $j=i$ and fix $\theta\in [0,1]$. In this case $n_U= 1$ and we can bound this number by $(N_{U_i}^{(i)})^\theta$. Noting that $\mathbb{E}\big ((N^{(i)}_{U_i})^\theta\big )\le \mathbb{E}\big (N^{(i)}_{U_i}\big )^\theta$, we can use similar estimates as in the proof of Theorem~\ref{VPUB} to now estimate $\mathcal H^s_{e^{-\frac{N}{\gamma_1}}}(\Pi_i(K)))$, and this yields $\dim_H \Pi_i(K)\le P_i(\theta)$ (here we followed the same idea as that used in \cite{Fal} to deal with projections of planar statistically self-similar limit sets of fractal  percolation).

Next, suppose $j\in\{i+1,\ldots, k\}$ and fix $\theta\in [\widetilde \theta_j,1]$. Denote by $C^{(i,j)}_{{U_j}}$ the set of cylinders of generation $\ell_{j-1}(n)$ in $X_i$ which intersects $\Pi_i(K)$, and project  to $[{U_j}_{|\ell_{j-1}(n)}]$ in $X_j$ via $\Pi_{i,j}$. Also denote by $N^{(i,j)}_{U_j}$ the cardinality of this set. Each cylinder in $C^{(i,j)}_{{U_j}}$ intersects at most one of the elements of $B_U$. Thus $n_U\le N^{(i,j)}_{U_j}$, so that:
$$
n_U\le \sum_{b\in C^{(i,j)}_{{U_j}}} 1\le  \sum_{b\in C^{(i,j)}_{{U_j}}}N^{(i)}_b= N^{(j)}_{U_{j-1}}.
$$
Then, the same lines as in the proof of Theorem~\ref{VPUB} yield $\dim_H \Pi_i(K)\le P_i(\theta)$ for all $\theta\in [\widetilde\theta_j,1]$.
\end{proof}

\begin{proof}[Sketch of the proof of  Theorem~\ref{dimB PiK}] This is similar to the proof of Theorem~\ref{thmdimB}, except that one must evaluate the cardinality of those $B\in \mathcal F_n^i$ such that $B\cap \Pi_i(K)\neq\emptyset$, and this time we exploit results known for the box dimension of projections of statistically self-similar fractal Euclidean percolation sets from dimension 2 to dimension 1.

We have to estimate the cardinality of those $U\in \mathcal A_{i}^{\ell_i(n)}\times\prod_{j=i+1}^k\mathcal A_j^{\ell_j(n)-\ell_{j-1}(n)}$ such that  $E^i(U)$ holds, with
\begin{align*}
E^i(U)&=\Big \{\exists\, (u_j)_{i\le j\le k}\in \mathcal A_{1}^{\ell_i(n)}\times\prod_{j=i+1}^k\mathcal A_1^{\ell_j(n)-\ell_{j-1}(n)}: \\ &\quad\quad\quad\quad\quad\quad \text{both }F_k^{i,(U_i,\ldots,U_k)}(u_i,\ldots, u_k) \text{ and }K^{u_iu_2\cdots u_k}\neq\emptyset\text{ hold}\},
\end{align*}
and where
$$
F_j^{i,(U_i,\ldots,U_j)}(u_i,\ldots, u_j)=\left\{\begin{array}{cl}
 &\Pi_{j'}(u_{j'})=U_{j'}, \ \forall\, i\le j'\le j,\\
 &[u_{i}]\cap K_{\ell_i(n)}\neq\emptyset\\
 &[u_{j'}]\cap K_{\ell_{j'}(n)-\ell_{j'-1}(n)}^{u_i\cdots u_{j'-1}}\neq\emptyset , \ \forall\,  i+1\le j'\le j\\
 \end{array}\right\}.
$$
One deduces easily from  \cite{DekGri} (see alternatively \cite{Fal} or \cite{BF2016}), which deal with the case $k=2$, that
$$
\lim_{n\to\infty}\frac{\log\#\{U_i\in  \mathcal A_{i}^{\ell_i(n)}:\, \exists u_i\in \mathcal A_{1}^{\ell_i(n)}, \  F_i^{i,U_i}(u_i)\text{ holds}\}}{n}=\frac{\widetilde \gamma_i}{\gamma_1}\psi_i(\widehat \theta_i).
$$
Then, a recursion similar to that used in the proof of Theorem~\ref{thmdimB} yields the desired result
$$
\lim_{n\to\infty}\frac{\log\#\{U\in \mathcal A_{i}^{\ell_i(n)}\times\prod_{j=i+1}^k\mathcal A_j^{\ell_j(n)-\ell_{j-1}(n)}:\, E^i(U)\text{ holds}\}}{n}=\frac{\widetilde \gamma_i}{\gamma_1}\psi_i(\widehat \theta_i)+\sum_{j=i+1}^k \frac{\gamma_j}{\gamma_1}\psi_j(\widehat\theta_j),
$$
i.e. $\dim_B\Pi_i(K)= \widetilde \gamma_i\psi_i(\widehat \theta_i)+\sum_{j=i+1}^k \gamma_j\psi_j(\widehat\theta_j)$ after normalizing by $\gamma_1^{-1}$.
\end{proof}

\begin{proof}[Proof of Corollary~\ref{cor2}] If $\widehat \theta_i=1$, using Proposition~\ref{prop4.1} we see that the equality between $\dim_H \Pi_i(K)$ and $\dim_B \Pi_i(K)$ imposes that $\dim_H \Pi_i(K)$ is attained by the branching measure, and the situation boils down to that of  Corollary~\ref{cor1}. This gives point (1) of the statement.

Suppose now that $\dim_H \Pi_i(K)=\dim_B \Pi_i(K)$,  $\widehat \theta_i<1$ and $\psi_i'(\widehat\theta_i)=0$ (which is automatically true if $0<\widehat \theta_i<1$).  The equality between $\dim_H \Pi_i(K)$ and $\dim_B \Pi_i(K)$ imposes that if~$\mu$ stands for the unique Mandelbrot measure supported on $K$ such that $\dim_e^{\vec{\boldsymbol\gamma}_i}(\mu)=\dim_H \Pi_i(K)$, then $\dim_e(\mu)=h_{\nu_i}(T_i)= \psi_i(\widehat \theta_i)$, where $\nu_i=\mathbb{E}( {\Pi_i}\mu)=\nu_{i,\widehat\theta_i}$. Also,  for  $j\in I_i(=\{i,\ldots,k\})$ such that $j\le j'_0$, we must have ${\Pi_{i,j}}{\nu_i}= \nu_{j,\widehat\theta_j}$. Using that  for all $b\in\widetilde A_j$, we have $\sum_{b'\in\Pi_{i,j}^{-1}(b)} \nu_{i,\widehat\theta_i}([b'])=\nu_j([b])$ and the fact that $\psi_i'(\widehat\theta_i)=\psi_j'(\widehat\theta_j)$ we can write
\begin{align*}
0=\psi_i'(\widehat\theta_i)-\psi_j'(\widehat\theta_j)&=\sum_{b'\in \widetilde {\mathcal A}_i}\nu_{i,\widehat\theta_i}([b'])\log \mathbb E(N_{b'}^{(i)})-\sum_{b\in \widetilde {\mathcal A}_j}\nu_{j,\widehat\theta_j}([b])\log \mathbb E(N_{b}^{(j)})\\
&=\sum_{b\in \widetilde {\mathcal A}_j}\sum_{b'\in\Pi_{i,j}^{-1}(b)}\nu_{i,\widehat\theta_i}([b'])\log \frac{\mathbb E(N_{b'}^{(i)})}{\mathbb{E}(N_b^{(j)})}.
\end{align*}
This implies that for all $b\in \widetilde {\mathcal A}_j$, the set $\Pi_{i,j}^{-1}(b)\cap \widetilde{\mathcal A}_i$ is a singleton $\{b'\}$ such that $\mathbb E(N_{b'}^{(i)})= \mathbb E(N_{b}^{(j)})$, hence $\psi_i=\psi_j$ and $\widehat \theta_j=\widehat\theta_i$. Let us now examine those $j>j'_0$ in $I_i$. The previous argument shows that  $\widehat \theta_j=0$ and $\psi_j'(0)>0$ (for otherwise $j'_0$ would be at least equal to the smallest of those $j$), and ${\Pi_{i,j}}{\nu_i}=\nu_{j,0}$,  so that ${\Pi_{i,j}}{\nu_i}$ is uniformly distributed, i.e.   $\sum_{b'\in \Pi_{i,j}^{-1}(b)}\mathbb{E} (N_{b'}^{(i)})^{\widehat \theta_i}$ does not depend on $b\in\widetilde{\mathcal A}_j$.  We conclude that  the conditions of point (2) are necessary. Conversely, if these conditions hold, one easily checks using Theorems~\ref{dim PiK} and~\ref{dimB PiK} that $\dim_H \Pi_i(K)=\dim_B \Pi_i(K)$.

The last case easily follows from the previous discussion.
\end{proof}

\section{Dimension of conditional measures. Proof of Theorem~\ref{thmfibers}} \label{pfcondmeas}

Since the proof of Theorem~\ref{thmfibers2} is very similar to that of Theorem~\ref{thmfibers}, we leave it to the reader.

 Point (3) of the statement simply follows from points (1) and (2) as well as the dimensions formula provided by Theorems~\ref{thmdim} and~\ref{thmdim2} for $\dim (\mu)$ and $\dim ({\Pi_i}\mu)$.

To get point (1) we notice that for any $x\in X_1$ and $n\ge 1$ we have $[x_{|\ell_k(n)}]\subset B(x, e^{-n/\gamma_1})\subset [x_{|n}]$, so  since   $\dim_e(\mu)\le h_{\nu_i}(T_i)$ we find that Theorem~\ref{dimentr}  implies that $\mu^z$ is exact dimensional with Hausdorff dimension equal to 0.

Now we prove point (2). The following lines  do not depend on ${\Pi_i}\mu$ being absolutely continuous with respect to $\nu_i$ or not.


When $\mu_\omega=\mu\neq 0$, for ${\Pi_i}\mu_\omega$-almost every $z$, the conditional measure $\mu^z_\omega $ is supported on $K^z=\pi^{-1} (\{z\})\cap K$, obtained as the weak-star limit, as $n\to\infty$, of the measures $\mu_{\omega,n}^z$  obtained on $K$  by assigning uniformly the mass $\frac{\mu_\omega([J]\cap \Pi_i^{-1}([z_{|n}]))}{{\Pi_i}\mu_\omega([z_{|n}])}$ to each cylinder $[J]$ of generation~$n$ in $X_1$.  {To be more specific, for any cylinder $[J]$, almost surely, the measurable set
$$
A_{J}=\left\{(\omega,z)\in\Omega\times X_i: \lim_{n\to\infty}\frac{\mu_\omega([J]\cap \Pi_i^{-1}([z_{|n}]))}{{\Pi_i}\mu_\omega([z_{|n}])} \text{ exists}\right \}$$
is of full $\widehat {\mathbb Q}$-probability, where  we define $\widehat {\mathbb Q }({\rm d}\omega,{\rm d}z)=\mathbb{P}({\rm d}\omega) {\Pi_i}\mu_\omega(\mathrm{d}z)$, and for all $(\omega,z)$ in a subset $A'_{J}$ of $A_{J}$ of full $\widehat {\mathbb Q}$-probability, we have $ \mu^z_\omega ([J])=\lim_{n\to\infty}\frac{\mu_\omega([J]\cap \Pi_i^{-1}([z_{|n}]))}{{\Pi_i}\mu_\omega([z_{|n}])}$. }

Suppose now that conditional on $\mu\neq 0$, ${\Pi_i}\mu_\omega$ is  absolutely continuous with respect to $\nu_i$.  There exists a measurable set $A'$ of full $\widehat {\mathbb Q}$-probability such that for all $(\omega,z)\in A'$, the limit
$$
\lim_{n\to\infty} \left(f_{\omega,n}(z):=\frac{{\Pi_i}\mu_\omega ([z_{|n}])}{\nu_i([z_{|n}])}\right)
$$
exists and is positive. We denote it by $f_{\omega}(z)$.

Set $A=A'\cap \bigcap_{J\in \Sigma}A'_{J}$. For all $(\omega,z)\in A$,   the sequence of measures $\widetilde\mu_{\omega,n}^z=  f_{\omega,n}(z) \mu_{\omega,n}^z $ weakly converges to the measure $\widetilde\mu_\omega^z$ defined as $ f_\omega(z) \mu_\omega^z$.

 Let
\begin{align*}
\Omega_A&=\{\omega:\  (\omega,z)\in A\text{ for some $z\in X_i$}\},\\
F^\omega&=\{z\in X_i: \ (\omega,x)\in A\}, \quad  \forall\,\omega\in\Omega_A.
\end{align*}
If $(\omega,x)\not\in  A$, set $\mu_\omega^x=\widetilde\mu_\omega^x=0$.

For $z\in  F^\omega$, $n\ge 1$ and $J\in\mathcal A_1^n$, we have
\begin{eqnarray*}
\widetilde \mu^z([J])=\lim_{p\to\infty} \widetilde \mu_p^z([J])= \lim_{p\to\infty} \frac{\mu([J]\cap \Pi_i^{-1}([z_{|n+p}])}{\nu_i ([z_{|n+p}])}.
\end{eqnarray*}

If $U=(U_1,\cdots, U_k)\in\prod_{i=1}^k\mathcal A_i^{\ell_i(n)-\ell_{i-1}(n)}$, the ball $B_U$ intersects $\Pi_i^{-1}(\{z\})$ if and only if  $\Pi_{j,i}(U_j)=T^{\ell_{j-1}(n)}_i(z)_{|\ell_j(n)-\ell_{j-1}(n)}$ for all $1\le j\le i-1$ and $U_j= \Pi_{i,j}(T^{\ell_{j-1}(n)}_i(z)_{|\ell_j(n)-\ell_{j-1}(n)})$ for $i\le j\le k$. Recalling \eqref{decomp}, we also have
\begin{align*}
\widetilde \mu^z(B_U)&=\sum_{(J_1,\ldots,J_k)\in\mathcal J_U} \widetilde \mu^z([J_1\cdots J_k]])\\
&=\sum_{(J_1,\ldots,J_k)\in\mathcal J_U}  \lim_{p\to\infty} \frac{\mu([J_1\cdots J_k]\cap \Pi_i^{-1}([z_{|n+p}]))}{\nu_i ([z_{|n+p}])}.
\end{align*}
Fix $q\ge 0$. For all $n\ge 1$, we are going to estimate the expectation of the partition function $\sum_{B\in\mathcal F_n}\widetilde \mu^z(B_U)^q$ with respect to the measure $\mathbb{P}\otimes \nu_i$.

Let $j_0=\min\{2\le j\le i-1:\dim_e(\mu)> h_{\nu_j}(T_j)\}$, with $\min(\emptyset)=i$, and  $$D=\widetilde\gamma_{j_0-1}(\dim_e(\mu)-h_{\nu_i}(T_i))+\sum_{j=j_0}^{i-1}\gamma_j(h_{\nu_j}(T_j)-h_{\nu_i}(T_i)),$$ which is precisely the value given by \eqref{dimfib} due to our choice of $j_0$. We will show that there exists $c>0$ such that for all $q$ in a neighbourhood of $1$, there exists $C_q>0$ such that we have
\begin{align*}\label{controlfiber}
\mathbb{E}_{\mathbb{P}\otimes \nu_i}\Big (\sum_{B_U\in\mathcal F_n}\widetilde \mu^z(B_U)^q\Big )\le  C_q\exp\Big (-\frac{n}{\gamma_1}(q-1) D+ O((q-1)^2))\Big ).
\end{align*}
This is enough to conclude that with probability 1, conditional on $\mu\neq 0$, for ${\Pi_i}{\mu}$-almost every $z$ (remember that ${\Pi_i}{\mu}$ is absolutely continuous with respect to $\nu_i$), one has $\tau_{\widetilde \mu^z}(q))\ge (q-1) D- c(q-1)^2$  in some neighbourhood of $1$. But since $\mu^z$ is a multiple of $\widetilde \mu^z$, the same holds for $\mu^z$. This implies that the concave functions $\tau_{\mu^z}$ and $q\mapsto (q-1) D- c(q-1)^2$ share the same derivative at $1$, namely $D$. Consequently,  $\mu^z$ is exact dimensional with dimension $D$.

Now we prove \eqref{controlfiber}. Recall that outside the set $A$, the measure $\widetilde \mu_\omega^z$ has been defined equal to 0. By Fatou's lemma, we have
\begin{align*}
\mathbb{E}&_{\mathbb{P}\otimes \nu_i} \sum_{B_U\in\mathcal F_n}\widetilde \mu^z(B_U)^q\\
&\le \liminf_{p\to\infty}\mathbb{E}_{\mathbb{P}\otimes \nu_i}\sum_{B_U\in\mathcal F_n}\Big (\sum_{(J_1,\ldots,J_k)\in\mathcal J_U} \frac{\mu([J_1\cdots J_k]\cap \Pi_i^{-1}([z_{|\ell_k(n)+p}]))}{\nu_i ([z_{|\ell_k(n)+p}])}\Big )^q\\
&= \liminf_{p\to\infty}\mathbb{E} \sum_{L\in\widetilde {\mathcal A_i}^{\ell_k(n)+p}} \sum_{B_U\in\mathcal F_n}\Big (\sum_{(J_1,\ldots,J_k)\in\mathcal J_U} \frac{\mu([J_1\cdots J_k]\cap \Pi_i^{-1}([L]))}{\nu_i ([L])}\Big )^q \nu_i([L])\\
&= \liminf_{p\to\infty}\mathbb{E} \sum_{L\in\widetilde {\mathcal A_i}^{\ell_k(n)+p}} \sum_{B_U\in\mathcal F_n}\Big (\sum_{(J_1,\ldots,J_k)\in\mathcal J_U} \sum_{J'_p\in \Pi_i^{-1}(T_i^{\ell_k(n)}L)}\frac{\mu([J_1\cdots J_kJ'_p])}{\nu_i ([L])}\Big )^q \nu_i([L]).
\end{align*}

Denote by $S$ the expectation in the right hand side of the previous inequality. Due to the remark made above about the condition for $B_U$ to intersects $\Pi_i^{-1}([u])$, and the multiplicativity property of the measure $\nu_i$, we can rewrite $S$ as follows:
$$
S=\mathbb{E} \sum_{L=(L_1,\ldots,L_{i-1})}\sum_{(U_1,\ldots,U_{i-1})_L}\sum_{L'} m(U_1,\ldots,U_{i-1},L,L')^q \nu_i([L_1\cdots L_{i-1}L']),
$$
where $L\in\prod_{j=1}^{i-1} \widetilde {\mathcal A}_i^{\ell_{i-1}(n)}$, $(U_1,\ldots,U_{i-1})_L\in \prod_{j=1}^{i-1}\widetilde {\mathcal A}_j^{\ell_j(n)-\ell_{j-1}(n)}$ is such that $\Pi_{j,i}(U_j)=L_j$ for each $1\le j\le i-1$, $L'\in \widetilde {\mathcal A}_i^{p+\ell_k(n)-\ell_{i-1}(n)}$, and taking the conventions that the words involved below whose writing uses the symbol $J$ belong to $\widetilde {\mathcal A}_1^*$,
\begin{align*}
m(U_1,\ldots,U_{i-1},L,L')
=\sum_{\substack{(J_1,\ldots ,J_{i-1}):\, \Pi_{j,i}(J_j)=U_j,\\
J=J_i\cdots J_k:\, \Pi_{i}(J_j)=L_j}}\sum_{J':\, \Pi_i(J')=L'} \frac{\mu([J_1\cdots J_{i-1}J'])}{\nu_i([L_1\cdots L_{i-1}L'])}.
\end{align*}

Suppose that $q\ge 1$. Using the same idea as in the proof of Theorem~\ref{pro1}, but rewriting $S$ as an expectation with respect to $\mathbb{P}\otimes \nu_i$ instead of $\mathbb P$, yields
\begin{align*}
\mathbb E(S)&\le 2^{3q(i-1)}\left(\prod_{j=1}^{i-1} S_{j,n}\right) \cdot R_{n,p},
\end{align*}
where
$$
S_{j,n}= \mathbb{E}\Big (\sum_{L_j}\nu_i([L_j]) \sum_{U_j:\, \Pi_{j,i}(U_j)=L_j} \Big (\sum_{J_j:\, \Pi_j(J_j)=U_j}\frac{\mu_{\ell_{j}(n)-\ell_{j-1}(n)}([J_j])}{\nu_i(L_j)}\Big )^q\Big )
$$
and
$$
R_{n,p}=\mathbb{E} \Big (\sum_{L'}\Big (\sum_{J':\, \Pi_i(J')=L'} \frac{\mu([J'])}{\nu_i([L'])}\Big )^q \nu_i([L'])\Big ).
$$
Note that $R_{n,p}=\mathbb{E}_{\mathbb P\otimes \nu_i}(X_{p+\ell_k(n)-\ell_{i-1}(n)}^q)$, where
$$
X_{n}(\omega,z)=\sum_{J\in \widetilde{\mathcal A}_1^{n}:\, \Pi_i(J)=z_{|n}} \frac{\mu_\omega([J])}{\nu_i([z_{|n}])}
$$
is a perturbation of the martingale in random environment
$$
\widetilde X_{n}(\omega,z)=\sum_{J\in \widetilde{\mathcal A}_i^n:\, \Pi_i(J)=z_{|n}} \frac{\mu_{\omega,n}([J])}{\nu_i([z_{|n}])}.
$$
Now, recalling the definition of the vectors $V^{(i)}_b$ in Section~\ref{skewed} and setting
$$
\varphi(q)= \log \sum_{b\in\widetilde {\mathcal A}_1}\nu_i([b])e^{-T_{V^{(i)}_b}(q)},
$$
our assumption that $\dim_e(\mu)>h_{\nu_i}(T_i)$ is equivalent to saying that at point 1 the function  $\varphi$ has a negative derivative, since  $\varphi'(1)=h_{\nu_i}(T_i)-T'(1)=h_{\nu_i}(T_i)-\dim_e(\mu)$. We can then apply \cite[Proposition 5.1]{BF2016} to $X_{p+\ell_k(n)-\ell_{i-1}(n)}$ and for $q$ close enough to $1+$, get a constant $C_q>0$ such that $R_{p,n}\le C_q$ independently of $n$ and $p$.

Next we estimate the terms $S_{j,n}$, for $1\le j\le i-1$.  For $j=1$, we simply have
\begin{equation*}
S_{1,n}=\mathbb{E}\Big (\sum_{L_1}\nu_i([L_1]) \sum_{U_1:\, \Pi_{i}(U_1)=L_1} \Big (\frac{\mu_{n}([U_1])}{\nu_i([L_1])}\Big )^q\Big )=e^{n \varphi(q)}=e^{n (q-1) (h_{\nu_i}(T_i)-\dim_e(\mu))+O((q-1)^2))}.
\end{equation*}
For $2\le j\le i-1$, we rewrite $S_{j,n}$  as (recall that $\nu_j$ stands for the expectation of ${\Pi_j}\mu$)
\begin{equation*}
S_{j,n}=\mathbb{E}\sum_{U_j} \phi_j(U_j)  \Big (\sum_{J_j:\, \Pi_j(J_j)=U_j}\frac{\mu_{\ell_{j}(n)-\ell_{j-1}(n)}([J_j])}{\nu_j([U_j])}\Big )^q,
\end{equation*}
where
$$
\phi_j(U_j)=\nu_j([U_j])^q\nu_i(\Pi_{j,i}([U_j]))^{1-q}.
$$
Let $\nu_{q,j}$ be the Bernoulli product measure on $\widetilde X_j$ associated with the probability vector $\nu_{q,j}([b])=\displaystyle\frac{\phi_j([b])}{\sum_{b'\in \widetilde{\mathcal A}_j}\phi_j([b'])}$, and define
$$
\varphi_j(q)=\log \sum_{b\in \widetilde{\mathcal A}_j} \phi_j([b]) \quad{and}\quad \widetilde X^{(j)}_{n}(\omega,z)=\sum_{J\in\widetilde{\mathcal A}_1^{n}:\, \Pi_j(J)=z_{|n}}\frac{\mu_{n}([J])}{\nu_j([z_{|n}])}, \ z\in \widetilde X_j.
$$
With these definitions, $S_{j,n}$ rewrites
$$
S_{j,n}=e^{(\ell_j(n)-\ell_{j-1}(n))\varphi_j(q)}\mathbb{E}_{\mathbb P\otimes\nu_{q,j}}((\widetilde X^{(j)}_{n})^q),
$$
Again, we can use \cite[Proposition 5.1]{BF2016}, and get a constant $C_{q,j}>0$ such that
$$
S_{j,n}\le C_{q,j}e^{(\ell_j(n)-\ell_{j-1}(n))\varphi_j(q)}\max \Big (1, \sum_{b\in \widetilde{\mathcal A}_j} \nu_{q,j}([b])e^{-T_{V^{(j)}_b}(q)}\Big )^{\ell_j(n)-\ell_{j-1}(n)}.
$$
A computations shows that the derivative at $1$ of the function $q\mapsto \sum_{b\in \widetilde{\mathcal A}_j} \nu_{q,j}([b])e^{-T_{V^{(j)}_b}(q)}$ is equal to $h_{\nu_j}(T_j)-T'(1)=h_{\nu_j}(T_j)-\dim_e(\mu)$.

Recall that $j_0=\min\{2\le j\le i-1:\dim_e(\mu)> h_{\nu_j}(T_j)\}$, with $\min(\emptyset)=i$. If~$j_0\le j\le i-1$ and $q$ is close enough to 1, we thus have $ \sum_{b\in \widetilde{\mathcal A}_j} \nu_{q,j}([b])e^{-T_{V^{(j)}_b}(q)}<1$, hence $S_{j,n}\le  C_{q,j}e^{(\ell_j(n)-\ell_{j-1}(n))\varphi_j(q)}$. If $2\le j< j_0$, using a Taylor expansion of order 2 we get  $\sum_{b\in \widetilde{\mathcal A}_j} \nu_{q,j}([b])e^{-T_{V^{(j)}_b}(q)}\le \exp ((q-1)(h_{\nu_j}(T_j)-\dim_e(\mu))+O((q-1)^2))$. Moreover, for any~$j$, $e^{\varphi_j(q)}=\exp ((q-1) (h_{\nu_i}(T_i)-h_{\nu_j}(T_j))+O((q-1)^2))$. So for  $2\le j< j_0$, $S_{j,n}\le  C_{q,j}\exp((\ell_j(n)-\ell_{j-1}(n))((q-1)(h_{\nu_i}(T_i)-\dim_e(\mu))+O((q-1)^2))$. Finally, for $q$ close enough to $1+$, there exists $C_q>0$ such that
\begin{align}
\nonumber\mathbb{E}_{\mathbb{P}\otimes \nu_i}\sum_{B_U\in\mathcal F_n}\widetilde \mu^z(B_U)^q&\le C_q\exp\Big  (q-1)(\ell_{j_0-1}(n) (h_{\nu_i}(T_i)-\dim_e(\mu))\\
\nonumber &\quad +(q-1)\sum_{j=j_0}^{i-1} (\ell_{j}(n)-\ell_{j-1}(n)) (h_{\nu_i}(T_i)-h_{\nu_j}(T_j))  +O((q-1)^2) n\Big )\\
\label{controlfiber}&=C_q\exp\Big (-\frac{n}{\gamma_1}(q-1) D+ O((q-1)^2))\Big ),
\end{align}
hence \eqref{controlfiber} holds.
\medskip

Suppose now that $q\in (0,1)$. Using the same idea as in the proof of Theorem~\ref{pro1} yields
\begin{align*}
\mathbb E(S)&\le \prod_{j=1}^{i-1} \widetilde S_{j,n},
\end{align*}
where
$$
\widetilde S_{j,n}= \begin{cases}\displaystyle\mathbb{E}\Big (\sum_{L_j}\nu_i([L_j]) \sum_{U_j:\, \Pi_{j,i}(U_j)=L_j} \Big (\sum_{J_j:\, \Pi_j(J_j)=U_j}\frac{\mu_{\ell_{j}(n)-\ell_{j-1}(n)}([J_j])^q}{\nu_i(L_j)^q}\Big )\Big )&\text{if }1\le j\le j_0-1,
\\
\displaystyle\mathbb{E}\Big (\sum_{L_j}\nu_i([L_j])^{1-q} \sum_{U_j:\, \Pi_{j,i}(U_j)=L_j} \nu_j(U_j)^q\Big )&\text{if }j_0\le j\le i-1.
\end{cases}
$$
With the notations introduced in the case $q\ge 1$, this rewrites
$$
\widetilde S_{j,n}= \begin{cases}\displaystyle\Big (\sum_{b\in\widetilde{\mathcal A}_i}\nu_i([b])e^{-T_{V^{(i)}_b(q)}}\Big )^{\ell_j(n)-\ell_{j-1}(n)} &\text{if }1\le j\le j_0-1,
\\
\displaystyle e^{\varphi_j(q)(\ell_j(n)-\ell_{j-1}(n))}&\text{if }j_0\le j\le i-1.
\end{cases}
$$
Using Taylor expansions we can get that  \eqref{controlfiber} holds for $q$ close to $1-$ as well.

\section{The case when $\{2\le i\le k:\gamma_i\neq 0\}=\emptyset$} \label{lastsection}

In our main statements about the Hausdorff and box-counting dimension of $K$ and its projections for simplicity we assumed all the $\gamma_i$, $2\le i\le k$ to be positive, which in the Euclidean realisation of Section~\ref{euclidean} corresponds to $m_1>\cdots >m_k\ge 2$. It turns out that up to slight modifications in the statement and proofs, our results  cover the general configuration $m_1\ge \cdots \ge m_k\ge 2$, for which the diagonal endomorphism ${\rm diag}(m_1,\ldots,m_k)$ may have eigenspaces of dimension at least 2 over which it is a similarity. In this case, in the expressions giving the dimensions of $K$ and its projections, when $m_{i}=m_{i-1}$, i.e. $\gamma_i=\frac{1}{\log(m_{i-1})}-\frac{1}{\log(m_{i})}=0$, the index $i$ has no contribution, and geometrically for any $1\le i<j\le k$, $x\in X_i$ and $n\ge 1$, for the induced metric by $d_{\vec{\boldsymbol\gamma}}$ on $X_i$, if $y\in B(x,e^{-n/\gamma_1})$, nothing is required on $T_i^{\ell_{j-1}(n)}(y)_{| \ell_{j}(n)-\ell_{j-1}(n)}$.

For all the statements of Section~\ref{dimsK} and Theorem~\ref{VPUB}, the only change to make to cover the case $\gamma_i\ge 0$ for all $2\le i\le k$ is to set $I=\{2\le i\le k: \gamma_i>0$ and replace $k$ by $\sup (I)$  in \eqref{formula1}. The proofs adapt readily.

For the statements of Section~\ref{Projections},  one has to replace $I_i$ by $\{i\}\cup\{i< j\le k: \gamma_j>0\}$ and replace $k$ by $\sup (I_i)$  in \eqref{formula2}. Again, the modifications in the proofs are left to the reader.

\appendix

\section{Main notation and conventions}
\label{B}
For the reader's convenience, we summarize in Table \ref{table-1} the main notation and typographical conventions used in this paper.
\begin{table}
\centering
\caption{Main notation and conventions}
\vspace{0.05 in}
\begin{footnotesize}
\begin{raggedright}
\begin{tabular}{p{1.2 in} p{4.3 in} }
\hline \rule{0pt}{3ex}
$\N$, $\N^*$, $\R_+$, $\R_+^*$ & sets of non-negative (resp. positive) integers,  non-negative (resp. positive) real numbers\\
$(X_i,T_i)$ &  one-sided full shift space over a finite alphabet $\A_i$, $i=1,\ldots, k$\\
$\M(X_i,T_i)$ & set of $T_i$-invariant  Borel probability measures on $X_i$, $i=1,\ldots, k$\\
$\pi_0:\; X_1\to X_1$ &  identity map on $X_1$\\
$\pi_i:\; X_i\to X_{i+1}$ &  one-block factor map, $i=1,\ldots, k-1$ \\
$\Pi_i:\; X_1\to X_i$ & $\Pi_i=\pi_{i-1}\circ\cdots \circ \pi_0$\\
$\Pi_{i,j}:\; X_i\to X_j$ & $\Pi_{i,j}=\pi_{j-1}\circ\cdots \circ \pi_i$ if $j>i$, and the identity map if $j=i$\\
$\vec{\boldsymbol\gamma}=(\gamma_1,\ldots, \gamma_k)$ & a fixed vector in $\R_+^*\times (\R_+)^{k-1}$\\
$\vec{\boldsymbol\gamma}^i=(\vec{\boldsymbol\gamma}^i_j)_{i\leq j\leq k}$ &  the vector  $(\gamma_1+\ldots+\gamma_i, \gamma_{i+1}, \ldots, \gamma_k)$\\
$d_{\vec{\boldsymbol\gamma}}$ & an ultrametric distance defined on $X_1$ (cf. \eqref{e-da}) \\
$(X_1, d_{\vec{\boldsymbol\gamma}})$  & self-affine symbolic space\\
$K=K(\omega)$ & symbolic random statistically self-affine sponge defined from a random subset $A$ of $\A_1$, where $A$ is alternatively written as $(c_a)_{a\in \A_1}$ (Section \ref{S-2.1})\\
$W=(W_a)_{a\in \A_1}$ & a non-negative random vector such that ${W_a>0}\subset \{c_a=1\}$ and $\E (\sum_{a\in \A_1}W_a)=1$ (Section \ref{Mandmeas}) \\
$(A(u), W(u))_{u\in \A_1^*}$ & independent copies of $(A, W)$ (Section \ref{Mandmeas}) \\
$\mu$ & Mandelbrot measure associated with $W$ (Section \ref{Mandmeas})\\
$K_\mu$ & topological support of a Mandelbrot measure $\mu$\\
$\dim_e(\nu)$ & entropy dimension of a finite Borel measure on $X_i$ (cf. \ref{e-entropy})\\
$h_\nu(T_i)$ & measure-theoretic entropy of  $T_i$ with respect to a $T_i$-invariant measure $\nu$\\
$\dim_e^{\vec{\boldsymbol\gamma}}(\mu)$ & $\sum_{i=1}^k \gamma_i \dim_e(\Pi_i \mu)=\gamma_1 \dim_e(\mu) +\sum_{i=2}^k\gamma_i\min(\dim_e(\mu), h_{\nu_i}(T_i))$
(cf. \eqref{e-LY})\\
$h_\nu^{\vec{\boldsymbol\beta}}(T_i)$ & $\sum_{j=i}^k \beta_j h_{\Pi_{i,j} \nu}(T_j)$, where  $\vec{\boldsymbol\beta}=(\beta_i,\ldots, \beta_k)$ and $\nu\in \M(X_i, T_i)$\\
$N_b^{(i)}$  & $\# \{a\in\mathcal A_1:\; [a]\subset  \Pi_i^{-1}([b]),\; [a]\cap K\neq\emptyset\}$, $b\in \A_i$\\
$\widetilde{\A}_i$, $\widetilde{X_i}$ & $\widetilde{\A}_i=\{b\in \A_i:\; \E(N_b^{(i)})>0\}$, and $\widetilde{X_i}$ is the one-sided full shift space over $\widetilde{\A}_i$\\
$P^{\vec{\boldsymbol\beta}}(\phi, T_i)$ & weighted topological pressure of $\phi\in C(\widetilde{X_i})$ (cf. \eqref{e-weighted-pressure})\\
$P_i(\theta)$ & $P^{\vec{\boldsymbol\gamma}^i}(\theta\phi_i, T_i)$, where $\phi_i(x)=(\gamma_1+\cdots+\gamma_i) \log \E (N_{x_1}^{(i)})$, $i=2,\ldots, k$\\
$i_0$, $\theta_{i_0}$ & (cf. \eqref{e-i_0} and \eqref{e-theta_0}) \\
$\psi_i$ & $\log \sum_{b\in {\widetilde \A}_i} \E (N_b^{(i)})^\theta$, $i=2, \ldots, k$\\
$\widehat{\theta}_i$  & (cf. the definition right after \eqref{e-psi_i})\\

\hline
\end{tabular}
\label{table-1}
\end{raggedright}
\end{footnotesize}
\end{table}

\noindent{\bf Acknowledgements}. The research of both authors was supported in part by University of Paris 13, the
HKRGC GRF grants (projects  CUHK401013, CUHK14302415),  and the
France/Hong Kong joint research scheme PROCORE (33160RE,   F-CUHK402/14).

\end{document}